\documentclass[10pt,a4paper]{article}
\usepackage{amsmath}
\usepackage{amsfonts}
\usepackage{amssymb}
\usepackage{amsthm}
\usepackage[utf8]{inputenc}
\usepackage{graphicx}
\usepackage{subcaption}
\usepackage[font=small,labelfont=bf]{caption}
\usepackage{listings}
\usepackage{fancyhdr}
\usepackage{latexsym,cancel}
\usepackage{makeidx}
\usepackage{cite}
\usepackage{multicol}
\theoremstyle{definition}

\newtheorem{teo}{Theorem}

\usepackage{stackrel}
\usepackage{enumitem}
\usepackage{authblk}
\begin{document} 

\title{A fractional Traub method with $(2\alpha+1)$th-order of convergence and its stability}
\author[2]{Giro Candelario}
\author[1]{Alicia Cordero}
\author[1]{Juan R. Torregrosa }
\affil[1]{{\small \textit{Instituto de Matemáticas Multidisciplinar, Universitat Polit$\grave{\text{e}}$cnica de Val$\grave{\text{e}}$ncia, Val$\grave{\text{e}}$ncia, Spain}}}
\affil[2]{{\small \textit{Área de Ciencias Básicas y Ambientales, Instituto Tecnológico de Santo Domingo, Santo Domingo, Dominican Republic}}}
\date{}
\maketitle

\begin{abstract}
Some fractional Newton methods have been proposed in order to find roots of nonlinear equations using fractional derivatives. In this paper we introduce a fractional Newton method with order $\alpha+1$ and compare with another fractional Newton method with order $2\alpha$. We also introduce a fractional Traub method with order $2\alpha+1$ and compare with its first step (fractional Newton method with order $\alpha+1$). Some tests and analysis of the dependence on the initial estimations are made for each case. \\

\end{abstract}

\section{Introduction}

In this section we introduce some conepts related with fractional calculus, and a fractional Newton method recently proposed with Caputo and Riemann-Liouville derivatives. \\
Caputo fractional derivative of a function $f(x)$ of order $\alpha>0$, $a$, $\alpha$, $x \in \mathbb{R}$ is defined as \\
\begin{equation}\label{e1}
cD_a^{\alpha}f(x)=\left\{
\begin{aligned}
\dfrac{1}{\Gamma(m-\alpha)}\int_a^x\dfrac{df^{(m)}(t)}{dt^{(m)}}\dfrac{dt}{(x-t)^{\alpha+1-m}},\hspace{10pt}m-1<\alpha\leq m\in\mathbb{N}, \\
\dfrac{df^{(m)}(t)}{dt^{(m)}},\hspace{100pt}\alpha=m\in\mathbb{N}.
\end{aligned}
\right.
\end{equation}
The Caputo derivative holds the propery of nonfractional derivative, $cD_a^{\alpha}C=0$, being $C$ a constant, as we can see in \cite{IP}. We will be using $m=1$ in this paper. \\
The following theorem provides a Taylor power serie of $f(x)$ with Caputo Derivative.
\begin{teo}[Theorem 3, \cite{ZMO}]
Let us suppose that $cD_a^{j\alpha}f(x)\in C([a, b])$ for $j=1,2,\dots,n+1$ where $\alpha\in(0,1]$, then we have
\begin{equation}\label{e2}
f(x)=\sum_{i=0}^ncD_a^{i\alpha}f(a)\dfrac{(x-a)^{i\alpha}}{\Gamma(i\alpha+1)}+cD^{(n+1)\alpha}f(\xi)\dfrac{(x-a)^{(n+1)\alpha}}{\Gamma((n+1)\alpha+1)},
\end{equation}
with $a\leq\xi\leq x$, for all $x\in(a,b]$ where $cD_a^{n\alpha}=cD_a^\alpha\cdot cD_a^\alpha\cdots cD_a^\alpha$ (n times).
\end{teo}
A fractional Newton method with Caputo derivative has been proposed in \cite{AJR}, as shown in the following iterative expression:
\begin{equation}\label{e3}
x_{k+1}=x_k-\Gamma(\alpha+1)\dfrac{f(x_k)}{cD_{\bar{x}}^{\alpha}f(x_k)},\hspace{10pt}k=0,1,2,\dots
\end{equation}
with $\Gamma(\alpha+1)$ as a damping parameter. Let us denote this method CFN$_1$. \\
Riemann-Liouville fractional derivative of first kind of $f(x)$ with order $\alpha$, $0<\alpha\leq1$, is defined as \\
\begin{equation}\label{e4}
D_{a^+}^{\alpha}f(x)=\left\{
\begin{aligned}
\dfrac{1}{\Gamma(1-\alpha)}\dfrac{d}{dx}\int_a^x\dfrac{f(t)}{(x-t)^{\alpha}}dt,\hspace{10pt}0<\alpha<1, \\
\dfrac{df(t)}{dt},\hspace{100pt}\alpha=1.
\end{aligned}
\right.
\end{equation}
The Riemann-Liouville derivative does not hold the propery of nonfractional derivative, $D_{a^+}^{\alpha}C\neq0$, being $C$ a constant. \\
The following theorem provides a Taylor power serie of $f(x)$ with Riemann-Liouville Derivative (see \cite{JJT}).
\begin{teo}[Proposition 3.1, \cite{GJ}]
Let us assume the continuous funtion $f:\mathbb{R}\longrightarrow\mathbb{R}$ has fractional derivatives of order $k\alpha$, for any positive integer $k$ and any $\alpha$, $0<\alpha\leq1$, then the following equality holds,
\begin{equation}\label{e5}
f(x+h)=\sum_{k=0}^{+\infty}\dfrac{h^{\alpha k}}{\Gamma(\alpha k+1)}D_{a^+}^{\alpha k}f(x),
\end{equation}
where $D_{a^+}^{\alpha k}f(x)$ is the Riemann-Liouville derivative of order $\alpha k$ of $f(x)$.
\end{teo}
Another fractional Newton method was proposed in \cite{AJR} with Riemann-Liouville derivative, as shown in the following iterative expression:
\begin{equation}\label{e6}
x_{k+1}=x_k-\Gamma(\alpha+1)\dfrac{f(x_k)}{D_{a^+}^{\alpha}f(x_k)},\hspace{10pt}k=0,1,2,\dots
\end{equation}
with $\Gamma(\alpha+1)$ as a damping parameter. Let us denote this method R-LFN$_1$. \\
Now we introduce the design of a fractional Newton method with Caputo derivative and without damping parameter, as shown in the next subsection. Let us denote this method CFN$_2$.

\section{Convergence analysis}

\subsection{Newton method with Caputo derivative}

\begin{teo}
Let the continuous function $f:D\subseteq\mathbb{R}\longrightarrow\mathbb{R}$ has fractional derivative with order $k\alpha$, for any positive integer $k$ and any $\alpha\in(0,1]$, in the interval $D$ containing the zero $\bar{x}$ of $f(x)$. Let us suppose $cD_{\bar{x}}^{\alpha}f(x)$ is continuous and not null at $\bar{x}$. If an initial approximation $x_0$ is sufficiently close to $\bar{x}$, then the local convergence order of the fractional Newton method of Caputo type
\begin{equation}\label{e7}
x_{k+1}=x_k-\left(\Gamma(\alpha+1)\dfrac{f(x_k)}{cD_{\bar{x}}^{\alpha}f(x_k)}\right)^{1/\alpha},\hspace{10pt}k=0,1,2,\dots
\end{equation}
is at least $\alpha+1$, being $0<\alpha\leq1$, and the error equation is
$$e_{k+1}=\dfrac{\Gamma(2\alpha+1)-(\Gamma(\alpha+1))^2}{\alpha(\Gamma(\alpha+1))^2}C_2e_k^{\alpha+1}+O\left(e_k^{2\alpha+1}\right).$$
\end{teo}
\begin{proof}
The Taylor expansion of $f(x)$ and its Caputo-derivative at $x_k$ around $\bar{x}$ can be expressed by \\
$$f(x_k)=\dfrac{cD_{\bar{x}}^{\alpha}f(\bar{x})}{\Gamma(\alpha+1)}\left[e_k^{\alpha}+C_2e_k^{2\alpha}+C_3e_k^{3\alpha}\right]+O\left(e_k^{4\alpha}\right)$$
and
$$cD_{\bar{x}}^{\alpha}f(x_k)=\dfrac{cD_{\bar{x}}^{\alpha}f(\bar{x})}{\Gamma(\alpha+1)}\left[\Gamma(\alpha+1)+\dfrac{\Gamma(2\alpha+1)}{\Gamma(\alpha+1)}C_2e_k^{\alpha}+\dfrac{\Gamma(3\alpha+1)}{\Gamma(2\alpha+1)}C_3e_k^{2\alpha}\right]+O\left(e_k^{3\alpha}\right)$$
being $C_j=\dfrac{\Gamma(\alpha+1)}{\Gamma(j\alpha+1)}\dfrac{cD_{\bar{x}}^{j\alpha}f(x_k)}{cD_{\bar{x}}^{\alpha}f(x_k)}$ for $j\geq2$. \\
The quotient $\dfrac{f(x_k)}{cD_{\bar{x}}^{\alpha}f(x_k)}$ can be calculated as
$$\dfrac{f(x_k)}{cD_{\bar{x}}^{\alpha}f(x_k)}=\dfrac{1}{\Gamma(\alpha+1)}e_k^{\alpha}+\dfrac{(\Gamma(\alpha+1))^2-\Gamma(2\alpha+1)}{(\Gamma(\alpha+1))^3}C_2e_k^{2\alpha}+O\left(e_k^{3\alpha}\right),$$
that multiplying by $\Gamma(\alpha+1)$ results
$$\Gamma(\alpha+1)\dfrac{f(x_k)}{cD_{\bar{x}}^{\alpha}f(x_k)}=e_k^{\alpha}+\dfrac{(\Gamma(\alpha+1))^2-\Gamma(2\alpha+1)}{(\Gamma(\alpha+1))^2}C_2e_k^{2\alpha}+O\left(e_k^{3\alpha}\right).$$
The expansion of Newton's binomial for fractional powers is given by
$$(x+y)^{r}=\sum_{k=0}^N\binom{r}{k}x^{r-k}y^k,\hspace{10pt}k, N\in\{0\}\cup\mathbb{N},$$
where the generalized binomial coefficient is (see \cite{AS})
$$\binom{r}{k}=\dfrac{\Gamma(r+1)}{k!\Gamma(r-k+1)},\hspace{10pt}k\in\{0\}\cup\mathbb{N}.$$
Thus,
\begin{eqnarray*}
\left(\Gamma(\alpha+1)\dfrac{f(x_k)}{cD_{\bar{x}}^{\alpha}f(x_k)}\right)^{1/\alpha}&=&\left(e_k^{\alpha}+\dfrac{(\Gamma(\alpha+1))^2-\Gamma(2\alpha+1)}{(\Gamma(\alpha+1))^2}C_2e_k^{2\alpha}\right)^{1/\alpha} \\
&=&e_k+\dfrac{\Gamma(\alpha+1)}{1!\Gamma(\alpha)}e_k^{1-\alpha}\dfrac{(\Gamma(\alpha+1))^2-\Gamma(2\alpha+1)}{(\Gamma(\alpha+1))^2}C_2e_k^{2\alpha}+O\left(e_k^{2\alpha+1}\right).
\end{eqnarray*}
As $\Gamma(1/\alpha+1)=\dfrac{1}{\alpha}\Gamma(1/\alpha)$, so simplifying
\begin{equation*}
\left(\Gamma(\alpha+1)\dfrac{f(x_k)}{cD_{\bar{x}}^{\alpha}f(x_k)}\right)^{1/\alpha}=e_k+\dfrac{(\Gamma(\alpha+1))^2-\Gamma(2\alpha+1)}{\alpha(\Gamma(\alpha+1))^2}C_2e_k^{\alpha+1}+O\left(e_k^{2\alpha+1}\right).
\end{equation*}
Let $x_{k+1}=e_{k+1}+\bar{x}$ and $x_k=e_k+\bar{x}$.
$$e_{k+1}+\bar{x}=e_k+\bar{x}-e_k+\dfrac{\Gamma(2\alpha+1)-(\Gamma(\alpha+1))^2}{\alpha(\Gamma(\alpha+1))^2}C_2e_k^{\alpha+1}+O\left(e_k^{2\alpha+1}\right).$$
Therefore
$$e_{k+1}=\dfrac{\Gamma(2\alpha+1)-(\Gamma(\alpha+1))^2}{\alpha(\Gamma(\alpha+1))^2}C_2e_k^{\alpha+1}+O\left(e_k^{2\alpha+1}\right).$$
\end{proof}
In the next subsection we introduce the design of a fractional Newton method with Riemann-Liouville derivative and without damping parameter. Let us denote this method R-LFN$_2$.

\subsection{Newton method with Riemann-Liouville derivative}

\begin{teo}
Let the continuous function $f:D\subseteq\mathbb{R}\longrightarrow\mathbb{R}$ has fractional derivative with order $k\alpha$, for any positive integer $k$ and any $\alpha\in(0,1]$, in the interval $D$ containing the zero $\bar{x}$ of $f(x)$. Let us suppose $D_{a^+}^{\alpha k}f(x)$ is continuous and nonsingular at $\bar{x}$. If an initial approximation $x_0$ is sufficiently close to $\bar{x}$, then the local convergence order of the fractional Newton method of Riemann-Liouville type
\begin{equation}\label{e8}
x_{k+1}=x_k-\left(\Gamma(\alpha+1)\dfrac{f(x_k)}{D_{a^+}^{\alpha k}f(x_k)}\right)^{1/\alpha},\hspace{10pt}k=0,1,2,\dots
\end{equation}
is at least $\alpha+1$, being $0<\alpha\leq1$, and again the error equation is
$$e_{k+1}=\dfrac{\Gamma(2\alpha+1)-(\Gamma(\alpha+1))^2}{\alpha(\Gamma(\alpha+1))^2}C_2e_k^{\alpha+1}+O\left(e_k^{2\alpha+1}\right).$$
\end{teo}
Now we introduce the design of a fractional Traub method with Caputo derivative and without damping parameter using CFN$_2$ as first step, as shown in the next subsection. Let us denote this method CFT.

\subsection{Traub method with Caputo derivative}

\begin{teo}
Let the continuous function $f:D\subseteq\mathbb{R}\longrightarrow\mathbb{R}$ has fractional derivative with order $k\alpha$, for any positive integer $k$ and any $\alpha\in(0,1]$, in the interval $D$ containing the zero $\bar{x}$ of $f(x)$. Let us suppose $cD_{\bar{x}}^{\alpha}f(x)$ is continuous and not null at $\bar{x}$. If an initial approximation $x_0$ is sufficiently close to $\bar{x}$, then the local convergence order of the fractional Traub method of Caputo type
\begin{equation}\label{e9}
x_{k+1}=y_k-\left(\Gamma(\alpha+1)\dfrac{f(y_k)}{cD_{\bar{x}}^{\alpha}f(x_k)}\right)^{1/\alpha},\hspace{10pt}k=0,1,2,\dots
\end{equation}
being
$$y_k=x_k-\left(\Gamma(\alpha+1)\dfrac{f(x_k)}{cD_{\bar{x}}^{\alpha}f(x_k)}\right)^{1/\alpha},\hspace{10pt}k=0,1,2,\dots$$
is at least $2\alpha+1$, being $0<\alpha\leq1$, and the error equation is
$$e_{k+1}=\left(\dfrac{B}{\alpha A^{1-1/\alpha}C_2^{\alpha-1}}+\dfrac{1}{\alpha}\left(\dfrac{A\Gamma(2\alpha+1)}{(\Gamma(\alpha+1))^2}-B\right)\right)e_k^{2\alpha+1}+O\left(e_k^{3\alpha+1}\right)$$
being
$$A=\left(\dfrac{\Gamma(2\alpha+1)-(\Gamma(\alpha+1))^2}{\alpha(\Gamma(\alpha+1))^2}\right)^{\alpha}C_2^{\alpha}$$
and
\begin{align*}
B&=\alpha\left(\dfrac{\Gamma(2\alpha+1)-(\Gamma(\alpha+1))^2}{\alpha(\Gamma(\alpha+1))^2}\right)^{\alpha-1}C_2^{\alpha-1} \\
&\left(\dfrac{1}{\alpha}\left(\dfrac{\Gamma(3\alpha+1)-\Gamma(2\alpha+1)\Gamma(\alpha+1)}{\Gamma(2\alpha+1)}C_3+\Gamma(2\alpha+1)\dfrac{(\Gamma(\alpha+1))^2-\Gamma(2\alpha+1)}{(\Gamma(\alpha+1))^3}C_2^2\right)\right. \\
&+\left.\dfrac{1}{2\alpha}\left(1-\dfrac{1}{\alpha}\right)\left(\dfrac{((\Gamma(\alpha+1))^2-\Gamma(2\alpha+1))^2}{(\Gamma(\alpha+1))^4}\right)C_2^2\right).
\end{align*}
\end{teo}
\begin{proof}
The Taylor expansion of $f(x)$ and its Caputo-derivative at $x_k$ around $\bar{x}$ can be expressed by \\
$$f(x_k)=\dfrac{cD_{\bar{x}}^{\alpha}f(\bar{x})}{\Gamma(\alpha+1)}\left[e_k^{\alpha}+C_2e_k^{2\alpha}+C_3e_k^{3\alpha}+C_4e_k^{4\alpha}\right]+O\left(e_k^{5\alpha}\right)$$
and
\begin{align*}
cD_{\bar{x}}^{\alpha}f(x_k)&=\dfrac{cD_{\bar{x}}^{\alpha}f(\bar{x})}{\Gamma(\alpha+1)}\left[\Gamma(\alpha+1)+\dfrac{\Gamma(2\alpha+1)}{\Gamma(\alpha+1)}C_2e_k^{\alpha}\right. \\
&+\left.\dfrac{\Gamma(3\alpha+1)}{\Gamma(2\alpha+1)}C_3e_k^{2\alpha}+\dfrac{\Gamma(4\alpha+1)}{\Gamma(3\alpha+1)}C_4e_k^{3\alpha}\right]+O\left(e_k^{4\alpha}\right)
\end{align*}
being $C_j=\dfrac{\Gamma(\alpha+1)}{\Gamma(j\alpha+1)}\dfrac{cD_{\bar{x}}^{j\alpha}f(x_k)}{cD_{\bar{x}}^{\alpha}f(x_k)}$ for $j\geq2$. \\
The quotient $\dfrac{f(x_k)}{cD_{\bar{x}}^{\alpha}f(x_k)}$ can be calculated as
\begin{align*}
\dfrac{f(x_k)}{cD_{\bar{x}}^{\alpha}f(x_k)}&=\dfrac{1}{\Gamma(\alpha+1)}e_k^{\alpha}+\dfrac{(\Gamma(\alpha+1))^2-\Gamma(2\alpha+1)}{(\Gamma(\alpha+1))^3}C_2e_k^{2\alpha} \\
&+\left(\left(\dfrac{\Gamma(2\alpha+1)\Gamma(\alpha+1)-\Gamma(3\alpha+1)}{\Gamma(2\alpha+1)\Gamma(\alpha+1)}\right)C_3\right. \\
&\left.-\dfrac{\Gamma(2\alpha+1)}{\Gamma(\alpha+1)}\left(\dfrac{(\Gamma(\alpha+1))^2-\Gamma(2\alpha+1)}{(\Gamma(\alpha+1))^3}\right)C_2^2\right)e_k^{3\alpha}+O\left(e_k^{4\alpha}\right),
\end{align*}
that multiplying by $\Gamma(\alpha+1)$ results
\begin{align*}
\Gamma(\alpha+1)\dfrac{f(x_k)}{cD_{\bar{x}}^{\alpha}f(x_k)}&=e_k^{\alpha}+\dfrac{(\Gamma(\alpha+1))^2-\Gamma(2\alpha+1)}{(\Gamma(\alpha+1))^2}C_2e_k^{2\alpha} \\
&+\left(\left(\dfrac{\Gamma(2\alpha+1)\Gamma(\alpha+1)-\Gamma(3\alpha+1)}{\Gamma(2\alpha+1)}\right)C_3\right. \\
&\left.-\Gamma(2\alpha+1)\left(\dfrac{(\Gamma(\alpha+1))^2-\Gamma(2\alpha+1)}{(\Gamma(\alpha+1))^3}\right)C_2^2\right)e_k^{3\alpha}+O\left(e_k^{4\alpha}\right).
\end{align*}
Raising this expression to the power $1/\alpha$:
\begin{align*}
\left(\Gamma(\alpha+1)\dfrac{f(x_k)}{cD_{\bar{x}}^{\alpha}f(x_k)}\right)^{1/\alpha}&=e_k+\dfrac{1}{\alpha}e_k^{1-\alpha}\left(\left(\dfrac{(\Gamma(\alpha+1))^2-\Gamma(2\alpha+1)}{(\Gamma(\alpha+1))^2}C_2e_k^{2\alpha}\right)\right. \\
&+\left(\dfrac{\Gamma(2\alpha+1)\Gamma(\alpha+1)-\Gamma(3\alpha+1)}{\Gamma(2\alpha+1)}C_3\right. \\
&\left.\left.-\Gamma(2\alpha+1)\dfrac{(\Gamma(\alpha+1))^2-\Gamma(2\alpha+1)}{(\Gamma(\alpha+1))^3}C_2^2\right)e_k^{3\alpha}\right) \\
&+\dfrac{\Gamma(1/\alpha+1)}{2\Gamma(1/\alpha-1)}e_k^{1-2\alpha}\left(\dfrac{(\Gamma(\alpha+1))^2-\Gamma(2\alpha+1)}{(\Gamma(\alpha+1))^2}C_2e_k^{2\alpha}\right)^2 \\
&+O\left(e_k^{3\alpha+1}\right).
\end{align*}
As $\Gamma(1/\alpha+1)=\dfrac{1}{\alpha}\Gamma(1/\alpha)=\dfrac{1}{\alpha}(1/\alpha-1)\Gamma(1/\alpha-1)$, this implies that $\dfrac{\Gamma(1/\alpha+1)}{2\Gamma(1/\alpha-1)}=\dfrac{1}{2\alpha}\left(\dfrac{1}{\alpha}-1\right)$. Simplifying:
\begin{align*}
\left(\Gamma(\alpha+1)\dfrac{f(x_k)}{cD_{\bar{x}}^{\alpha}f(x_k)}\right)^{1/\alpha}&=e_k+\left(\dfrac{(\Gamma(\alpha+1))^2-\Gamma(2\alpha+1)}{\alpha(\Gamma(\alpha+1))^2}\right) C_2e_k^{\alpha+1} \\
&+\left(\dfrac{1}{\alpha}\left(\dfrac{\Gamma(2\alpha+1)\Gamma(\alpha+1)-\Gamma(3\alpha+1)}{\Gamma(2\alpha+1)}C_3\right.\right. \\
&+\left.\Gamma(2\alpha+1)\dfrac{\Gamma(2\alpha+1)-(\Gamma(\alpha+1))^2}{(\Gamma(\alpha+1))^3}C_2^2\right) \\
&+\left.\dfrac{1}{2\alpha}\left(\dfrac{1}{\alpha}-1\right)\left(\dfrac{((\Gamma(\alpha+1))^2-\Gamma(2\alpha+1))^2}{(\Gamma(\alpha+1))^4}\right)C_2^2\right)e_k^{2\alpha+1} \\
&+O\left(e_k^{3\alpha+1}\right).
\end{align*}
Let $e_k=x_k-\bar{x}$, we can say that
\begin{align*}
y_k&=\bar{x}-\left(\dfrac{(\Gamma(\alpha+1))^2-\Gamma(2\alpha+1)}{\alpha(\Gamma(\alpha+1))^2}\right)C_2e_k^{\alpha+1} \\
&-\left(\dfrac{1}{\alpha}\left(\dfrac{\Gamma(2\alpha+1)\Gamma(\alpha+1)-\Gamma(3\alpha+1)}{\Gamma(2\alpha+1)}C_3\right.\right. \\
&+\left.\Gamma(2\alpha+1)\dfrac{\Gamma(2\alpha+1)-(\Gamma(\alpha+1))^2}{(\Gamma(\alpha+1))^3}C_2^2\right) \\
&+\left.\dfrac{1}{2\alpha}\left(\dfrac{1}{\alpha}-1\right)\left(\dfrac{((\Gamma(\alpha+1))^2-\Gamma(2\alpha+1))^2}{(\Gamma(\alpha+1))^4}\right)C_2^2\right)e_k^{2\alpha+1}+O\left(e_k^{3\alpha+1}\right).
\end{align*}
Let us evaluate $f(y_k)$:
$$f(y_k)=\dfrac{cD_{\bar{x}}^{\alpha}f(x_k)}{\Gamma(\alpha+1)}\left[\left(y_k-\bar{x}\right)^{\alpha}+C_2\left(y_k-\bar{x}\right)^{2\alpha}\right]+O\left(e_k^{3\alpha}\right)$$
where
\begin{align*}
\left(y_k-\bar{x}\right)^{\alpha}&=\left(\left(\dfrac{\Gamma(2\alpha+1)-(\Gamma(\alpha+1))^2}{\alpha(\Gamma(\alpha+1))^2}\right)C_2e_k^{\alpha+1}\right. \\
&+\left(\dfrac{1}{\alpha}\left(\dfrac{\Gamma(3\alpha+1)-\Gamma(2\alpha+1)\Gamma(\alpha+1)}{\Gamma(2\alpha+1)}C_3\right.\right. \\
&+\left.\Gamma(2\alpha+1)\dfrac{(\Gamma(\alpha+1))^2-\Gamma(2\alpha+1)}{(\Gamma(\alpha+1))^3}C_2^2\right) \\
&+\left.\left.\dfrac{1}{2\alpha}\left(1-\dfrac{1}{\alpha}\right)\left(\dfrac{((\Gamma(\alpha+1))^2-\Gamma(2\alpha+1))^2}{(\Gamma(\alpha+1))^4}\right)C_2^2\right)e_k^{2\alpha+1}\right)^{\alpha}+O\left(e_k^{\alpha^2+3\alpha}\right) \\
&=\left(\dfrac{\Gamma(2\alpha+1)-(\Gamma(\alpha+1))^2}{\alpha(\Gamma(\alpha+1))^2}\right)^{\alpha}C_2^{\alpha}e_k^{\alpha^2+\alpha} \\
&+\alpha\left(\dfrac{\Gamma(2\alpha+1)-(\Gamma(\alpha+1))^2}{\alpha(\Gamma(\alpha+1))^2}\right)^{\alpha-1}C_2^{\alpha-1}e_k^{\alpha^2-1} \\
&\left(\dfrac{1}{\alpha}\left(\dfrac{\Gamma(3\alpha+1)-\Gamma(2\alpha+1)\Gamma(\alpha+1)}{\Gamma(2\alpha+1)}C_3+\Gamma(2\alpha+1)\dfrac{(\Gamma(\alpha+1))^2-\Gamma(2\alpha+1)}{(\Gamma(\alpha+1))^3}C_2^2\right)\right. \\
&+\left.\dfrac{1}{2\alpha}\left(1-\dfrac{1}{\alpha}\right)\left(\dfrac{((\Gamma(\alpha+1))^2-\Gamma(2\alpha+1))^2}{(\Gamma(\alpha+1))^4}\right)C_2^2\right)e_k^{2\alpha+1}+O\left(e_k^{\alpha^2+3\alpha}\right) \\
&=\left(\dfrac{\Gamma(2\alpha+1)-(\Gamma(\alpha+1))^2}{\alpha(\Gamma(\alpha+1))^2}\right)^{\alpha}C_2^{\alpha}e_k^{\alpha^2+\alpha} \\
&+\alpha\left(\dfrac{\Gamma(2\alpha+1)-(\Gamma(\alpha+1))^2}{\alpha(\Gamma(\alpha+1))^2}\right)^{\alpha-1}C_2^{\alpha-1} \\
&\left(\dfrac{1}{\alpha}\left(\dfrac{\Gamma(3\alpha+1)-\Gamma(2\alpha+1)\Gamma(\alpha+1)}{\Gamma(2\alpha+1)}C_3+\Gamma(2\alpha+1)\dfrac{(\Gamma(\alpha+1))^2-\Gamma(2\alpha+1)}{(\Gamma(\alpha+1))^3}C_2^2\right)\right. \\
&+\left.\dfrac{1}{2\alpha}\left(1-\dfrac{1}{\alpha}\right)\left(\dfrac{((\Gamma(\alpha+1))^2-\Gamma(2\alpha+1))^2}{(\Gamma(\alpha+1))^4}\right)C_2^2\right)e_k^{\alpha^2+2\alpha}+O\left(e_k^{\alpha^2+3\alpha}\right).
\end{align*}
We can see that $\alpha^2+3\alpha\geq3\alpha$ for all $\alpha\in[0,1]$. It is also clear that if we choose the first element of expansion $C_2\left(y_k-\bar{x}\right)^{2\alpha}$ will have order $(\alpha+1)2\alpha=2\alpha^2+2\alpha\geq3\alpha$ for all $\alpha\in[0.5,1]$. So,
\begin{align*}
f(y_k)&=\dfrac{cD_{\bar{x}}^{\alpha}f(x_k)}{\Gamma(\alpha+1)}\left[\left(\dfrac{\Gamma(2\alpha+1)-(\Gamma(\alpha+1))^2}{\alpha(\Gamma(\alpha+1))^2}\right)^{\alpha}C_2^{\alpha}e_k^{\alpha^2+\alpha}\right. \\
&+\alpha\left(\dfrac{\Gamma(2\alpha+1)-(\Gamma(\alpha+1))^2}{\alpha(\Gamma(\alpha+1))^2}\right)^{\alpha-1}C_2^{\alpha-1} \\
&\left(\dfrac{1}{\alpha}\left(\dfrac{\Gamma(3\alpha+1)-\Gamma(2\alpha+1)\Gamma(\alpha+1)}{\Gamma(2\alpha+1)}C_3+\Gamma(2\alpha+1)\dfrac{(\Gamma(\alpha+1))^2-\Gamma(2\alpha+1)}{(\Gamma(\alpha+1))^3}C_2^2\right)\right. \\
&+\left.\left.\dfrac{1}{2\alpha}\left(1-\dfrac{1}{\alpha}\right)\left(\dfrac{((\Gamma(\alpha+1))^2-\Gamma(2\alpha+1))^2}{(\Gamma(\alpha+1))^4}\right)C_2^2\right)e_k^{\alpha^2+2\alpha}\right]+O\left(e_k^{\alpha^2+3\alpha}\right).
\end{align*}
Let us call
$$A=\left(\dfrac{\Gamma(2\alpha+1)-(\Gamma(\alpha+1))^2}{\alpha(\Gamma(\alpha+1))^2}\right)^{\alpha}C_2^{\alpha}$$
and
\begin{align*}
B&=\alpha\left(\dfrac{\Gamma(2\alpha+1)-(\Gamma(\alpha+1))^2}{\alpha(\Gamma(\alpha+1))^2}\right)^{\alpha-1}C_2^{\alpha-1} \\
&\left(\dfrac{1}{\alpha}\left(\dfrac{\Gamma(3\alpha+1)-\Gamma(2\alpha+1)\Gamma(\alpha+1)}{\Gamma(2\alpha+1)}C_3+\Gamma(2\alpha+1)\dfrac{(\Gamma(\alpha+1))^2-\Gamma(2\alpha+1)}{(\Gamma(\alpha+1))^3}C_2^2\right)\right. \\
&+\left.\dfrac{1}{2\alpha}\left(1-\dfrac{1}{\alpha}\right)\left(\dfrac{((\Gamma(\alpha+1))^2-\Gamma(2\alpha+1))^2}{(\Gamma(\alpha+1))^4}\right)C_2^2\right),
\end{align*}
then
$$f(y_k)=\dfrac{cD_{\bar{x}}^{\alpha}f(x_k)}{\Gamma(\alpha+1)}\left[Ae_k^{\alpha^2+\alpha}+Be_k^{\alpha^2+2\alpha}\right]+O\left(e_k^{\alpha^2+3\alpha}\right).$$
The quotient $\dfrac{f(y_k)}{cD_{\bar{x}}^{\alpha}f(x_k)}$ results
$$\dfrac{f(y_k)}{cD_{\bar{x}}^{\alpha}f(x_k)}=\dfrac{Ae_k^{\alpha^2+\alpha}}{\Gamma(\alpha+1)}+\dfrac{1}{\Gamma(\alpha+1)}\left(B-\dfrac{A\Gamma(2\alpha+1)}{(\Gamma(\alpha+1))^2}\right)e_k^{\alpha^2+2\alpha}+O\left(e_k^{\alpha^2+3\alpha}\right),$$
that multiplying by $\Gamma(\alpha+1)$ results
$$\Gamma(\alpha+1)\dfrac{f(y_k)}{cD_{\bar{x}}^{\alpha}f(x_k)}=Ae_k^{\alpha^2+\alpha}+\left(B-\dfrac{A\Gamma(2\alpha+1)}{(\Gamma(\alpha+1))^2}\right)e_k^{\alpha^2+2\alpha}+O\left(e_k^{\alpha^2+3\alpha}\right).$$
Raising this expression to the power $1/\alpha$:
\begin{align*}
\left(\Gamma(\alpha+1)\dfrac{f(y_k)}{cD_{\bar{x}}^{\alpha}f(x_k)}\right)^{1/\alpha}&=A^{1/\alpha}e_k^{\alpha+1}+\dfrac{1}{\alpha}e_k^{-\alpha^2+1}\left(B-\dfrac{A\Gamma(2\alpha+1)}{(\Gamma(\alpha+1))^2}\right)e_k^{\alpha^2+2\alpha}+O\left(e_k^{3\alpha+1}\right) \\
&=A^{1/\alpha}e_k^{\alpha+1}+\dfrac{1}{\alpha}\left(B-\dfrac{A\Gamma(2\alpha+1)}{(\Gamma(\alpha+1))^2}\right)e_k^{2\alpha+1}+O\left(e_k^{3\alpha+1}\right).
\end{align*}
Let $x_{k+1}=e_{k+1}+\bar{x}$ and $x_k=e_k+\bar{x}$.
\begin{align*}
e_{k+1}+\bar{x}&=e_k+\bar{x}+\left[-e_k+A^{1/\alpha}e_k^{\alpha+1}+\left(\dfrac{B}{\alpha A^{1-1/\alpha}C_2^{\alpha-1}}\right)e_k^{2\alpha+1}\right] \\
&-\left[A^{1/\alpha}e_k^{\alpha+1}+\dfrac{1}{\alpha}\left(B-\dfrac{A\Gamma(2\alpha+1)}{(\Gamma(\alpha+1))^2}\right)e_k^{2\alpha+1}\right]+O\left(e_k^{3\alpha+1}\right).
\end{align*}
Therefore
$$e_{k+1}=\left(\dfrac{B}{\alpha A^{1-1/\alpha}C_2^{\alpha-1}}+\dfrac{1}{\alpha}\left(\dfrac{A\Gamma(2\alpha+1)}{(\Gamma(\alpha+1))^2}-B\right)\right)e_k^{2\alpha+1}+O\left(e_k^{3\alpha+1}\right).$$
\end{proof}
In the next subsection we introduce the design of a fractional Traub method with Riemann-Liouville derivative and without damping parameter using R-LFN$_2$ as first step. Let us denote this method R-LFT.

\subsection{Traub method with Riemann-Liouville derivative}

\begin{teo}
Let the continuous function $f:D\subseteq\mathbb{R}\longrightarrow\mathbb{R}$ has fractional derivative with order $k\alpha$, for any positive integer $k$ and any $\alpha\in(0,1]$, in the interval $D$ containing the zero $\bar{x}$ of $f(x)$. Let us suppose $D_{a^+}^{\alpha k}f(x)$ is continuous and nonsingular at $\bar{x}$. If an initial approximation $x_0$ is sufficiently close to $\bar{x}$, then the local convergence order of the fractional Traub method of Riemann-Liouville type
\begin{equation}\label{e10}
x_{k+1}=y_k-\left(\Gamma(\alpha+1)\dfrac{f(y_k)}{D_{a^+}^{\alpha k}f(x_k)}\right)^{1/\alpha},\hspace{10pt}k=0,1,2,\dots
\end{equation}
being
$$y_k=x_k-\left(\Gamma(\alpha+1)\dfrac{f(x_k)}{D_{a^+}^{\alpha k}f(x_k)}\right)^{1/\alpha},\hspace{10pt}k=0,1,2,\dots$$
is at least $2\alpha+1$, being $0<\alpha\leq1$, and again the error equation is
$$e_{k+1}=\left(\dfrac{B}{\alpha A^{1-1/\alpha}C_2^{\alpha-1}}+\dfrac{1}{\alpha}\left(\dfrac{A\Gamma(2\alpha+1)}{(\Gamma(\alpha+1))^2}-B\right)\right)e_k^{2\alpha+1}+O\left(e_k^{3\alpha+1}\right)$$
being
$$A=\left(\dfrac{\Gamma(2\alpha+1)-(\Gamma(\alpha+1))^2}{\alpha(\Gamma(\alpha+1))^2}\right)^{\alpha}C_2^{\alpha}$$
and
\begin{align*}
B&=\alpha\left(\dfrac{\Gamma(2\alpha+1)-(\Gamma(\alpha+1))^2}{\alpha(\Gamma(\alpha+1))^2}\right)^{\alpha-1}C_2^{\alpha-1} \\
&\left(\dfrac{1}{\alpha}\left(\dfrac{\Gamma(3\alpha+1)-\Gamma(2\alpha+1)\Gamma(\alpha+1)}{\Gamma(2\alpha+1)}C_3+\Gamma(2\alpha+1)\dfrac{(\Gamma(\alpha+1))^2-\Gamma(2\alpha+1)}{(\Gamma(\alpha+1))^3}C_2^2\right)\right. \\
&+\left.\dfrac{1}{2\alpha}\left(1-\dfrac{1}{\alpha}\right)\left(\dfrac{((\Gamma(\alpha+1))^2-\Gamma(2\alpha+1))^2}{(\Gamma(\alpha+1))^4}\right)C_2^2\right).
\end{align*}
\end{teo}
In the next section we are going to test functions, and analyze the dependence on the initial estimation of the Newton and Traub methods shown before.

\section{Numerical stability}

In this section we will be using Matlab R2018b with double precission arithmetics, $|x_{k+1}-x_k|<10^{-8}$ or $|f(x_{k+1})|<10^{-8}$ as stopping criterium, and a maximum of 500 iterations. For calculation of Gamma function we use the program made in \cite{CL}. For Mittag-Leffler function we use the program provided by Igor Podlubny in Mathworks.

\subsection{Numerical results}

In this subsection we are going to test 4 functions in order to make a comparisson between the methods desgned before by using different initial estimations. \\
Our first function is $f_1(x)=-12.84x^6-25.6x^5+16.55x^4-2.21x^3+26.71x^2-4.29x-15.21$ with roots $\bar{x}_1=0.82366+0.24769i$, $\bar{x}_2=0.82366-0.24769i$, $\bar{x}_3=-2.62297$, $\bar{x}_4=-0.584$, $\bar{x}_5=-0.21705+0.99911i$ and $\bar{x}_6=-0.21705-0.99911i$. \\
In table \ref{t1} we can see that CFN$_1$ requires less iterations than CFN$_2$ for a real value of $x_0$, while in tables \ref{t2} and \ref{t3} CFN$_2$ requires less iterations than CFN$_1$ for a large value of imaginary part.
\begin{center}
\begin{tabular}{|c|c|c|c|c|c|c|c|c|}
\hline
 & \multicolumn{4}{c|}{CFN$_1$ method} & \multicolumn{4}{c|}{CFN$_2$ method} \\ \hline
$\alpha$ & $\bar{x}$ & $|x_{k+1}-x_k|$ & $|f(x_{k+1})|$ & iter & $\bar{x}$ & $|x_{k+1}-x_k|$ & $|f(x_{k+1})|$ & iter \\ \hline
0.6 & -0.85348-0.1491i & 0.29821 & 28.343 & 500 & -0.58406 & 1.7603e-07 & 0.0035619 & 500 \\
0.65 & -0.71052-0.087443i & 0.17488 & 11.329 & 500 & -0.58401 & 4.1154e-08 & 6.7515e-04 & 500 \\
0.7 & -0.62035-0.029249i & 0.058499 & 2.98929 & 500 & -0.584 & 9.9926e-09 & 1.1322e-04 & 432 \\
0.75 & -0.584+4.9256e-09i & 9.6537e-09 & 4.1645e-07 & 151 & -0.584 & 9.8524e-09 & 4.6756e-05 & 230 \\
0.8 & -0.584-2.882e-09i & 8.5475e-09 & 3.0465e-07 & 50 & -0.584 & 9.6579e-09 & 1.8943e-05 & 124 \\
0.85 & -0.584-2.5108e-09i & 9.468e-09 & 2.606e-07 & 28 & -0.584 & 9.9396e-09 & 7.7541e-06 & 67 \\
0.9 & -0.584+1.0144e-09i & 3.9203e-09 & 7.3851e-08 & 19 & -0.584 & 9.109e-09 & 2.6706e-06 & 37 \\
0.95 & -0.584+3.1405e-10i & 2.5822e-09 & 2.4894e-08 & 13 & -0.584 & 7.3622e-09 & 6.4461e-07 & 20 \\
1 & -0.584 & 3.0876e-06 & 8.8694e-10 & 6 & -0.584 & 3.0876e-06 & 8.8694e-10 & 6 \\ \hline
\end{tabular}
\captionof{table}{Fractional Newton$_1$ and Newton$_2$ results for $f_1(x)$ with Caputo derivative and initial estimation $x_0=-1.5$}\label{t1}
\end{center}
\begin{center}
\begin{tabular}{|c|c|c|c|c|c|c|c|c|}
\hline
 & \multicolumn{4}{c|}{CFN$_1$ method} & \multicolumn{4}{c|}{CFN$_2$ method} \\ \hline
$\alpha$ & $\bar{x}$ & $|x_{k+1}-x_k|$ & $|f(x_{k+1})|$ & iter & $\bar{x}$ & $|x_{k+1}-x_k|$ & $|f(x_{k+1})|$ & iter \\ \hline
0.6 & -7.22e+02-1.05e+01i & 15.573 & 1.8225e+18 & 500 & -0.21705+0.99914i & 1.2135e-07 & 0.0052781 & 500 \\
0.65 & -2.7481-0.1014i & 0.2028 & 4.26e+02 & 500 & -0.21705+0.99911i & 2.9641e-08 & 0.0010503 & 500 \\
0.7 & -2.6804-0.05406i & 0.10812 & 1.8541e+02 & 500 & -0.21705+0.99911i & 9.9222e-09 & 2.2511e-04 & 410 \\
0.75 & -2.629+0.0062659i &  0.012531 & 18.709 & 500 & -0.21705+0.99911i & 9.9267e-09 & 9.7312e-05 & 249 \\
0.8 & -2.6229+3.6288e-09i &  8.1391e-09 & 1.0204e-05 & 203 & -0.21705+0.99911i & 9.9588e-09 & 4.1526e-05 & 161 \\
0.85 & -2.6229-2.8915e-09i & 8.548e-09 & 8.3609e-06 & 128 & -0.21705+0.99911i & 9.214e-09 & 1.6051e-05 & 113 \\
0.9 &  -2.6229-1.2931e-09i & 5.0518e-09 & 3.4133e-06 & 92 & -0.21705+0.99911i & 8.5555e-09 & 5.7295e-06 & 85 \\
0.95 & -0.584 &  1.7332e-09 & 1.6709e-08 & 73 & -0.21705+0.99911i & 7.8139e-09 & 1.5795e-06 & 68 \\
1 & -0.21705+0.99911i & 2.8445e-07 & 2.9798e-11 & 54 & -0.21705+0.99911i & 2.8445e-07 & 2.9798e-11 & 54 \\ \hline
\end{tabular}
\captionof{table}{Fractional Newton$_1$ and Newton$_2$ results for $f_1(x)$ with Caputo derivative and initial estimation $x_0=-1.5+1e04i$}\label{t2}
\end{center}
\begin{center}
\begin{tabular}{|c|c|c|c|c|c|c|c|c|}
\hline
 & \multicolumn{4}{c|}{CFN$_1$ method} & \multicolumn{4}{c|}{CFN$_2$ method} \\ \hline
$\alpha$ & $\bar{x}$ & $|x_{k+1}-x_k|$ & $|f(x_{k+1})|$ & iter & $\bar{x}$ & $|x_{k+1}-x_k|$ & $|f(x_{k+1})|$ & iter \\ \hline
0.6 & -7.22e+02-1.03e+01i & 15.571 & 1.8197e+18 & 500 & -0.21705+0.99914i & 1.2135e-07 & 0.0052781 & 500 \\
0.65 & -2.7481-0.1014i & 0.2028 & 4.26e+02 & 500 & -0.21705+0.99911i & 2.9641e-08 & 0.0010503 & 500 \\
0.7 & -2.6804-0.054i & 0.10812 & 1.8541e+02 & 500 & -0.21705+0.99911i & 9.9222e-09 & 2.2511e-04 & 410 \\
0.75 & -2.629+0.0062659i & 0.012531 & 18.709 & 500 & -0.21705+0.99911i & 9.9267e-09 & 9.7312e-05 & 249 \\
0.8 & -2.6229+3.631e-09i & 8.1356e-09 & 1.02e-05 & 203 & -0.21705+0.99911i & 9.9588e-09 & 4.1526e-05 & 161 \\
0.85 & -2.6229-2.8848e-09i & 8.5408e-09 & 8.3538e-06 & 128 & -0.21705+0.99911i & 9.214e-09 & 1.6051e-05 & 113 \\
0.9 & -2.6229-1.2902e-09i & 5.0472e-09 & 3.4103e-06 & 92 & -0.21705+0.99911i & 8.5555e-09 & 5.7295e-06 & 85 \\
0.95 & -0.584+1.5616e-10i & 1.7329e-09 & 1.6707e-08 & 73 & -0.21705+0.99911i & 7.8139e-09 & 1.5795e-06 & 68 \\
1 & -0.21705+0.99911i & 2.8474e-07 & 2.986e-11 & 54 & -0.21705+0.99911i & 2.8474e-07 & 2.986e-11 & 54 \\ \hline
\end{tabular}
\captionof{table}{Fractional Newton$_1$ and Newton$_2$ results for $f_1(x)$ with Caputo derivative and initial estimation $x_0=1e04i$}\label{t3}
\end{center}
In the case of R-LFN$_1$ and R-LFN$_1$ can be observed the same behavior as in Caputo case. In table \ref{t4}, R-LFN$_1$ requires less iterations than R-LFN$_2$ for a real value of $x_0$, while in tables \ref{t5} and \ref{t6} R-LFN$_2$ requires less iterations than R-LFN$_1$ for a large value of imaginary part.
\begin{center}
\begin{tabular}{|c|c|c|c|c|c|c|c|c|}
\hline
 & \multicolumn{4}{c|}{R-LFN$_1$ method} & \multicolumn{4}{c|}{R-LFN$_2$ method} \\ \hline
$\alpha$ & $\bar{x}$ & $|x_{k+1}-x_k|$ & $|f(x_{k+1})|$ & iter & $\bar{x}$ & $|x_{k+1}-x_k|$ & $|f(x_{k+1})|$ & iter \\ \hline
0.6 & -0.82233-0.15603i & 0.31207 & 24.999 & 500 & -0.58402 & 8.2311e-08 & 0.001664 & 500 \\
0.65 & -0.71232-0.10603i & 0.21206 & 12.298 & 500 & -0.584 & 2.0291e-08 & 3.3261e-04 & 500 \\
0.7 & -0.64534-0.058604i & 0.11720 & 5.6486 & 500 & -0.584 & 9.9746e-09 & 9.2638e-05 & 354 \\
0.75 & -0.59261-0.008992i & 0.017984 & 0.76427 & 500 & -0.584 & 9.9858e-09 & 4.0402e-05 & 196 \\
0.8 & -0.584-3.9918e-09i & 9.8676e-09 & 3.5368e-07 & 72 & -0.584 & 9.6983e-09 & 1.6881e-05 & 110 \\
0.85 & -0.584-2.8194e-09i & 8.8611e-09 & 2.4785e-07 & 32 & -0.584 & 9.4322e-09 & 6.8049e-06 & 62 \\
0.9 & -0.584-9.6669e-10i & 4.1273e-09 & 7.9749e-08 & 20 & -0.584 & 8.923e-09 & 2.4689e-06 & 35 \\
0.95 & -0.584+4.9765e-10i & 3.9835e-09 & 3.9713e-08 & 13 & -0.584 & 9.0322e-09 & 7.4692e-07 & 19 \\
1 & -0.584 & 3.0876e-06 & 8.8694e-10 & 6 & -0.584 & 3.0876e-06 & 8.8694e-10 & 6 \\ \hline
\end{tabular}
\captionof{table}{Fractional Newton$_1$ and Newton$_2$ results for $f_1(x)$ with Riemann-Liouville derivative and initial estimation $x_0=-1.5$}\label{t4}
\end{center}
\begin{center}
\begin{tabular}{|c|c|c|c|c|c|c|c|c|}
\hline
 & \multicolumn{4}{c|}{R-LFN$_1$ method} & \multicolumn{4}{c|}{R-LFN$_2$ method} \\ \hline
$\alpha$ & $\bar{x}$ & $|x_{k+1}-x_k|$ & $|f(x_{k+1})|$ & iter & $\bar{x}$ & $|x_{k+1}-x_k|$ & $|f(x_{k+1})|$ & iter \\ \hline
0.6 & -7.22e+02-1.05e+01i & 15.573 & 1.8225e+18 & 500 & -0.21706+0.99913i & 9.3943e-08 & 0.0040855 & 500 \\
0.65 & -2.748-0.10116i & 0.20233 & 4.2559e+02 & 500 & -0.21705+0.99911i & 2.3654e-08 & 8.3809e-04 & 500 \\
0.7 & -2.6802-0.053771i & 0.10754 & 1.8454e+02 & 500 & -0.21705+0.99911i & 9.9497e-09 & 2.1234e-04 & 389 \\
0.75 & -2.6287+0.0059679i & 0.011935 & 17.824 & 500 & -0.21705+0.99911i & 9.8157e-09 & 9.2268e-05 & 241 \\
0.8 & -2.6229-4.1357e-09i & 9.5805e-09 & 1.201e-05 & 202 & -0.21705+0.99911i & 9.7617e-09 & 3.9571e-05 & 158 \\
0.85 & -2.6229-2.824e-09i & 8.3565e-09 & 8.1719e-06 & 128 & -0.21705+0.99911i & 9.8212e-09 & 1.6555e-05 & 111 \\
0.9 & -2.6229-1.2802e-09i & 5.0059e-09 & 3.3815e-06 & 92 & -0.21705+0.99911i & 9.607e-09 & 6.2489e-06 & 84 \\
0.95 & -0.584+2.5552e-10i & 2.6575e-09 & 2.6494e-08 & 73 & -0.21705+0.99911i & 7.1581e-09 & 1.4445e-06 & 68 \\
1 & -0.21705+0.99911i & 2.8445e-07 & 2.9798e-11 & 54 & -0.21705+0.99911i & 2.8445e-07 & 2.9798e-11 & 54 \\ \hline
\end{tabular}
\captionof{table}{Fractional Newton$_1$ and Newton$_2$ results for $f_1(x)$ with Riemann-Liouville derivative and initial estimation $x_0=-1.5+1e04i$}\label{t5}
\end{center}
\begin{center}
\begin{tabular}{|c|c|c|c|c|c|c|c|c|}
\hline
 & \multicolumn{4}{c|}{R-LFN$_1$ method} & \multicolumn{4}{c|}{R-LFN$_2$ method} \\ \hline
$\alpha$ & $\bar{x}$ & $|x_{k+1}-x_k|$ & $|f(x_{k+1})|$ & iter & $\bar{x}$ & $|x_{k+1}-x_k|$ & $|f(x_{k+1})|$ & iter \\ \hline
0.6 & -7.22e+02-1.03e+01i & 15.571 & 1.8197e+18 & 500 & -0.21706+0.99913i & 9.3943e-08 & 0.0040855 & 500 \\
0.65 & -2.748-0.10116i & 0.20233 & 4.2559e+02 & 500 & -0.21705+0.99911i & 2.3654e-08 & 8.3809e-04 & 500 \\
0.7 & -2.6802-0.053771i & 0.10754 & 1.8454e+02 & 500 & -0.21705+0.99911i & 9.9497e-09 & 2.1234e-04 & 389 \\
0.75 & -2.6287+0.0059679i & 0.011935 & 17.824 & 500 & -0.21705+0.99911i & 9.8157e-09 & 9.2268e-05 & 241 \\
0.8 & -2.6229-4.1292e-09i & 9.5763e-09 & 1.2005e-05 & 202 & -0.21705+0.99911i & 9.7617e-09 & 3.9571e-05 & 158 \\
0.85 & -2.6229-2.8175e-09i & 8.3495e-09 & 8.1651e-06 & 128 & -0.21705+0.99911i & 9.8212e-09 & 1.6555e-05 & 111 \\
0.9 & -2.6229-1.2773e-09i & 5.0014e-09 & 3.3785e-06 & 92 & -0.21705+0.99911i & 9.607e-09 & 6.2489e-06 & 84 \\
0.95 & -0.584+2.5644e-10i & 2.6577e-09 & 2.6496e-08 & 73 & -0.21705+0.99911i & 7.1581e-09 & 1.4444e-06 & 68 \\
1 & -0.21705+0.99911i & 2.8474e-07 & 2.986e-11 & 54 & -0.21705+0.99911i & 2.8474e-07 & 2.986e-11 & 54 \\ \hline
\end{tabular}
\captionof{table}{Fractional Newton$_1$ and Newton$_2$ results for $f_1(x)$ with Riemann-Liouville derivative and initial estimation $x_0=1e04i$}\label{t6}
\end{center}
Now, let us compare the Traub methods CFT and R-LFT with their first steps CFN$_2$ and R-LFN$_2$ respectively for $f_1(x)$. In both tables \ref{t7} and \ref{t8} we can see that Taub method requires less iterations than its first step.
\begin{center}
\begin{tabular}{|c|c|c|c|c|c|c|c|c|}
\hline
 & \multicolumn{4}{c|}{CFN$_2$ method} & \multicolumn{4}{c|}{CFT method} \\ \hline
$\alpha$ & $\bar{x}$ & $|x_{k+1}-x_k|$ & $|f(x_{k+1})|$ & iter & $\bar{x}$ & $|x_{k+1}-x_k|$ & $|f(x_{k+1})|$ & iter \\ \hline
0.6 & -0.58406 & 1.7603e-07 & 0.0035619 & 500 & -0.58402 & 6.2898e-08 & 0.0012681 & 500 \\
0.65 & -0.58401 & 4.1154e-08 & 6.7515e-04 & 500 & -0.584 & 1.1562e-08 & 1.8867e-04 & 500 \\
0.7 & -0.584 & 9.9926e-09 & 1.1322e-04 & 432 & -0.584 & 9.9588e-09 & 6.9453e-05 & 268 \\
0.75 & -0.584 & 9.8524e-09 & 4.6756e-05 & 230 & -0.584 & 9.9889e-09 & 2.7995e-05 & 138 \\
0.8 & -0.584 & 9.6579e-09 & 1.8943e-05 & 124 & -0.584 & 9.5606e-09 & 1.0693e-05 & 73 \\
0.85 & -0.584 & 9.9396e-09 & 7.7541e-06 & 67 & -0.584 & 9.4657e-09 & 4.0225e-06 & 39 \\
0.9 & -0.584 & 9.109e-09 & 2.6706e-06 & 37 & -0.584 & 6.8084e-09 & 1.0286e-06 & 22 \\
0.95 & -0.584 & 7.3622e-09 & 6.4461e-07 & 20 & -0.584 & 5.2078e-09 & 1.8928e-07 & 12 \\
1 & -0.584 & 3.0876e-06 & 8.8694e-10 & 6 & -0.584 & 2.2023e-10 & 5.329e-15 & 5 \\ \hline
\end{tabular}
\captionof{table}{Fractional Newton$_2$ and Traub results for $f_1(x)$ with Caputo derivative and initial estimation $x_0=-1.5$}\label{t7}
\end{center}
\begin{center}
\begin{tabular}{|c|c|c|c|c|c|c|c|c|}
\hline
 & \multicolumn{4}{c|}{R-LFN$_2$ method} & \multicolumn{4}{c|}{R-LFT method} \\ \hline
$\alpha$ & $\bar{x}$ & $|x_{k+1}-x_k|$ & $|f(x_{k+1})|$ & iter & $\bar{x}$ & $|x_{k+1}-x_k|$ & $|f(x_{k+1})|$ & iter \\ \hline
0.6 & -0.58402 & 8.2311e-08 & 0.001664 & 500 & -0.58401 & 2.9398e-08 & 5.9231e-04 & 500 \\
0.65 & -0.584 & 2.0291e-08 & 3.3261e-04 & 500 & -0.584 & 9.9696e-09 & 1.3359e-04 & 411 \\
0.7 & -0.584 & 9.9746e-09 & 9.2638e-05 & 354 & -0.584 & 9.9316e-09 & 5.6773e-05 & 220 \\
0.75 & -0.584 & 9.9858e-09 & 4.0402e-05 & 196 & -0.584 & 9.7458e-09 & 2.35e-05 & 119 \\
0.8 & -0.584 & 9.6983e-09 & 1.6881e-05 & 110 & -0.584 & 9.5628e-09 & 9.4891e-06 & 65 \\
0.85 & -0.584 & 9.4322e-09 & 6.8049e-06 & 62 & -0.584 & 9.3134e-09 & 3.6307e-06 & 36 \\
0.9 & -0.584 & 8.923e-09 & 2.4689e-06 & 35 & -0.584 & 9.6151e-09 & 1.3022e-06 & 20 \\
0.95 & -0.584 & 9.0322e-09 & 7.4692e-07 & 19 & -0.584 & 3.6166e-09 & 1.3015e-07 & 12 \\
1 & -0.584 & 3.0876e-06 & 8.8694e-10 & 6 & -0.584 & 2.2023e-10 & 5.329e-15 & 5 \\ \hline
\end{tabular}
\captionof{table}{Fractional Newton$_2$ and Traub results for $f_1(x)$ with Riemann-Liouville derivative and initial estimation $x_0=-1.5$}\label{t8}
\end{center}
Our second function is $f_2(x)=ix^{1.8}-x^{0.9}-16$, with roots $\bar{x}_1=2.90807-4.24908i$ and $\bar{x}_2=-3.85126+1.74602i$. In tables \ref{t9}-\ref{t16} can be observed exactly the same behavior for $f_2(x)$ as in tables \ref{t1}-\ref{t8} for $f_1(x)$.
\begin{center}
\begin{tabular}{|c|c|c|c|c|c|c|c|c|}
\hline
 & \multicolumn{4}{c|}{CFN$_1$ method} & \multicolumn{4}{c|}{CFN$_2$ method} \\ \hline
$\alpha$ & $\bar{x}$ & $|x_{k+1}-x_k|$ & $|f(x_{k+1})|$ & iter & $\bar{x}$ & $|x_{k+1}-x_k|$ & $|f(x_{k+1})|$ & iter \\ \hline
0.6 & -3.8512+1.746i & 9.6447e-09 & 9.1533e-08 & 135 & -3.8518+1.7463i & 1.9397e-06 & 0.0040584 & 500 \\
0.65 & -3.8512+1.746i & 8.7093e-09 & 7.1238e-08 & 83 & -3.8513+1.746i & 4.5637e-07 & 7.7439e-04 & 500 \\
0.7 & -3.8512+1.746i & 8.0195e-09 & 5.5214e-08 & 57 & -3.8512+1.746i & 7.0051e-08 & 9.4467e-05 & 500 \\
0.75 & -3.8512+1.746i & 7.7567e-09 & 4.3576e-08 & 41 & -3.8512+1.746i & 9.934e-09 & 8.9604e-06 & 430 \\
0.8 & -3.8512+1.746i & 9.2798e-09 & 4.0719e-08 & 30 & -3.8512+1.746i & 9.9356e-09 & 3.2931e-06 & 207 \\
0.85 & -3.8512+1.746i & 5.0655e-09 & 1.623e-08 & 23 & -3.8512+1.746i & 9.9551e-09 & 1.1831e-06 & 100 \\
0.9 & -3.8512+1.746i & 4.2584e-09 & 8.8324e-09 & 17 & -3.8512+1.746i & 9.5415e-09 & 3.8313e-07 & 49 \\
0.95 & -3.8512+1.746i & 2.6355e-09 & 2.6468e-09 & 12 & -3.8512+1.746i & 7.6885e-09 & 8.4483e-08 & 23 \\
1 & -3.8512+1.746i & 5.3275e-06 & 1.5148e-11 & 4 & -3.8512+1.746i & 5.3275e-06 & 1.5148e-11 & 4 \\ \hline
\end{tabular}
\captionof{table}{Fractional Newton$_1$ and Newton$_2$ results for $f_2(x)$ with Caputo derivative and initial estimation $x_0=-4.5$}\label{t9}
\end{center}
\begin{center}
\begin{tabular}{|c|c|c|c|c|c|c|c|c|}
\hline
 & \multicolumn{4}{c|}{CFN$_1$ method} & \multicolumn{4}{c|}{CFN$_2$ method} \\ \hline
$\alpha$ & $\bar{x}$ & $|x_{k+1}-x_k|$ & $|f(x_{k+1})|$ & iter & $\bar{x}$ & $|x_{k+1}-x_k|$ & $|f(x_{k+1})|$ & iter \\ \hline
0.6 & -2.13e+06+6.44e+06i & 7.3797e+03 & 1.9819e+12 & 500 & -3.8518+1.7463i & 2.2298e-06 & 0.0044117 & 500 \\
0.65 & -2.66e+06+2.85e+06i & 1.111e+04 & 7.3325e+11 & 500 & -3.8513+1.746i & 5.3091e-07 & 8.5424e-04 & 500 \\
0.7 & -9.52e+05+8.15e+02i & 8.7223e+03 & 5.783e+10 & 500 & -3.8512+1.746i & 8.1441e-08 & 1.0495e-04 & 500 \\
0.75 & -33.9538-2.0387i & 8.805 & 5.7372e+02 & 500 & -3.8512+1.746i & 9.9227e-09 & 8.9528e-06 & 449 \\
0.8 & -3.8512+1.746i & 6.2175e-09 & 2.7282e-08 & 279 & -3.8512+1.746i & 9.8951e-09 & 3.2824e-06 & 226 \\
0.85 & -3.8512+1.746i & 6.3861e-09 & 2.0462e-08 & 148 & -3.8512+1.746i & 9.7471e-09 & 1.1622e-06 & 120 \\
0.9 & -3.8512+1.746i & 3.3987e-09 & 7.0493e-09 & 83 & -3.8512+1.746i & 9.6653e-09 & 3.8753e-07 & 68 \\
0.95 & -3.8512+1.746i & 3.3555e-09 & 3.3698e-09 & 46 & -3.8512+1.746i & 6.6872e-09 & 7.4289e-08 & 42 \\
1 & -3.8512+1.746i & 3.5728e-09 & 5.4025e-15 & 23 & -3.8512+1.746i & 3.5728e-09 & 5.4025e-15 & 23 \\ \hline
\end{tabular}
\captionof{table}{Fractional Newton$_1$ and Newton$_2$ results for $f_2(x)$ with Caputo derivative and initial estimation $x_0=-4.5+1e07i$}\label{t10}
\end{center}
\begin{center}
\begin{tabular}{|c|c|c|c|c|c|c|c|c|}
\hline
 & \multicolumn{4}{c|}{CFN$_1$ method} & \multicolumn{4}{c|}{CFN$_2$ method} \\ \hline
$\alpha$ & $\bar{x}$ & $|x_{k+1}-x_k|$ & $|f(x_{k+1})|$ & iter & $\bar{x}$ & $|x_{k+1}-x_k|$ & $|f(x_{k+1})|$ & iter \\ \hline
0.6 & -2.13e+06+6.44e+06i & 7.3797e+03 & 1.9819e+12 & 500 & -3.8518+1.7463i & 2.2298e-06 & 0.0044117 & 500 \\
0.65 & -2.66e+06+2.85e+06i & 1.111e+04 & 7.3325e+11 & 500 & -3.8513+1.746i & 5.3091e-07 & 8.5424e-04 & 500 \\
0.7 & -9.52e+05+8.16e+02i & 8.7223e+03 & 5.783e+10 & 500 & -3.8512+1.746i & 8.1441e-08 & 1.0495e-04 & 500 \\
0.75 & -33.9537-2.0384i & 8.805 & 5.7372e+02 & 500 & -3.8512+1.746i & 9.9227e-09 & 8.9528e-06 & 449 \\
0.8 & -3.8512+1.746i & 6.2175e-09 & 2.7282e-08 & 279 & -3.8512+1.746i & 9.8951e-09 & 3.2824e-06 & 226 \\
0.85 & -3.8512+1.746i & 6.3862e-09 & 2.0462e-08 & 148 & -3.8512+1.746i & 9.7471e-09 & 1.1622e-06 & 120 \\
0.9 & -3.8512+1.746i & 3.3987e-09 & 7.0493e-09 & 83 & -3.8512+1.746i & 9.6653e-09 & 3.8753e-07 & 68 \\
0.95 & -3.8512+1.746i & 3.3555e-09 & 3.3698e-09 & 46 & -3.8512+1.746i & 6.6872e-09 & 7.4289e-08 & 42 \\
1 & -3.8512+1.746i & 3.5728e-09 & 3.662e-15 & 23 & -3.8512+1.746i & 3.5728e-09 & 3.662e-15 & 23 \\ \hline
\end{tabular}
\captionof{table}{Fractional Newton$_1$ and Newton$_2$ results for $f_2(x)$ with Caputo derivative and initial estimation $x_0=1e07i$}\label{t11}
\end{center}
\begin{center}
\begin{tabular}{|c|c|c|c|c|c|c|c|c|}
\hline
 & \multicolumn{4}{c|}{R-LFN$_1$ method} & \multicolumn{4}{c|}{R-LFN$_2$ method} \\ \hline
$\alpha$ & $\bar{x}$ & $|x_{k+1}-x_k|$ & $|f(x_{k+1})|$ & iter & $\bar{x}$ & $|x_{k+1}-x_k|$ & $|f(x_{k+1})|$ & iter \\ \hline
0.6 & -3.8512+1.746i & 9.8908e-09 & 7.0198e-08 & 321 & -3.8514+1.7461i & 7.6079e-07 & 0.001593 & 500 \\
0.65 & -3.8512+1.746i & 9.5716e-09 & 6.0692e-08 & 105 & -3.8513+1.746i & 1.8874e-07 & 3.2037e-04 & 500 \\
0.7 & -3.8512+1.746i & 8.3392e-09 & 4.5994e-08 & 61 & -3.8512+1.746i & 2.9902e-08 & 4.0532e-05 & 500 \\
0.75 & -3.8512+1.746i & 6.5731e-09 & 3.048e-08 & 43 & 2.908-4.249i & 9.9785e-09 & 7.4479e-06 & 367 \\
0.8 & -3.8512+1.746i & 7.592e-09 & 2.8258e-08 & 30 & 2.908-4.249i & 9.8775e-09 & 2.8264e-06 & 192 \\
0.85 & -3.8512+1.746i & 8.3034e-09 & 2.3137e-08 & 22 & -3.8512+1.746i & 9.4686e-09 & 1.0137e-06 & 93 \\
0.9 & -3.8512+1.746i & 9.4241e-09 & 1.7387e-08 & 16 & -3.8512+1.746i & 9.9072e-09 & 3.6628e-07 & 45 \\
0.95 & -3.8512+1.746i & 1.649e-09 & 1.5036e-09 & 12 & -3.8512+1.746i & 7.7314e-09 & 8.0728e-08 & 22 \\
1 & -3.8512+1.746i & 5.3275e-06 & 1.5148e-11 & 4 & -3.8512+1.746i & 5.3275e-06 & 1.5148e-11 & 4 \\ \hline
\end{tabular}
\captionof{table}{Fractional Newton$_1$ and Newton$_2$ results for $f_2(x)$ with Riemann-Liouville derivative and initial estimation $x_0=-4.5$}\label{t12}
\end{center}
\begin{center}
\begin{tabular}{|c|c|c|c|c|c|c|c|c|}
\hline
 & \multicolumn{4}{c|}{R-LFN$_1$ method} & \multicolumn{4}{c|}{R-LFN$_2$ method} \\ \hline
$\alpha$ & $\bar{x}$ & $|x_{k+1}-x_k|$ & $|f(x_{k+1})|$ & iter & $\bar{x}$ & $|x_{k+1}-x_k|$ & $|f(x_{k+1})|$ & iter \\ \hline
0.6 & -2.13e+06+6.44e+06i & 7.3797e+03 & 1.9819e+12 & 500 & -3.8515+1.7461i & 8.7906e-07 & 0.001737 & 500 \\
0.65 & -2.66e+06+2.85e+06i & 1.111e+04 & 7.3325e+11 & 500 & -3.8513+1.746i & 2.2124e-07 & 3.5516e-04 & 500 \\
0.7 & -9.52e+05+8.15e+02i & 8.7223e+03 & 5.783e+10 & 500 & -3.8512+1.746i & 3.5062e-08 & 4.53e-05 & 500 \\
0.75 & -33.9404-2.18i & 8.8264 & 5.7359e+02 & 500 & -3.8512+1.746i & 9.9505e-09 & 7.3526e-06 & 371 \\
0.8 & -3.8512+1.746i & 6.9815e-09 & 2.5986e-08 & 277 & -3.8512+1.746i & 9.8631e-09 & 2.8101e-06 & 197 \\
0.85 & -3.8512+1.746i & 7.3824e-09 & 2.057e-08 & 147 & -3.8512+1.746i & 9.4884e-09 & 1.0155e-06 & 110 \\
0.9 & -3.8512+1.746i & 4.3678e-09 & 8.0585e-09 & 82 & -3.8512+1.746i & 9.0672e-09 & 3.3871e-07 & 65 \\
0.95 & -3.8512+1.746i & 2.0381e-09 & 1.8585e-09 & 46 & -3.8512+1.746i & 7.0456e-09 & 7.4115e-08 & 41 \\
1 & -3.8512+1.746i & 3.5728e-09 & 5.4025e-15 & 23 & -3.8512+1.746i & 3.5728e-09 & 5.4025e-15 & 23 \\ \hline
\end{tabular}
\captionof{table}{Fractional Newton$_1$ and Newton$_2$ results for $f_2(x)$ with Riemann-Liouville derivative and initial estimation $x_0=-4.5+1e07i$}\label{t13}
\end{center}
\begin{center}
\begin{tabular}{|c|c|c|c|c|c|c|c|c|}
\hline
 & \multicolumn{4}{c|}{R-LFN$_1$ method} & \multicolumn{4}{c|}{R-LFN$_2$ method} \\ \hline
$\alpha$ & $\bar{x}$ & $|x_{k+1}-x_k|$ & $|f(x_{k+1})|$ & iter & $\bar{x}$ & $|x_{k+1}-x_k|$ & $|f(x_{k+1})|$ & iter \\ \hline
0.6 & -2.13e+06+6.44e+06i & 7.3797e+03 & 1.9819e+12 & 500 & -3.8515+1.7461i & 8.7906e-07 & 0.001737 & 500 \\
0.65 & -2.66e+06+2.85e+06i & 1.111e+04 & 7.3325e+11 & 500 & -3.8513+1.746i & 2.2124e-07 & 3.5516e-04 & 500 \\
0.7 & -9.52e+05+8.16e+02i & 8.7223e+03 & 5.783e+10 & 500 & -3.8512+1.746i & 3.5062e-08 & 4.53e-05 & 500 \\
0.75 & -33.9403-2.1796i & 8.8264 & 5.7358e+02 & 500 & -3.8512+1.746i & 9.9505e-09 & 7.3526e-06 & 371 \\
0.8 & -3.8512+1.746i & 6.9817e-09 & 2.5987e-08 & 277 & -3.8512+1.746i & 9.8631e-09 & 2.8101e-06 & 197 \\
0.85 & -3.8512+1.746i & 7.3825e-09 & 2.0571e-08 & 147 & -3.8512+1.746i & 9.4884e-09 & 1.0155e-06 & 110 \\
0.9 & -3.8512+1.746i & 4.3678e-09 & 8.0585e-09 & 82 & -3.8512+1.746i & 9.0672e-09 & 3.3871e-07 & 65 \\
0.95 & -3.8512+1.746i & 2.0381e-09 & 1.8584e-09 & 46 & -3.8512+1.746i & 7.0456e-09 & 7.4115e-08 & 41 \\
1 & -3.8512+1.746i & 3.5728e-09 & 3.662e-15 & 23 & -3.8512+1.746i & 3.5728e-09 & 3.662e-15 & 23 \\ \hline
\end{tabular}
\captionof{table}{Fractional Newton$_1$ and Newton$_2$ results for $f_2(x)$ with Riemann-Liouville derivative and initial estimation $x_0=1e07i$}\label{t14}
\end{center}
\begin{center}
\begin{tabular}{|c|c|c|c|c|c|c|c|c|}
\hline
 & \multicolumn{4}{c|}{CFN$_2$ method} & \multicolumn{4}{c|}{CFT method} \\ \hline
$\alpha$ & $\bar{x}$ & $|x_{k+1}-x_k|$ & $|f(x_{k+1})|$ & iter & $\bar{x}$ & $|x_{k+1}-x_k|$ & $|f(x_{k+1})|$ & iter \\ \hline
0.6 & -3.8518+1.7463i & 1.9397e-06 & 0.0040584 & 500 & -3.8514+1.7461i & 6.9328e-07 & 0.0014451 & 500 \\
0.65 & -3.8513+1.746i & 4.5637e-07 & 7.7439e-04 & 500 & -3.8512+1.746i & 1.2842e-07 & 2.1663e-04 & 500 \\
0.7 & -3.8512+1.746i & 7.0051e-08 & 9.4467e-05 & 500 & -3.8512+1.746i & 1.577e-08 & 2.0497e-05 & 500 \\
0.75 & -3.8512+1.746i & 9.934e-09 & 8.9604e-06 & 430 & 2.908-4.249i & 9.859e-09 & 5.2033e-06 & 257 \\
0.8 & -3.8512+1.746i & 9.9356e-09 & 3.2931e-06 & 207 & -3.8512+1.746i & 9.6449e-09 & 1.8367e-06 & 120 \\
0.85 & -3.8512+1.746i & 9.9551e-09 & 1.1831e-06 & 100 & -3.8512+1.746i & 9.1797e-09 & 6.0224e-07 & 57 \\
0.9 & -3.8512+1.746i & 9.5415e-09 & 3.8313e-07 & 49 & -3.8512+1.746i & 9.1356e-09 & 1.8638e-07 & 27 \\
0.95 & -3.8512+1.746i & 7.6885e-09 & 8.4483e-08 & 23 & -3.8512+1.746i & 5.5216e-09 & 2.6074e-08 & 13 \\
1 & -3.8512+1.746i & 5.3275e-06 & 1.5148e-11 & 4 & -3.8512+1.746i & 1.1681e-05 & 3.5527e-15 & 3 \\ \hline
\end{tabular}
\captionof{table}{Fractional Newton$_2$ and Traub results for $f_2(x)$ with Caputo derivative and initial estimation $x_0=-4.5$}\label{t15}
\end{center}
\begin{center}
\begin{tabular}{|c|c|c|c|c|c|c|c|c|}
\hline
 & \multicolumn{4}{c|}{R-LFN$_2$ method} & \multicolumn{4}{c|}{R-LFT method} \\ \hline
$\alpha$ & $\bar{x}$ & $|x_{k+1}-x_k|$ & $|f(x_{k+1})|$ & iter & $\bar{x}$ & $|x_{k+1}-x_k|$ & $|f(x_{k+1})|$ & iter \\ \hline
0.6 & -3.8514+1.7461i & 7.6079e-07 & 0.001593 & 500 & -3.8513+1.746i & 2.7288e-07 & 5.6848e-04 & 500 \\
0.65 & -3.8513+1.746i & 1.8874e-07 & 3.2037e-04 & 500 & -3.8512+1.746i & 5.3251e-08 & 8.9778e-05 & 500 \\
0.7 & -3.8512+1.746i & 2.9902e-08 & 4.0532e-05 & 500 & -3.8512+1.746i & 9.9711e-09 & 1.1573e-05 & 435 \\
0.75 & 2.908-4.249i & 9.9785e-09 & 7.4479e-06 & 367 & -3.8512+1.746i & 9.9117e-09 & 4.3496e-06 & 215 \\
0.8 & 2.908-4.249i & 9.8775e-09 & 2.8264e-06 & 192 & -3.8512+1.746i & 9.5765e-09 & 1.5662e-06 & 105 \\
0.85 & -3.8512+1.746i & 9.4686e-09 & 1.0137e-06 & 93 & -3.8512+1.746i & 9.7086e-09 & 5.6297e-07 & 52 \\
0.9 & -3.8512+1.746i & 9.9072e-09 & 3.6628e-07 & 45 & -3.8512+1.746i & 8.1865e-09 & 1.5588e-07 & 25 \\
0.95 & -3.8512+1.746i & 7.7314e-09 & 8.0728e-08 & 22 & -3.8512+1.746i & 8.0857e-09 & 3.4712e-08 & 12 \\
1 & -3.8512+1.746i & 5.3275e-06 & 1.5148e-11 & 4 & -3.8512+1.746i & 1.1681e-05 & 3.5527e-15 & 3 \\ \hline
\end{tabular}
\captionof{table}{Fractional Newton$_2$ and Traub results for $f_2(x)$ with Riemann-Liouville derivative and initial estimation $x_0=-4.5$}\label{t16}
\end{center}
Our third function is $f_3(x)=e^x-1$ with only real root $\bar{x}=0$. It is necessary to use a value of $\alpha$ close to 1 to ensure the convergence. In this case have not been used imaginary values for the initial estimations due to erratic behavior of results, i.e., not necessarily CFN$_2$ and R-LFN$_2$ will have better results than CFN$_1$ and R-LFN$_1$ respectively by increasing the absolute value of imaginary part of initial estimation. In table \ref{t17} we can see that CFN$_2$ requires much less iterations than CFN$_1$, while in table \ref{t18} the behavior is almost the same.
\begin{center}
\begin{tabular}{|c|c|c|c|c|c|c|c|c|}
\hline
 & \multicolumn{4}{c|}{CFN$_1$ method} & \multicolumn{4}{c|}{CFN$_2$ method} \\ \hline
$\alpha$ & $\bar{x}$ & $|x_{k+1}-x_k|$ & $|f(x_{k+1})|$ & iter & $\bar{x}$ & $|x_{k+1}-x_k|$ & $|f(x_{k+1})|$ & iter \\ \hline
0.9 & -8.1428e-05+4.6123e-04i & 9.2247e-04 & 4.6835e-04 & 500 & 2.6106e-09 & 2.5158e-08 & 2.6106e-09 & 8 \\
0.91 & -3.2586e-05+2.0793e-04i & 4.1587e-04 & 2.1047e-04 & 500 & 1.2569e-09 & 1.351e-08 & 1.2569e-09 & 8 \\
0.92 & -1.0636e-05+7.7362e-05i & 1.5472e-04 & 7.8089e-05 & 500 & 7.2879e-09 & 8.8464e-08 & 7.2879e-09 & 7 \\
0.93 & -2.598e-06+2.187e-05i & 4.374e-05 & 2.2023e-05 & 500 & 3.2741e-09 & 4.5591e-08 & 3.2741e-09 & 7 \\
0.94 & -4.1185e-07+4.0939e-06i & 8.1878e-06 & 4.1145e-06 & 500 & 1.3017e-09 & 2.1227e-08 & 1.3017e-09 & 7 \\
0.95 & -3.2825e-08+3.9614e-07i & 7.9228e-07 & 3.975e-07 & 500 & 9.0542e-09 & 1.7783e-07 & 9.0542e-09 & 6 \\
0.96 & -7.9157e-10+1.2077e-08i & 2.4154e-08 & 1.2102e-08 & 500 & 2.9901e-09 & 7.3682e-08 & 2.9901e-09 & 6 \\
0.97 & 1.9341e-09-8.4739e-09i & 2.1421e-08 & 8.6918e-09 & 17 & 7.2773e-10 & 2.3997e-08 & 7.2773e-10 & 6 \\
0.98 & 3.0856e-09-5.6805e-09i & 2.2259e-08 & 6.4645e-09 & 11 & 5.2033e-09 & 2.5831e-07 & 5.2033e-09 & 5 \\
0.99 & 7.0684e-09-6.7644e-09i & 6.6485e-08 & 9.7837e-09 & 7 & 4.0463e-10 & 4.0319e-08 & 4.0463e-10 & 5 \\
1 & 6.5401e-17 & 1.5194e-08 & 0 & 4 & 6.5401e-17 & 1.5194e-08 & 0 & 4 \\ \hline
\end{tabular}
\captionof{table}{Fractional Newton$_1$ and Newton$_2$ results for $f_3(x)$ with Caputo derivative and initial estimation $x_0=0.2$}\label{t17}
\end{center}
\begin{center}
\begin{tabular}{|c|c|c|c|c|c|c|c|c|}
\hline
 & \multicolumn{4}{c|}{R-LFN$_1$ method} & \multicolumn{4}{c|}{R-LFN$_2$ method} \\ \hline
$\alpha$ & $\bar{x}$ & $|x_{k+1}-x_k|$ & $|f(x_{k+1})|$ & iter & $\bar{x}$ & $|x_{k+1}-x_k|$ & $|f(x_{k+1})|$ & iter \\ \hline
0.9 & -1.01e-04+4.84e-08i & 2.3837e-07 & 1.0176e-04 & 500 & -3.72e-04+1.23e-19i & 6.7672e-07 & 3.7289e-04 & 500 \\
0.91 & -9.71e-05+3.77e-08i & 2.2269e-07 & 9.7116e-05 & 500 & -3.15e-04+1.32e-22i & 5.7799e-07 & 3.1534e-04 & 500 \\
0.92 & -9.16e-05+3.03e-08i & 2.0596e-07 & 9.1637e-05 & 500 & -2.63e-04+7.94e-23i & 4.8744e-07 & 2.6335e-04 & 500 \\
0.93 & -8.5e-05+2.52e-08i & 1.8797e-07 & 8.5216e-05 & 500 & -2.16e-04+7.45e-20i & 4.0451e-07 & 2.1646e-04 & 500 \\
0.94 & -7.77e-05+2.11e-08i & 1.6846e-07 & 7.7728e-05 & 500 & -1.74e-04-1.32e-23i & 3.2865e-07 & 1.7423e-04 & 500 \\
0.95 & -6.90e-05+1.74e-08i & 1.4717e-07 & 6.9027e-05 & 500 & -1.36e-04-1.98e-23i & 2.5931e-07 & 1.3628e-04 & 500 \\
0.96 & -5.89e-05+1.34e-08i & 1.2373e-07 & 5.8939e-05 & 500 & -1.02e-04+9.92e-24i & 1.9576e-07 & 1.0215e-04 & 500 \\
0.97 & -4.72e-05+9.08e-09i & 9.768e-08 & 4.7242e-05 & 500 & -7.03e-05+2.42e-20i & 1.33e-07 & 7.036e-05 & 500 \\
0.98 & -3.24e-05+3.98e-09i & 6.3507e-08 & 3.2418e-05 & 500 & -4.57e-05-1.75e-09i & 9.0754e-08 & 4.5746e-05 & 500 \\
0.99 & -1.84e-05-8.37e-10i & 3.7563e-08 & 1.8423e-05 & 500 & -2.13e-05+7.35e-21i & 4.2567e-08 & 2.1369e-05 & 500 \\
1 & 6.5401e-17 & 1.5194e-08 & 0 & 4 & 6.5401e-17 & 1.5194e-08 & 0 & 4 \\ \hline
\end{tabular}
\captionof{table}{Fractional Newton$_1$ and Newton$_2$ results for $f_3(x)$ with Riemann-Liouville derivative and initial estimation $x_0=0.2$}\label{t18}
\end{center}
Let us now compare Traub method with its first step for $f_3(x)$. In both tables \ref{t19} and \ref{t20} can be observed that Taub method requires less iterations than its first step.
\begin{center}
\begin{tabular}{|c|c|c|c|c|c|c|c|c|}
\hline
 & \multicolumn{4}{c|}{CFN$_2$ method} & \multicolumn{4}{c|}{CFT method} \\ \hline
$\alpha$ & $\bar{x}$ & $|x_{k+1}-x_k|$ & $|f(x_{k+1})|$ & iter & $\bar{x}$ & $|x_{k+1}-x_k|$ & $|f(x_{k+1})|$ & iter \\ \hline
0.9 & 2.6106e-09 & 2.5158e-08 & 2.6106e-09 & 8 & 9.2854e-09 & 3.1633e-07 & 9.2854e-09 & 5 \\
0.91 & 1.2569e-09 & 1.351e-08 & 1.2569e-09 & 8 & 4.3293e-09 & 1.7542e-07 & 4.3293e-09 & 5 \\
0.92 & 7.2879e-09 & 8.8464e-08 & 7.2879e-09 & 7 & 1.8388e-09 & 9.0552e-08 & 1.8388e-09 & 5 \\
0.93 & 3.2741e-09 & 4.5591e-08 & 3.2741e-09 & 7 & 6.9418e-10 & 4.2699e-08 & 6.9418e-10 & 5 \\
0.94 & 1.3017e-09 & 2.1227e-08 & 1.3017e-09 & 7 & 2.2469e-10 & 1.79e-08 & 2.2469e-10 & 5 \\
0.95 & 9.0542e-09 & 1.7783e-07 & 9.0542e-09 & 6 & 6.4561e-09 & 7.0013e-07 & 6.4561e-09 & 4 \\
0.96 & 2.9901e-09 & 7.3682e-08 & 2.9901e-09 & 6 & 1.8294e-09 & 2.9042e-07 & 1.8294e-09 & 4 \\
0.97 & 7.2773e-10 & 2.3997e-08 & 7.2773e-10 & 6 & 3.6228e-10 & 9.4541e-08 & 3.6228e-10 & 4 \\
0.98 & 5.2033e-09 & 2.5831e-07 & 5.2033e-09 & 5 & 3.7599e-11 & 1.9964e-08 & 3.7599e-11 & 4 \\
0.99 & 4.0463e-10 & 4.0319e-08 & 4.0463e-10 & 5 & 1.5139e-09 & 2.7685e-06 & 1.5139e-09 & 3 \\
1 & 6.5401e-17 & 1.5194e-08 & 0 & 4 & 1.3111e-17 & 1.7116e-08 & 0 & 3 \\ \hline
\end{tabular}
\captionof{table}{Fractional Newton$_2$ and Traub results for $f_3(x)$ with Caputo derivative and initial estimation $x_0=0.2$}\label{t19}
\end{center}
\begin{center}
\begin{tabular}{|c|c|c|c|c|c|c|c|c|}
\hline
 & \multicolumn{4}{c|}{R-LFN$_2$ method} & \multicolumn{4}{c|}{R-LFT method} \\ \hline
$\alpha$ & $\bar{x}$ & $|x_{k+1}-x_k|$ & $|f(x_{k+1})|$ & iter & $\bar{x}$ & $|x_{k+1}-x_k|$ & $|f(x_{k+1})|$ & iter \\ \hline
0.9 & -3.72e-04+1.23e-19i & 6.7672e-07 & 3.7289e-04 & 500 & -1.99e-04+3.82e-07i & 3.6114e-07 & 1.993e-04 & 500 \\
0.91 & -3.15e-04+1.32e-22i & 5.7799e-07 & 3.1534e-04 & 500 & -1.67e-04+5.22e-07i & 3.0667e-07 & 1.6748e-04 & 500 \\
0.92 & -2.63e-04+7.94e-23i & 4.8744e-07 & 2.6335e-04 & 500 & -1.38e-04+6.6e-07i & 2.5688e-07 & 1.3892e-04 & 500 \\
0.93 & -2.16e-04+7.45e-20i & 4.0451e-07 & 2.1646e-04 & 500 & -1.13e-04+7.52e-07i & 2.1135e-07 & 1.133e-04 & 500 \\
0.94 & -1.74e-04-1.32e-23i & 3.2865e-07 & 1.7423e-04 & 500 & -9.04e-05+7.47e-07i & 1.6986e-07 & 9.0407e-05 & 500 \\
0.95 & -1.36e-04-1.98e-23i & 2.5931e-07 & 1.3628e-04 & 500 & -7e-05+6.52e-07i & 1.3232e-07 & 7.0024e-05 & 500 \\
0.96 & -1.02e-04+9.92e-24i & 1.9576e-07 & 1.0215e-04 & 500 & -5.19e-05+5.46e-07i & 9.8323e-08 & 5.1901e-05 & 500 \\
0.97 & -7.03e-05+2.42e-20i & 1.33e-07 & 7.036e-05 & 500 & -3.46e-05+7.8e-07i & 6.319e-08 & 3.466e-05 & 500 \\
0.98 & -4.57e-05-1.75e-09i & 9.0754e-08 & 4.5746e-05 & 500 & -2.37e-05-1.66e-09i & 4.8178e-08 & 2.3718e-05 & 500 \\
0.99 & -2.13e-05+7.35e-21i & 4.2567e-08 & 2.1369e-05 & 500 & -1.11e-05+4.41e-08i & 2.2946e-08 & 1.1127e-05 & 500 \\
1 & 6.5401e-17 & 1.5194e-08 & 0 & 4 & 1.3111e-17 & 1.7116e-08 & 0 & 3 \\ \hline
\end{tabular}
\captionof{table}{Fractional Newton$_2$ and Traub results for $f_3(x)$ with Riemann-Liouville derivative and initial estimation $x_0=0.2$}\label{t20}
\end{center}
Our fourth function is $f_4(x)=\sin10x-0.5x+0.2$ with real roots $\bar{x}_1=-1.4523$, $\bar{x}_2=-1.3647$, $\bar{x}_3=-0.87345$, $\bar{x}_4=-0.6857$, $\bar{x}_5=-0.27949$, $\bar{x}_6=-0.021219$, $\bar{x}_7=0.31824$, $\bar{x}_8=0.64036$, $\bar{x}_9=0.91636$, $\bar{x}_{10}=1.3035$, $\bar{x}_{11}=1.5118$, $\bar{x}_{12}=1.9756$ and $\bar{x}_{13}=2.0977$.
\begin{center}
\begin{tabular}{|c|c|c|c|c|c|c|c|c|}
\hline
 & \multicolumn{4}{c|}{CFN$_1$ method} & \multicolumn{4}{c|}{CFN$_2$ method} \\ \hline
$\alpha$ & $\bar{x}$ & $|x_{k+1}-x_k|$ & $|f(x_{k+1})|$ & iter & $\bar{x}$ & $|x_{k+1}-x_k|$ & $|f(x_{k+1})|$ & iter \\ \hline
0.6 & 1.9756 & 8.408e-09 & 1.8613e-08 & 36 & 1.9757-5.4314e-07i & 1.4339e-07 & 2.6898e-04 & 500 \\
0.65 & 1.9756 & 9.1296e-09 & 1.6527e-08 & 24 & 2.0977 & 9.8912e-09 & 5.2792e-06 & 169 \\
0.7 & 1.3035 & 8.4237e-09 & 3.0203e-08 & 63 & 2.0977 & 9.9073e-09 & 3.0502e-06 & 121 \\
0.75 & 1.9756 & 4.0708e-09 & 4.2781e-09 & 10 & 1.9756+7.6535e-08i & 9.868e-09 & 4.5063e-06 & 235 \\
0.8 & 1.3035 & 9.976e-09 & 2.3499e-08 & 22 & 1.5118+7.3334e-08i & 9.5351e-09 & 1.7694e-06 & 76 \\
0.85 & 1.3035 & 4.7561e-09 & 8.2545e-09 & 18 & 0.91636 & 9.1437e-09 & 9.5832e-07 & 51 \\
0.9 & -0.6857+1.3551e-09i & 4.3677e-09 & 1.0885e-08 & 18 & -1.3647-1.2595e-09i & 9.5876e-09 & 2.2755e-07 & 42 \\
0.95 & 3.8845 & 0.1703 & 0.83083 & 500 & 20.89+0.30176i & 9.9619e-09 & 1.3042e-06 & 81 \\
1 & 4.3892 & 0.22037 & 2.0846 & 500 & 4.3892 & 0.22037 & 2.0846 & 500 \\ \hline
\end{tabular}
\captionof{table}{Fractional Newton$_1$ and Newton$_2$ results for $f_4(x)$ with Caputo derivative and initial estimation $x_0=3$}\label{t1}
\end{center}
Tests with imaginary initial estimations have not been included because there is no convergence for these cases with $f_4(x)$.
\begin{center}
\begin{tabular}{|c|c|c|c|c|c|c|c|c|}
\hline
 & \multicolumn{4}{c|}{R-LFN$_1$ method} & \multicolumn{4}{c|}{R-LFN$_2$ method} \\ \hline
$\alpha$ & $\bar{x}$ & $|x_{k+1}-x_k|$ & $|f(x_{k+1})|$ & iter & $\bar{x}$ & $|x_{k+1}-x_k|$ & $|f(x_{k+1})|$ & iter \\ \hline
0.6 & 1.9756 & 7.7242e-09 & 1.6581e-08 & 33 & 1.9757 & 1.4235e-07 & 2.7305e-04 & 500 \\
0.65 & 1.9756 & 7.3338e-09 & 1.2864e-08 & 23 & 2.0977 & 9.8532e-09 & 4.9292e-06 & 158 \\
0.7 & -52.34-1.8753i & 0.18129 & 6.9714e+07 & 500 & 2.0977 & 9.9595e-09 & 2.9517e-06 & 116 \\
0.75 & 1.9756 & 3.7455e-09 & 3.8012e-09 & 12 & 1.9756+4.5877e-08i & 9.8631e-09 & 4.5403e-06 & 238 \\
0.8 & -0.8821-0.015482i & 0.030964 & 0.15115 & 500 & 1.5118 & 9.9545e-09 & 1.817e-06 & 77 \\
0.85 & 1.5118 & 6.0787e-09 & 1.913e-08 & 32 & 0.91636 & 9.2941e-09 & 9.6601e-07 & 51 \\
0.9 & -0.6857-6.4093e-10i & 4.9751e-09 & 1.243e-08 & 15 & -1.3647+1.3711e-08i & 9.0915e-09 & 2.162e-07 & 42 \\
0.95 & 5.6793 & 0.28537 & 2.397 & 500 & 22.146+0.30774i & 6.5531e-09 & 9.4542e-07 & 51 \\
1 & 4.3892 & 0.22037 & 2.0846 & 500 & 4.3892 & 0.22037 & 2.0846 & 500 \\ \hline
\end{tabular}
\captionof{table}{Fractional Newton$_1$ and Newton$_2$ results for $f_4(x)$ with Riemann-Liouville derivative and initial estimation $x_0=3$}\label{t2}
\end{center}
\begin{center}
\begin{tabular}{|c|c|c|c|c|c|c|c|c|}
\hline
 & \multicolumn{4}{c|}{CFN$_2$ method} & \multicolumn{4}{c|}{CFT method} \\ \hline
$\alpha$ & $\bar{x}$ & $|x_{k+1}-x_k|$ & $|f(x_{k+1})|$ & iter & $\bar{x}$ & $|x_{k+1}-x_k|$ & $|f(x_{k+1})|$ & iter \\ \hline
0.6 & 1.9757-5.4314e-07i & 1.4339e-07 & 2.6898e-04 & 500 & 1.9757+4.4325e-08i & 5.3605e-08 & 9.8416e-05 & 500 \\
0.65 & 2.0977 & 9.8912e-09 & 5.2792e-06 & 169 & NaN+NaNi & NaN & NaN & 3 \\
0.7 & 2.0977 & 9.9073e-09 & 3.0502e-06 & 121 & -0.6857+6.042e-09i & 9.9241e-09 & 7.4684e-06 & 224 \\
0.75 & 1.9756+7.6535e-08i & 9.868e-09 & 4.5063e-06 & 235 & 1.9756+2.9642e-08i & 9.7623e-09 & 2.6495e-06 & 142 \\
0.8 & 1.5118+7.3334e-08i & 9.5351e-09 & 1.7694e-06 & 76 & 0.91636 & 9.7441e-09 & 1.3325e-06 & 52 \\
0.85 & 0.91636 & 9.1437e-09 & 9.5832e-07 & 51 & 0.31824 & 8.3509e-09 & 5.3216e-07 & 32 \\
0.9 & -1.3647-1.2595e-09i & 9.5876e-09 & 2.2755e-07 & 42 & 0.91636-1.7656e-09i & 8.5295e-09 & 1.7466e-07 & 23 \\
0.95 & 20.89+0.30176i & 9.9619e-09 & 1.3042e-06 & 81 & NaN+NaNi & NaN & NaN & 2 \\
1 & 4.3892 & 0.22037 & 2.0846 & 500 & 4.0815 & 0.25782 & 1.8154 & 500 \\ \hline
\end{tabular}
\captionof{table}{Fractional Newton$_2$ and Traub results for $f_4(x)$ with Caputo derivative and initial estimation $x_0=3$}\label{t3}
\end{center}
\begin{center}
\begin{tabular}{|c|c|c|c|c|c|c|c|c|}
\hline
 & \multicolumn{4}{c|}{R-LFN$_2$ method} & \multicolumn{4}{c|}{R-LFT method} \\ \hline
$\alpha$ & $\bar{x}$ & $|x_{k+1}-x_k|$ & $|f(x_{k+1})|$ & iter & $\bar{x}$ & $|x_{k+1}-x_k|$ & $|f(x_{k+1})|$ & iter \\ \hline
0.6 & 1.9757 & 1.4235e-07 & 2.7305e-04 & 500 & NaN+NaNi & NaN & NaN & 3 \\
0.65 & 2.0977 & 9.8532e-09 & 4.9292e-06 & 158 & 2.6655+0.051224i & 9.9621e-09 & 1.2812e-05 & 450 \\
0.7 & 2.0977 & 9.9595e-09 & 2.9517e-06 & 116 & 2.0977+2.0416e-09i & 9.8235e-09 & 1.7922e-06 & 74 \\
0.75 & 1.9756+4.5877e-08i & 9.8631e-09 & 4.5403e-06 & 238 & 1.9756+1.5019e-08i & 9.8996e-09 & 2.6985e-06 & 143 \\
0.8 & 1.5118 & 9.9545e-09 & 1.817e-06 & 77 & 0.91636 & 9.5075e-09 & 1.2959e-06 & 52 \\
0.85 & 0.91636 & 9.2941e-09 & 9.6601e-07 & 51 & -0.87345+1.6392e-10i & 8.282e-09 & 3.3276e-07 & 31 \\
0.9 & -1.3647+1.3711e-08i & 9.0915e-09 & 2.162e-07 & 42 & -0.021219+2.2677e-09i & 6.3651e-09 & 9.7585e-08 & 18 \\
0.95 & 22.146+0.30774i & 6.5531e-09 & 9.4542e-07 & 51 & NaN+NaNi & NaN & NaN & 2 \\
1 & 4.3892 & 0.22037 & 2.0846 & 500 & 4.0815 & 0.25782 & 1.8154 & 500 \\ \hline
\end{tabular}
\captionof{table}{Fractional Newton$_2$ and Traub results for $f_4(x)$ with Riemann-Liouville derivative and initial estimation $x_0=3$}\label{t4}
\end{center}
We can see that the number of iterations does not necessarily reduce when $\alpha$ increases and the methods converges to multiple roots.

\subsection{Convergence planes}

In this subsection we are going to analyze the dependence on the initial estimation of the Newton and Traub methods by using convergence planes defined in \cite{AAM} and used in \cite{AJR} for the same purposes as in this paper. \\
Let us regard $f_1(x)$. In figures \ref{f1}, \ref{f2}, \ref{f3} and \ref{f4} we can see that CFN$_1$ and R-LFN$_1$ have a higher percentage of convergence than CFN$_2$ and R-LFN$_2$ respectively, not only with real or imaginary initial estimations, but also with Caputo or Riemann-Liouville derivative. \\
\begin{minipage}[c]{0.5\textwidth}
\begin{center}
\includegraphics[width=\textwidth]{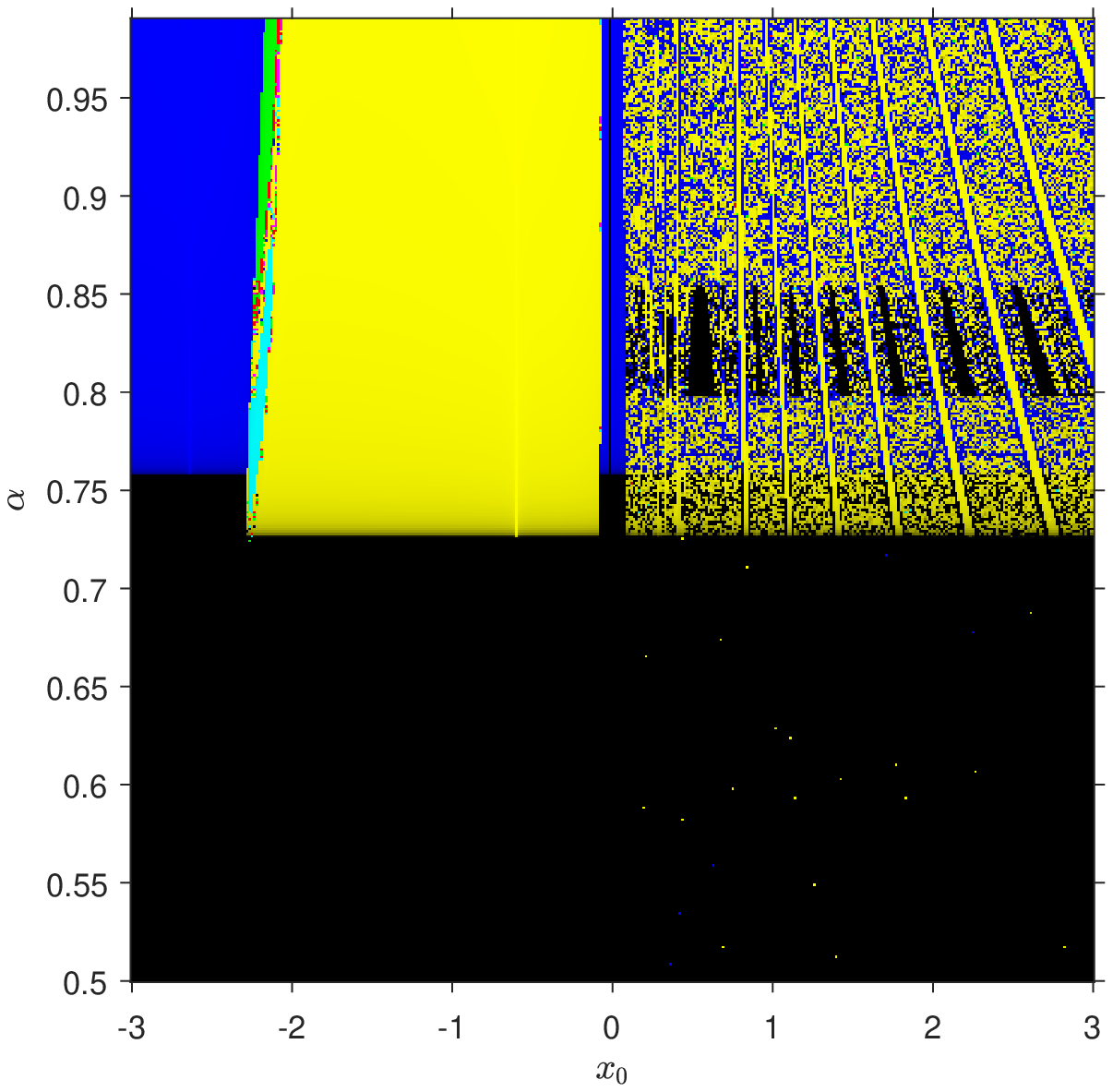}
(a) CFN$_1$, $-3\leq x_0\leq3$, 49.54\% convergence
\end{center}
\end{minipage}
\begin{minipage}[c]{0.5\textwidth}
\begin{center}
\includegraphics[width=\textwidth]{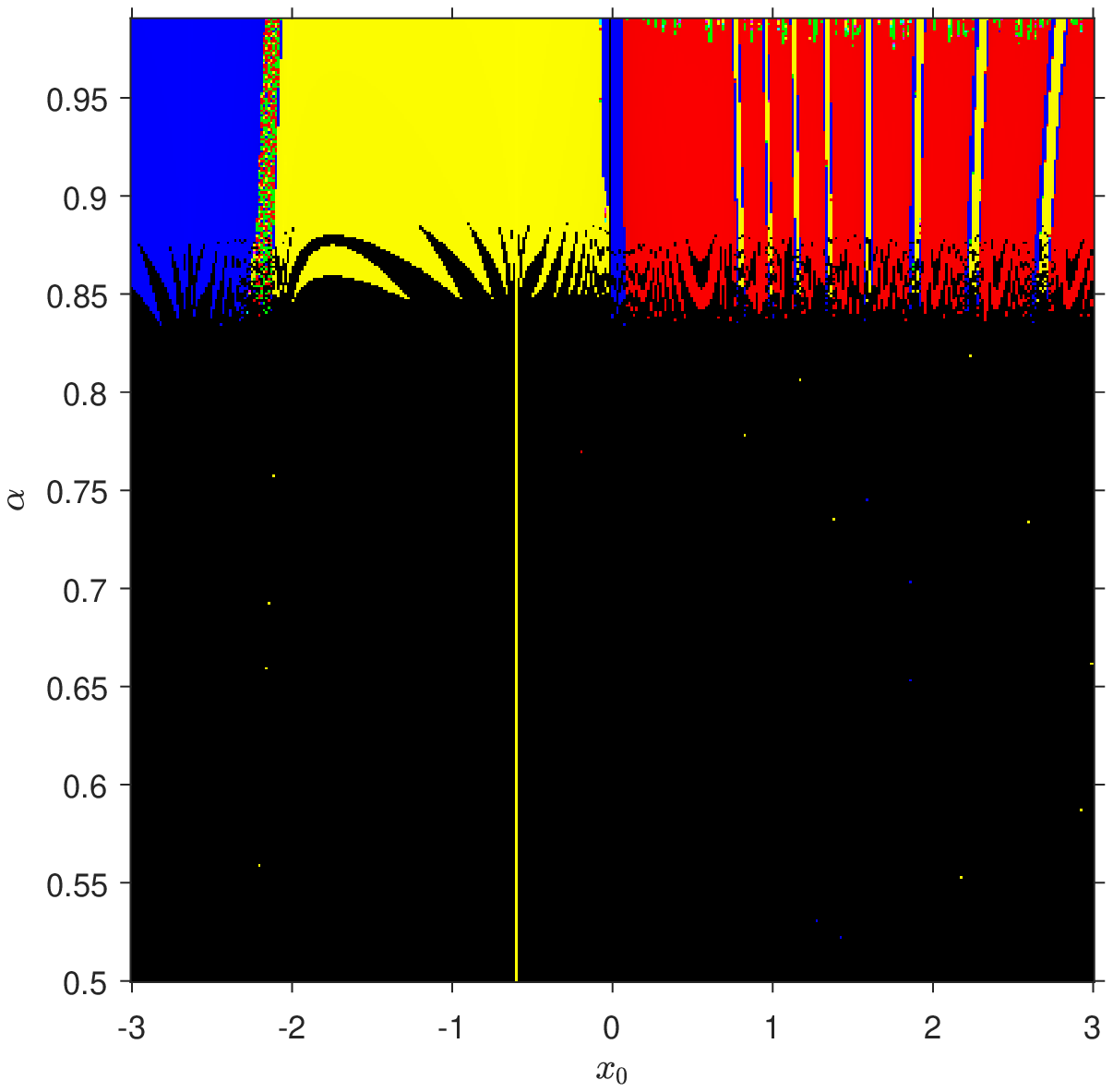}
(b) CFN$_2$, $-3\leq x_0\leq3$, 27.15\% convergence
\end{center}
\end{minipage}
\captionof{figure}{Convergence planes of CFN$_1$ and CFN$_2$ on $f_1(x)$ with $x_0$ real}\label{f1}
\begin{minipage}[c]{0.5\textwidth}
\begin{center}
\includegraphics[width=\textwidth]{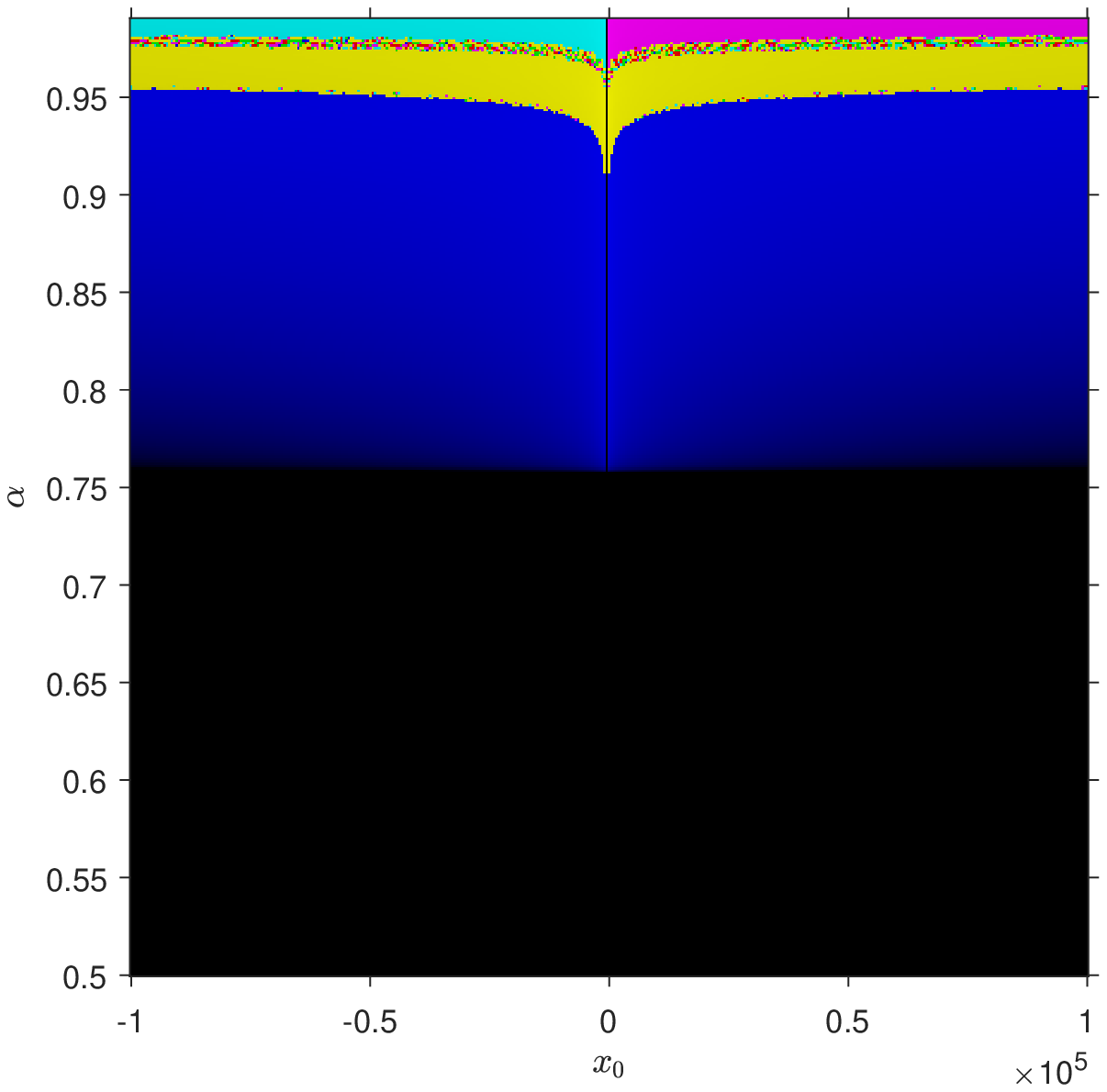}
(a) CFN$_1$, $-1e+05i\leq x_0\leq1e+05i$, 47.47\% convergence
\end{center}
\end{minipage}
\begin{minipage}[c]{0.5\textwidth}
\begin{center}
\includegraphics[width=\textwidth]{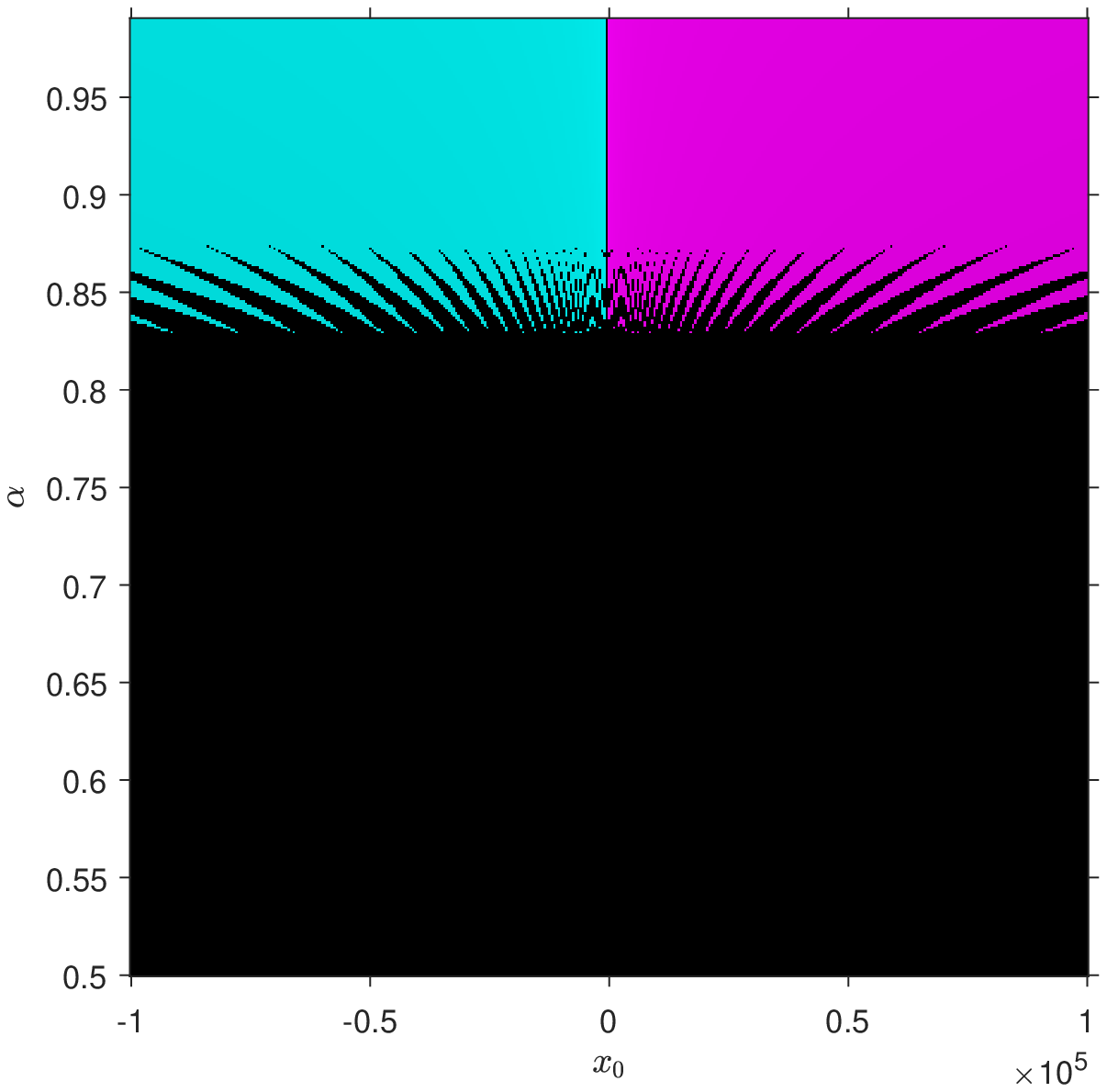}
(b) CFN$_2$, $-1e+05i\leq x_0\leq1e+05i$, 29\% convergence
\end{center}
\end{minipage}
\captionof{figure}{Convergence planes of CFN$_1$ and CFN$_2$ on $f_1(x)$ with $x_0$ imaginary}\label{f2}
\begin{minipage}[c]{0.5\textwidth}
\begin{center}
\includegraphics[width=\textwidth]{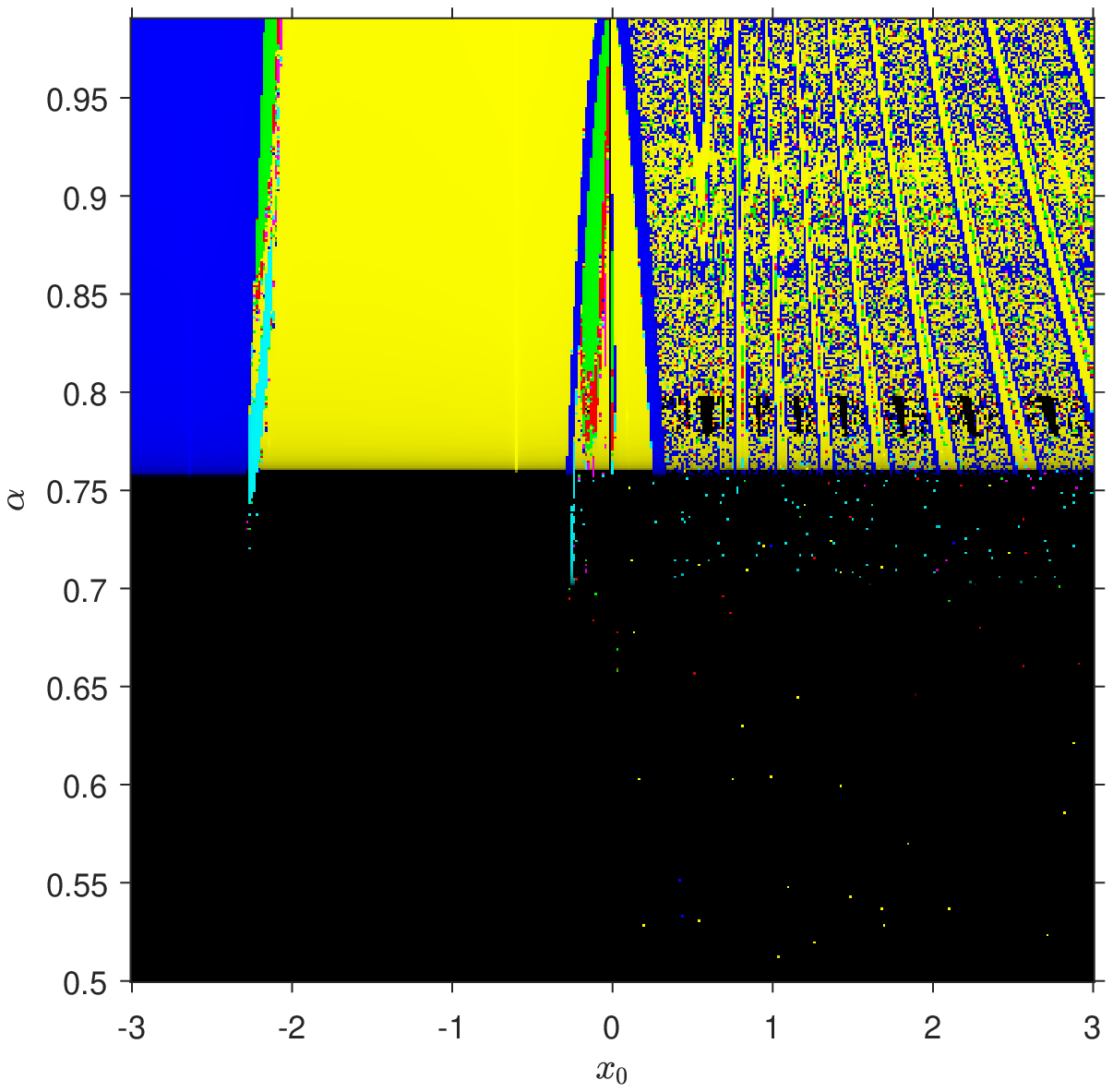}
(a) R-LFN$_1$, $-3\leq x_0\leq3$, 47.08\% convergence
\end{center}
\end{minipage}
\begin{minipage}[c]{0.5\textwidth}
\begin{center}
\includegraphics[width=\textwidth]{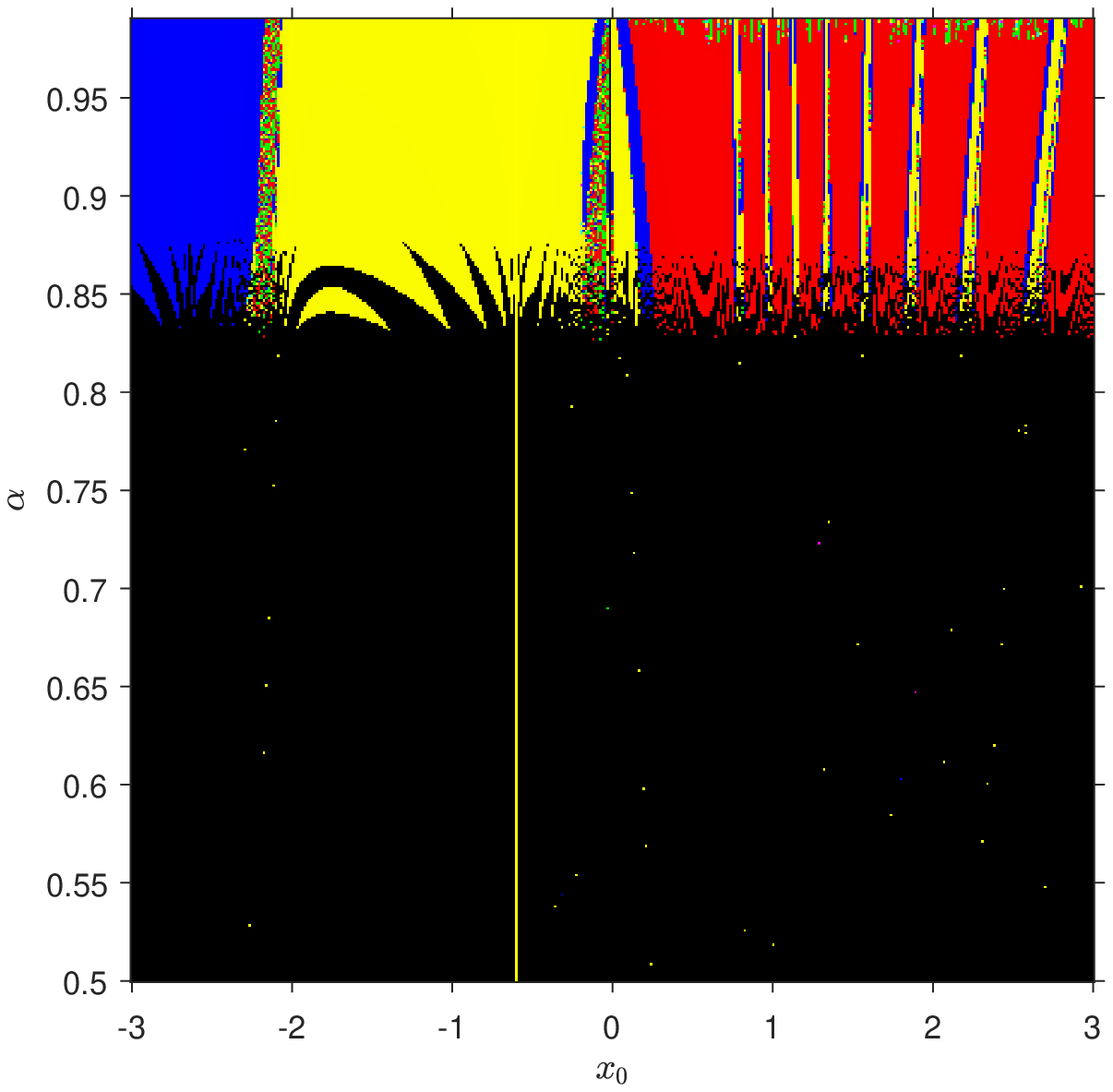}
(b) R-LFN$_2$, $-3\leq x_0\leq3$, 28.85\% convergence
\end{center}
\end{minipage}
\captionof{figure}{Convergence planes of R-LFN$_1$ and R-LFN$_2$ on $f_1(x)$ with $x_0$ real}\label{f3}
\begin{minipage}[c]{0.5\textwidth}
\begin{center}
\includegraphics[width=\textwidth]{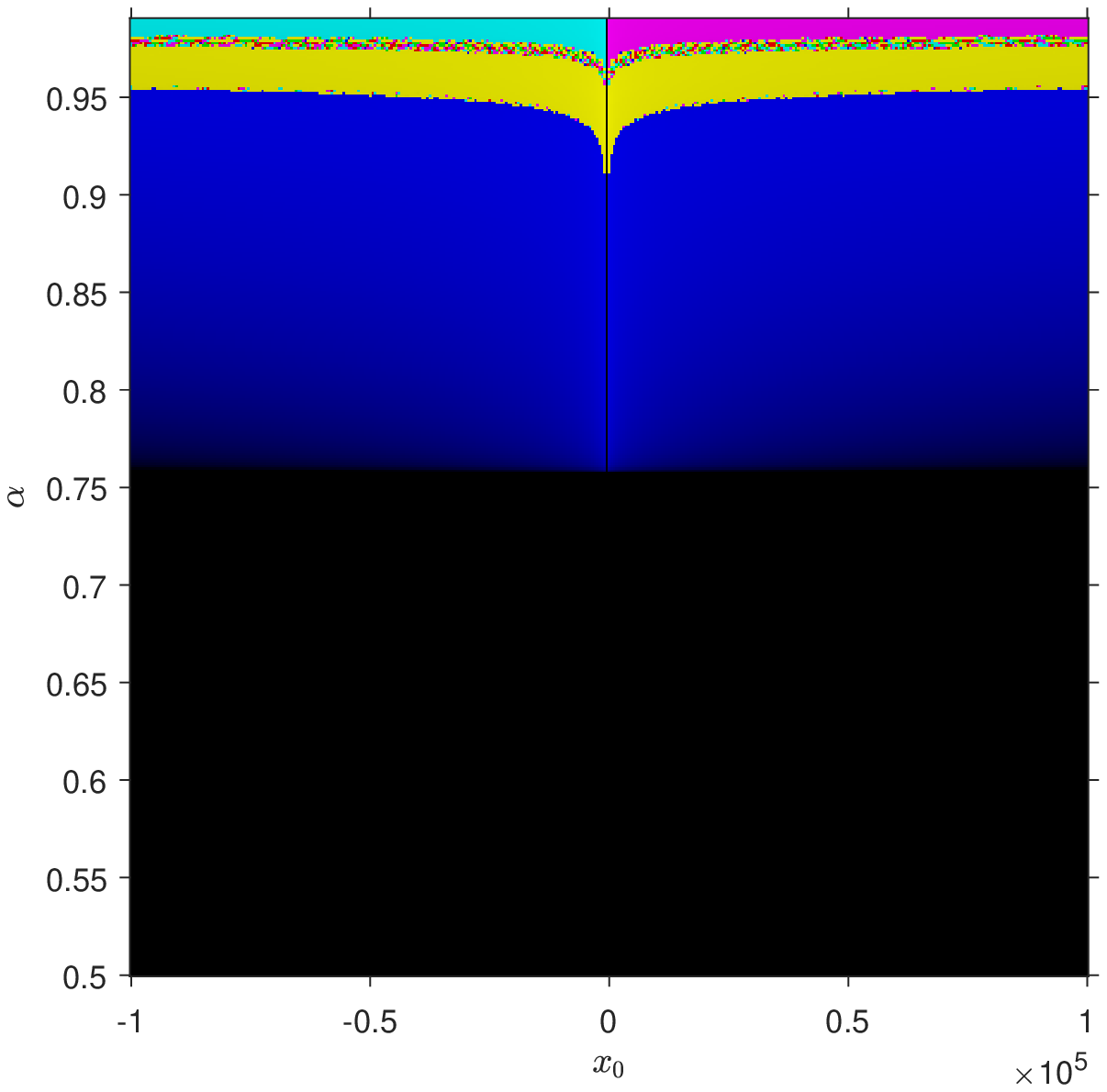}
(a) R-LFN$_1$, $-1e+05i\leq x_0\leq1e+05i$, 47.54\% convergence
\end{center}
\end{minipage}
\begin{minipage}[c]{0.5\textwidth}
\begin{center}
\includegraphics[width=\textwidth]{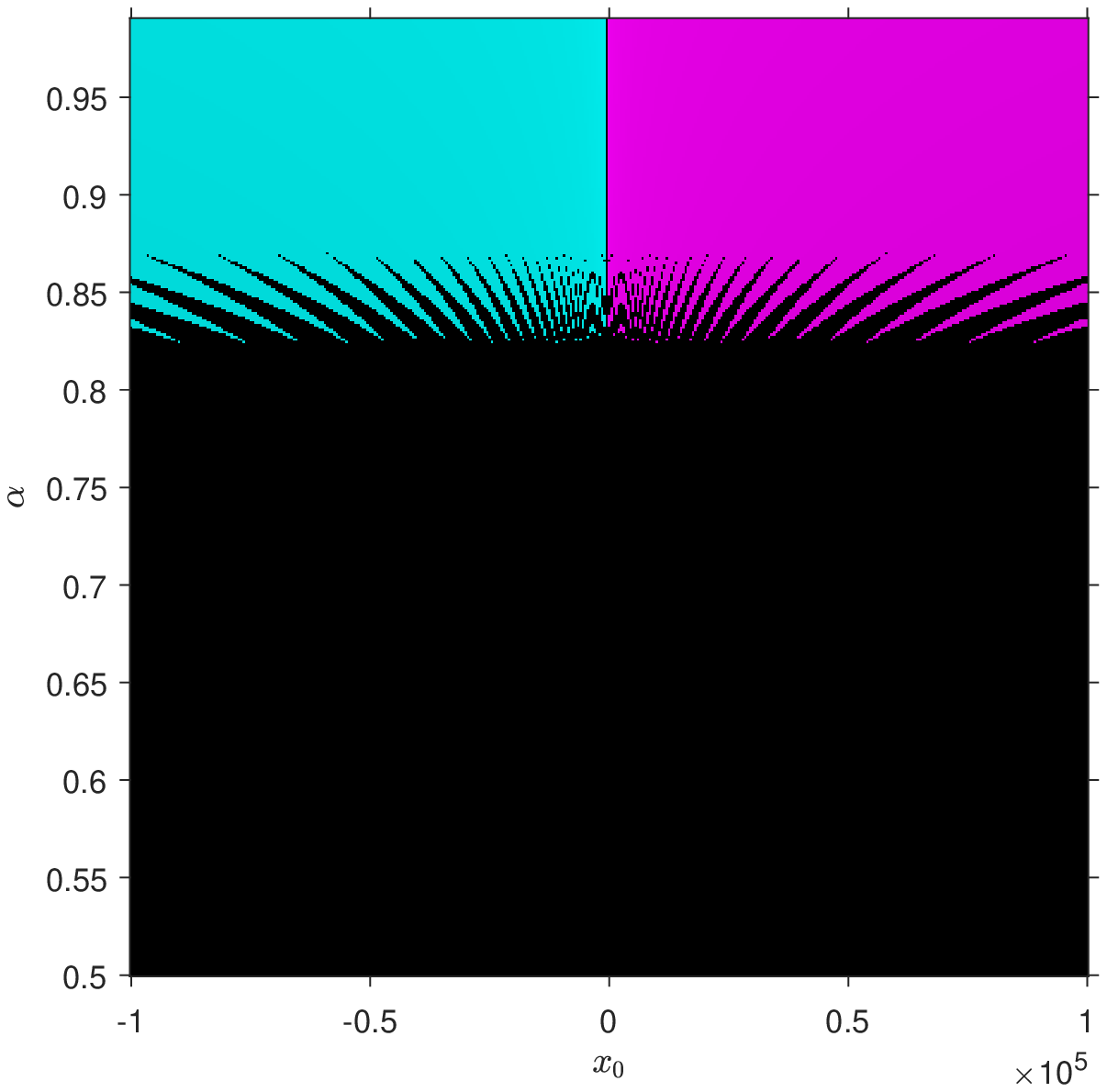}
(b) R-LFN$_2$, $-1e+05i\leq x_0\leq1e+05i$, 29.84\% convergence
\end{center}
\end{minipage}
\captionof{figure}{Convergence planes of R-LFN$_1$ and R-LFN$_2$ on $f_1(x)$ with $x_0$ imaginary}\label{f4}
\vspace{20pt}
Now, let us analyze the case of Traub and its first step. It can be observed in figures \ref{f5} and \ref{f6} that Traub methods have a higher percentage of convergence than their first steps. \\
\begin{minipage}[c]{0.5\textwidth}
\begin{center}
\includegraphics[width=\textwidth]{recursos/c_c_1}
(a) CFN$_2$, $-3\leq x_0\leq3$, 27.15\% convergence
\end{center}
\end{minipage}
\begin{minipage}[c]{0.5\textwidth}
\begin{center}
\includegraphics[width=\textwidth]{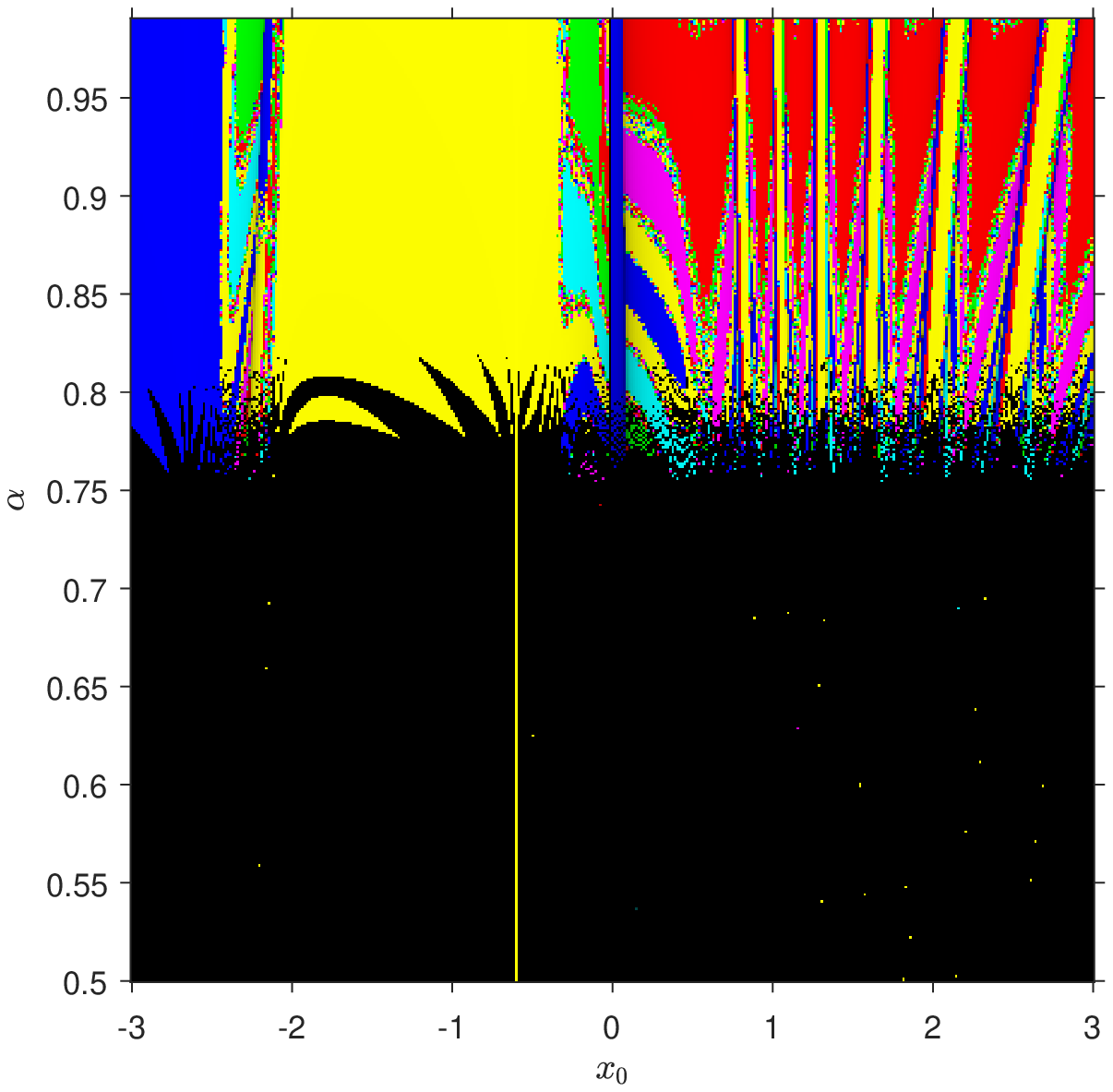}
(b) CFT, $-3\leq x_0\leq3$, 41.5\% convergence
\end{center}
\end{minipage}
\captionof{figure}{Convergence planes of CFN$_2$ and CFT on $f_1(x)$ with $x_0$ real}\label{f5}
\begin{minipage}[c]{0.5\textwidth}
\begin{center}
\includegraphics[width=\textwidth]{recursos/rl_c_1}
(a) R-LFN$_2$, $-3\leq x_0\leq3$, 28.85\% convergence
\end{center}
\end{minipage}
\begin{minipage}[c]{0.5\textwidth}
\begin{center}
\includegraphics[width=\textwidth]{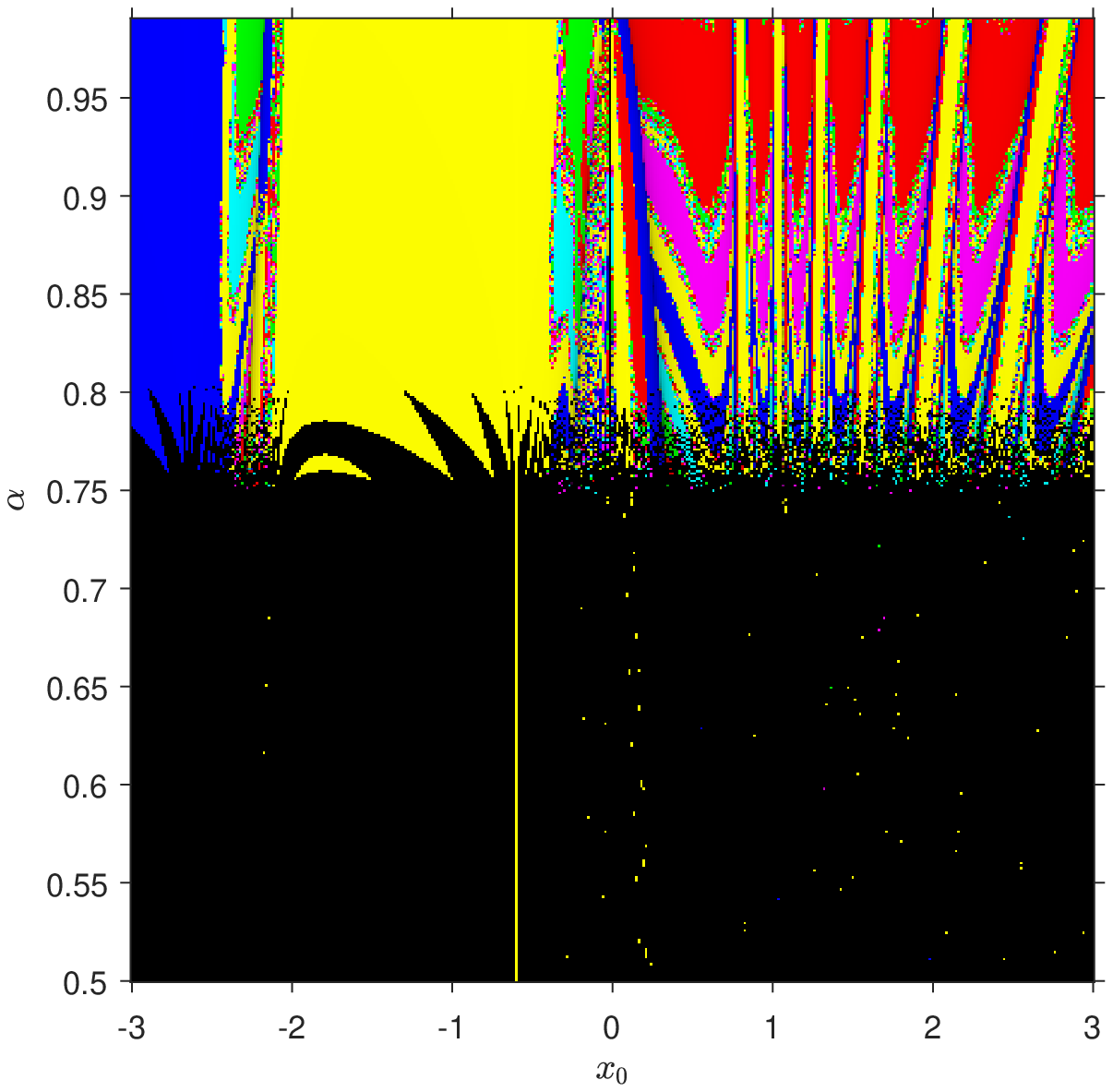}
(b) R-LFT, $-3\leq x_0\leq3$, 44\% convergence
\end{center}
\end{minipage}
\captionof{figure}{Convergence planes of R-LFN$_2$ and R-LFT on $f_1(x)$ with $x_0$ real}\label{f6}
\vspace{20pt}
Let us regard $f_2(x)$. In figures \ref{f7}-\ref{f12} we can observe exactly the same behavior for $f_2(x)$ as in figures \ref{f1}-\ref{f6} for $f_1(x)$. \\
\begin{minipage}[c]{0.5\textwidth}
\begin{center}
\includegraphics[width=\textwidth]{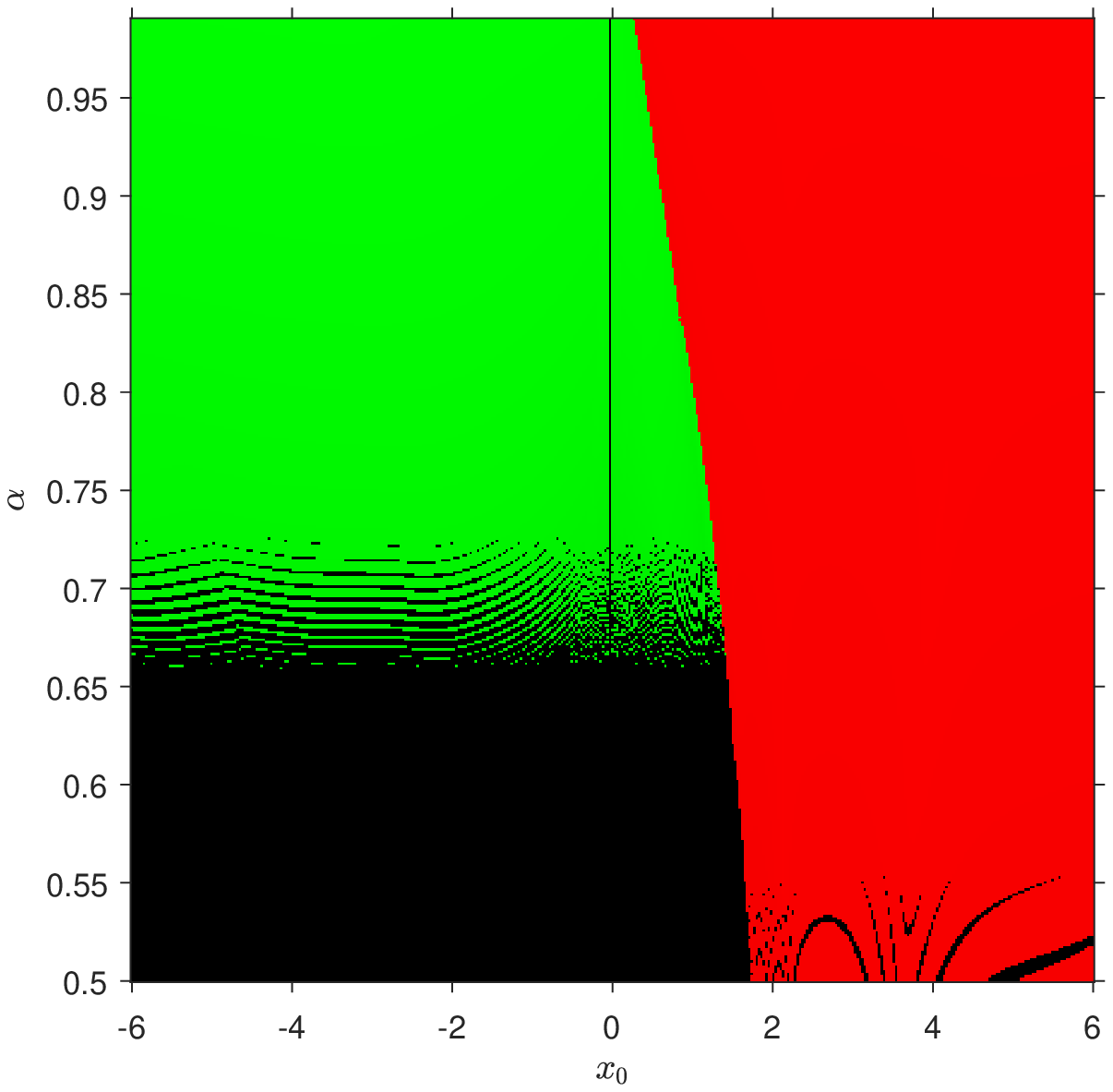}
(a) CFN$_1$, $-6\leq x_0\leq6$, 75.5\% convergence
\end{center}
\end{minipage}
\begin{minipage}[c]{0.5\textwidth}
\begin{center}
\includegraphics[width=\textwidth]{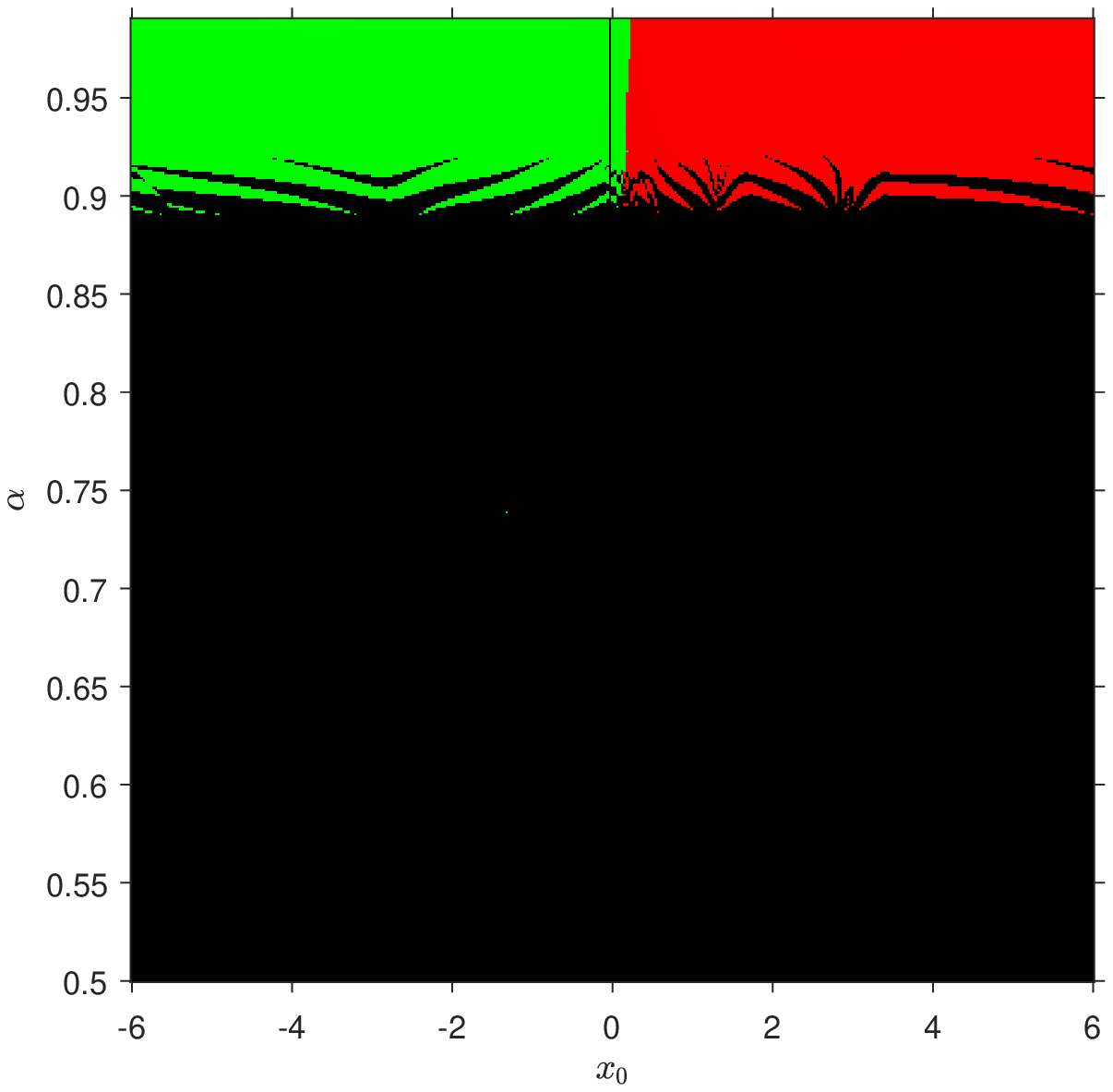}
(b) CFN$_2$, $-6\leq x_0\leq6$, 18.82\% convergence
\end{center}
\end{minipage}
\captionof{figure}{Convergence planes of CFN$_1$ and CFN$_2$ on $f_2(x)$ with $x_0$ real}\label{f7}
\begin{minipage}[c]{0.5\textwidth}
\begin{center}
\includegraphics[width=\textwidth]{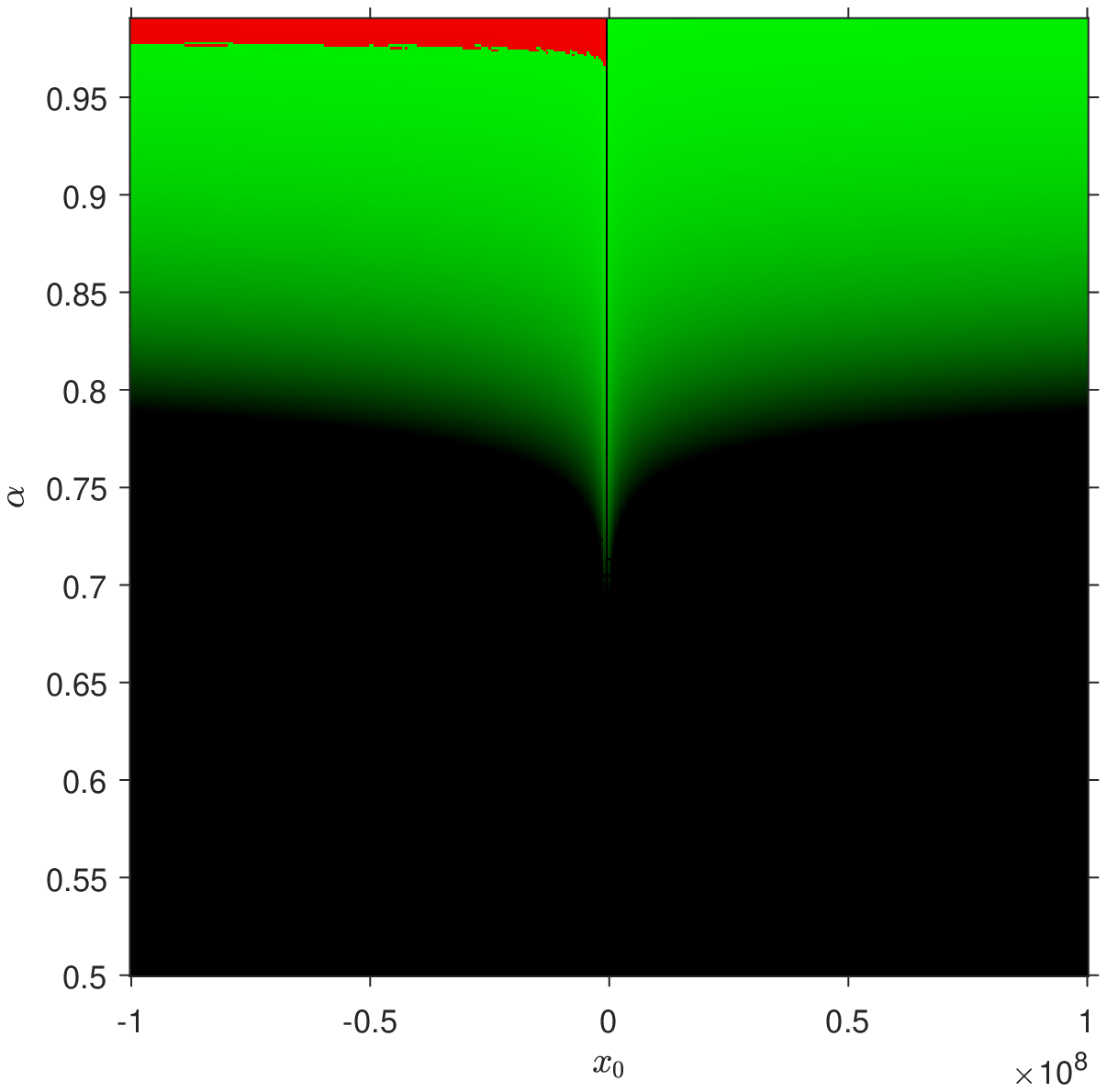}
(a) CFN$_1$, $-1e+08i\leq x_0\leq1e+08i$, 44.07\% convergence
\end{center}
\end{minipage}
\begin{minipage}[c]{0.5\textwidth}
\begin{center}
\includegraphics[width=\textwidth]{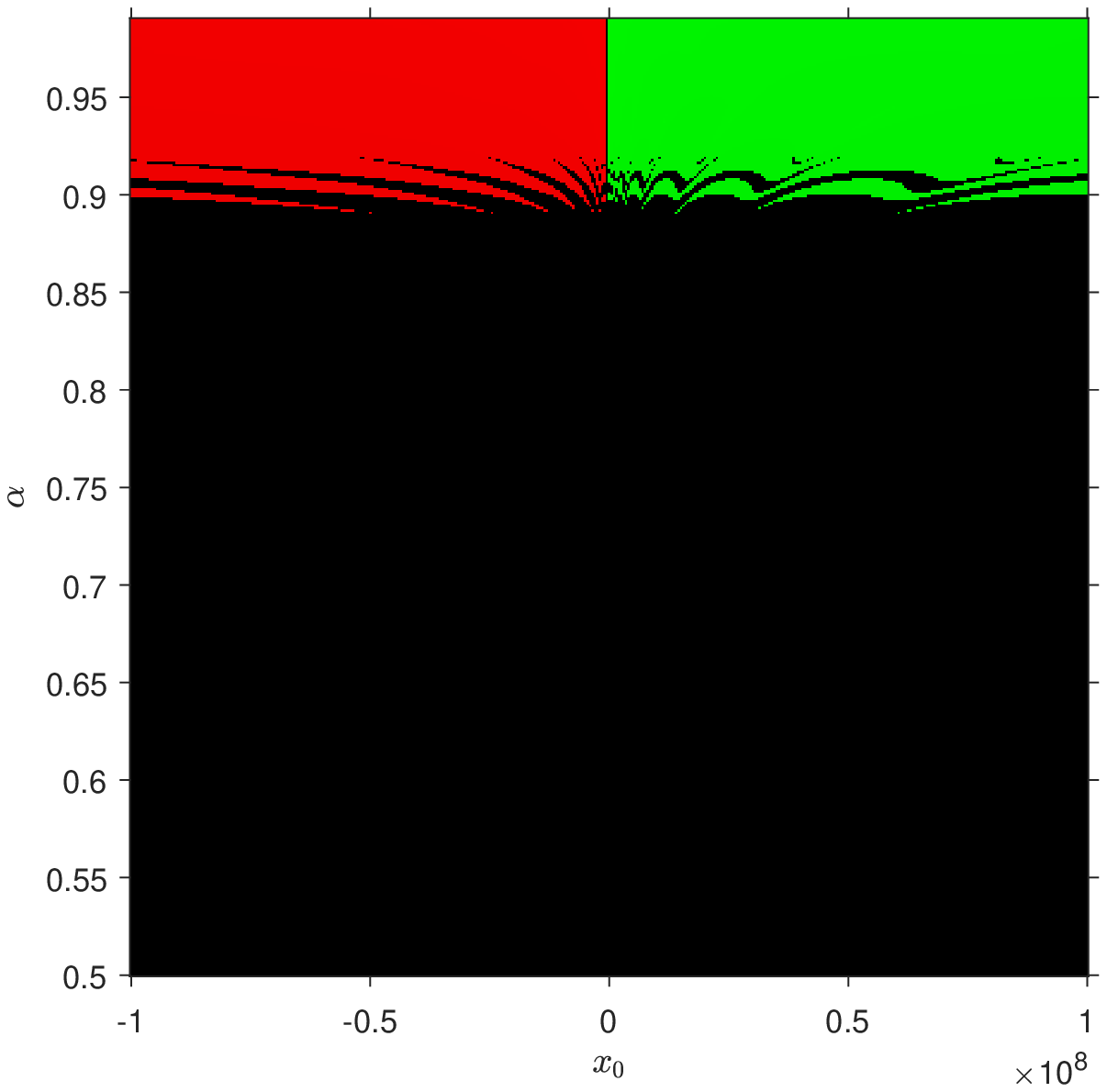}
(b) CFN$_2$, $-1e+08i\leq x_0\leq1e+08i$, 17.81\% convergence
\end{center}
\end{minipage}
\captionof{figure}{Convergence planes of CFN$_1$ and CFN$_2$ on $f_2(x)$ with $x_0$ imaginary}\label{f8}
\begin{minipage}[c]{0.5\textwidth}
\begin{center}
\includegraphics[width=\textwidth]{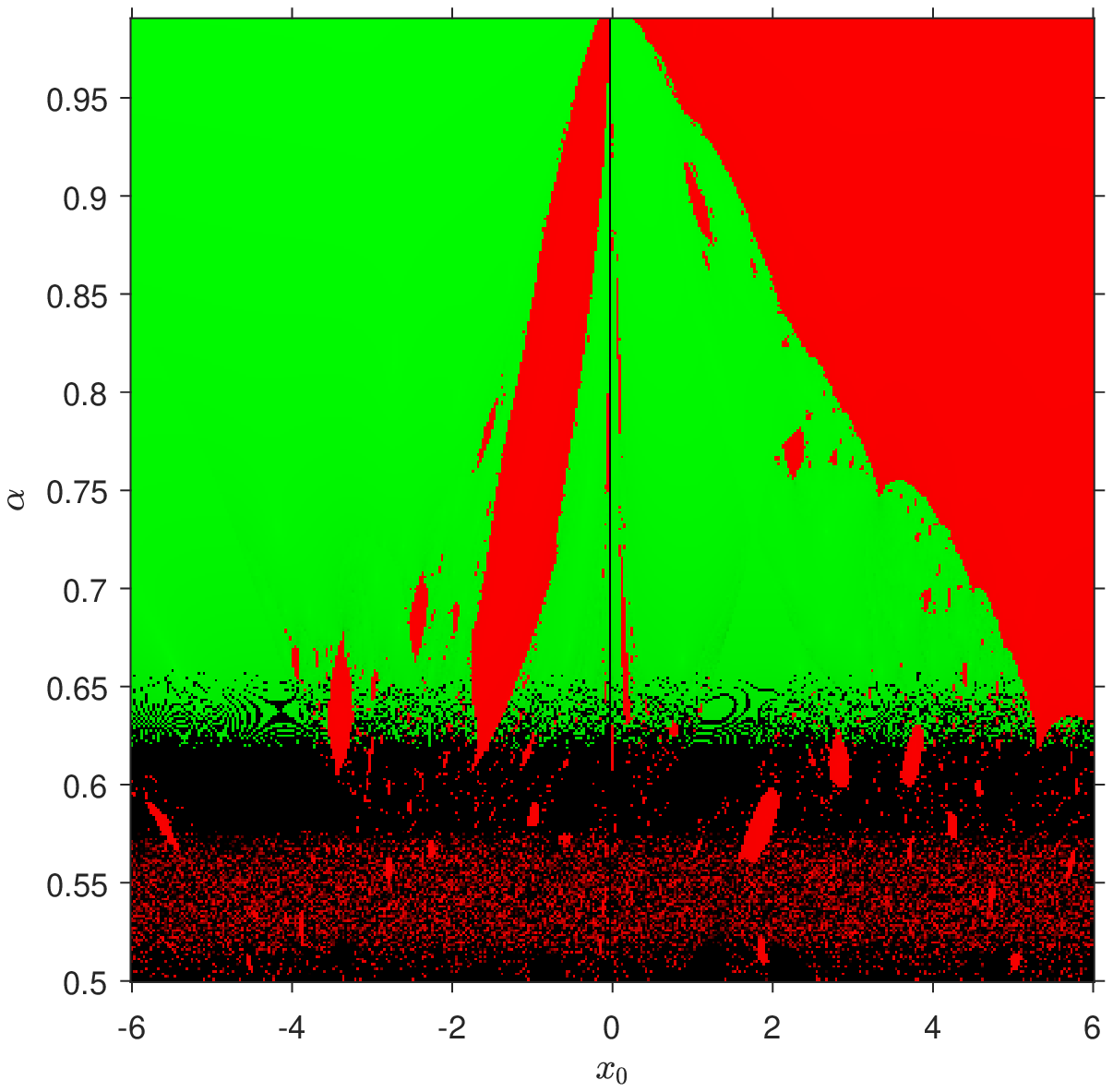}
(a) R-LFN$_1$, $-6\leq x_0\leq6$, 80.52\% convergence
\end{center}
\end{minipage}
\begin{minipage}[c]{0.5\textwidth}
\begin{center}
\includegraphics[width=\textwidth]{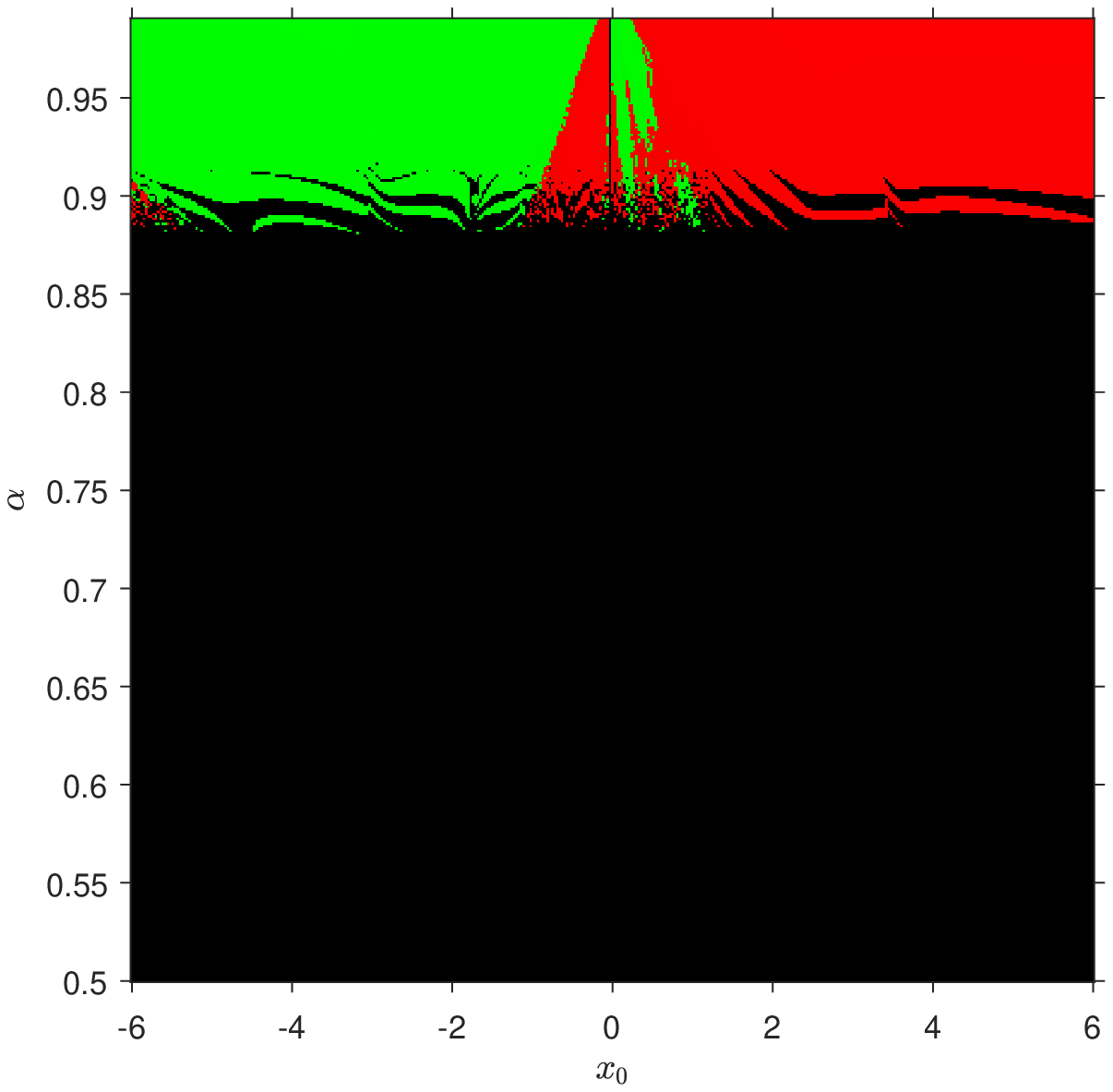}
(b) R-LFN$_2$, $-6\leq x_0\leq6$, 19.39\% convergence
\end{center}
\end{minipage}
\captionof{figure}{Convergence planes of R-LFN$_1$ and R-LFN$_2$ on $f_2(x)$ with $x_0$ real}\label{f9}
\begin{minipage}[c]{0.5\textwidth}
\begin{center}
\includegraphics[width=\textwidth]{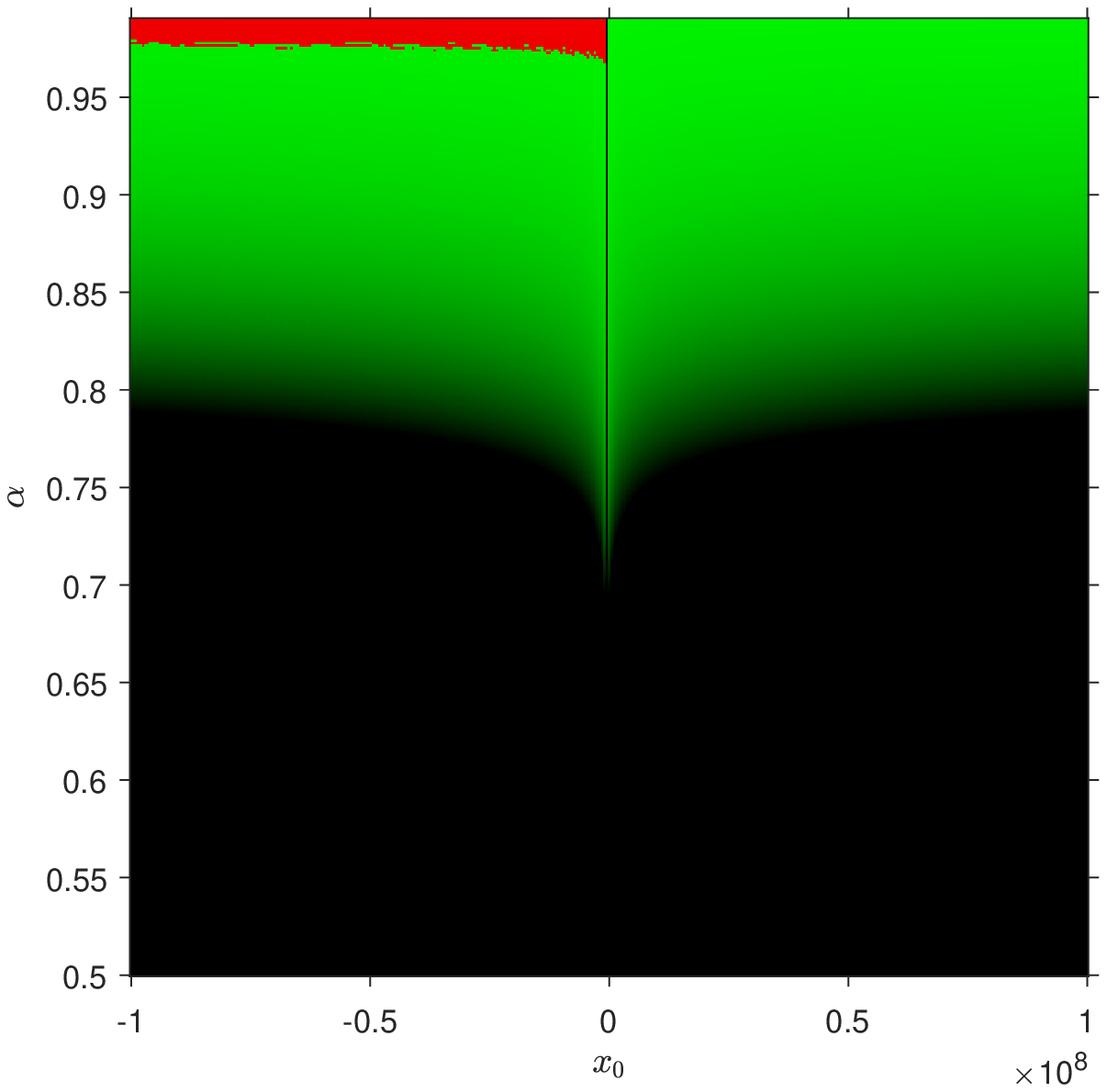}
(a) R-LFN$_1$, $-1e+08i\leq x_0\leq1e+08i$, 44.09\% convergence
\end{center}
\end{minipage}
\begin{minipage}[c]{0.5\textwidth}
\begin{center}
\includegraphics[width=\textwidth]{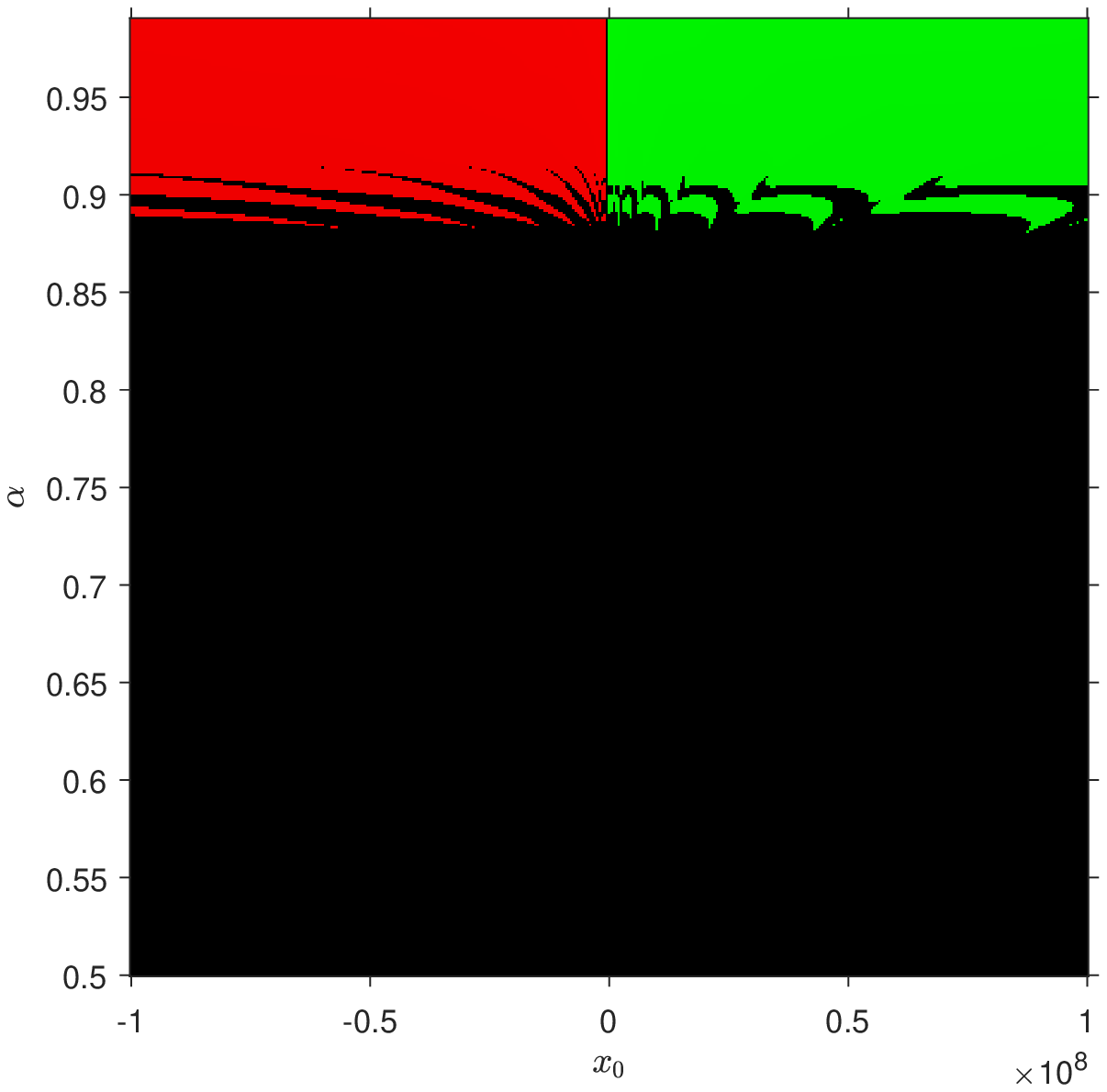}
(b) R-LFN$_2$, $-1e+08i\leq x_0\leq1e+08i$, 19.37\% convergence
\end{center}
\end{minipage}
\captionof{figure}{Convergence planes of R-LFN$_1$ and R-LFN$_2$ on $f_2(x)$ with $x_0$ imaginary}\label{f10}
\begin{minipage}[c]{0.5\textwidth}
\begin{center}
\includegraphics[width=\textwidth]{recursos/c_c_2}
(a) CFN$_2$, $-6\leq x_0\leq6$, 17.82\% convergence
\end{center}
\end{minipage}
\begin{minipage}[c]{0.5\textwidth}
\begin{center}
\includegraphics[width=\textwidth]{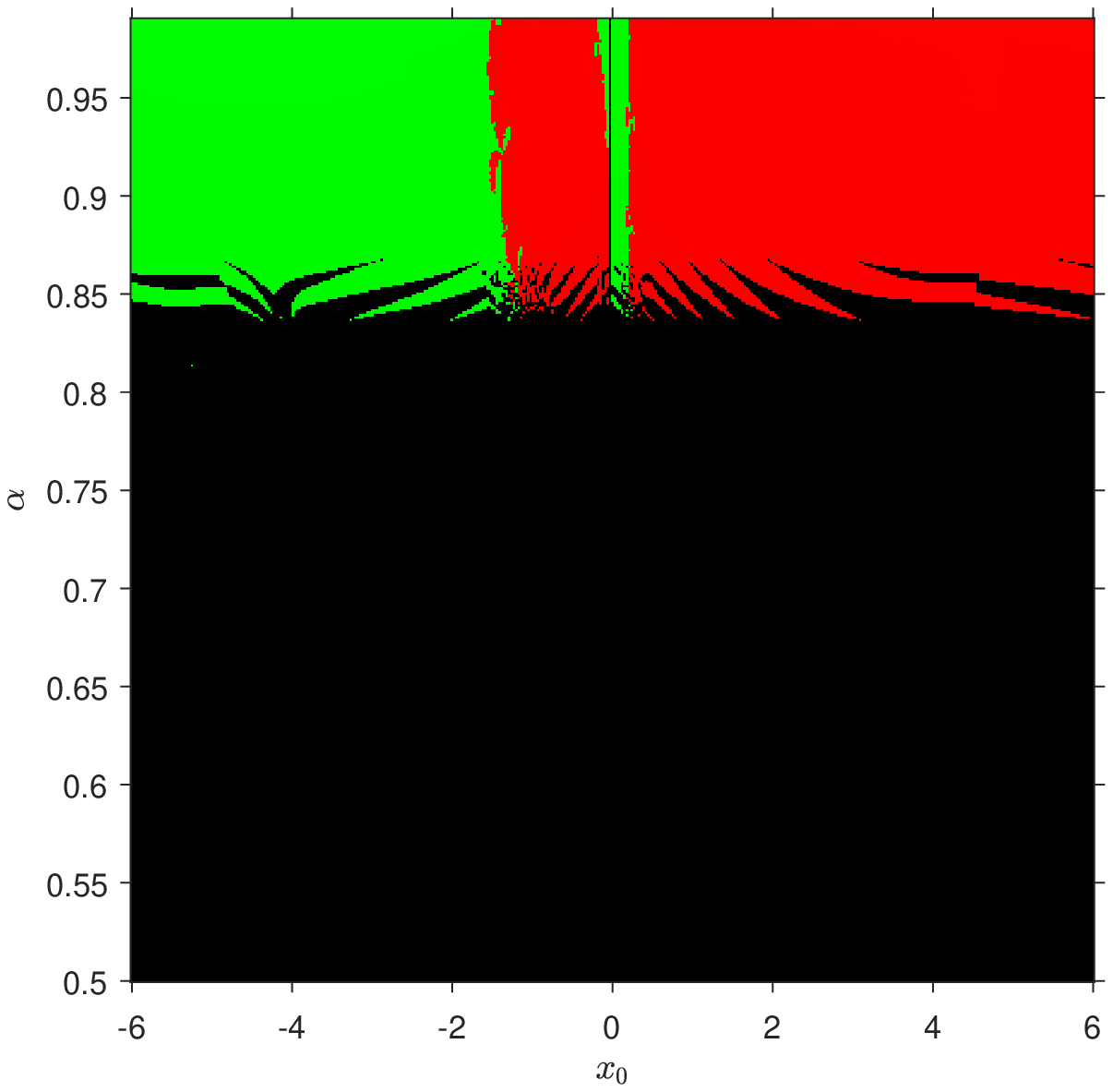}
(b) CFT, $-6\leq x_0\leq6$, 28.68\% convergence
\end{center}
\end{minipage}
\captionof{figure}{Convergence planes of CFN$_2$ and CFT on $f_2(x)$ with $x_0$ real}\label{f11}
\begin{minipage}[c]{0.5\textwidth}
\begin{center}
\includegraphics[width=\textwidth]{recursos/rl_c_2}
(a) R-LFN$_2$, $-6\leq x_0\leq6$, 19.39\% convergence
\end{center}
\end{minipage}
\begin{minipage}[c]{0.5\textwidth}
\begin{center}
\includegraphics[width=\textwidth]{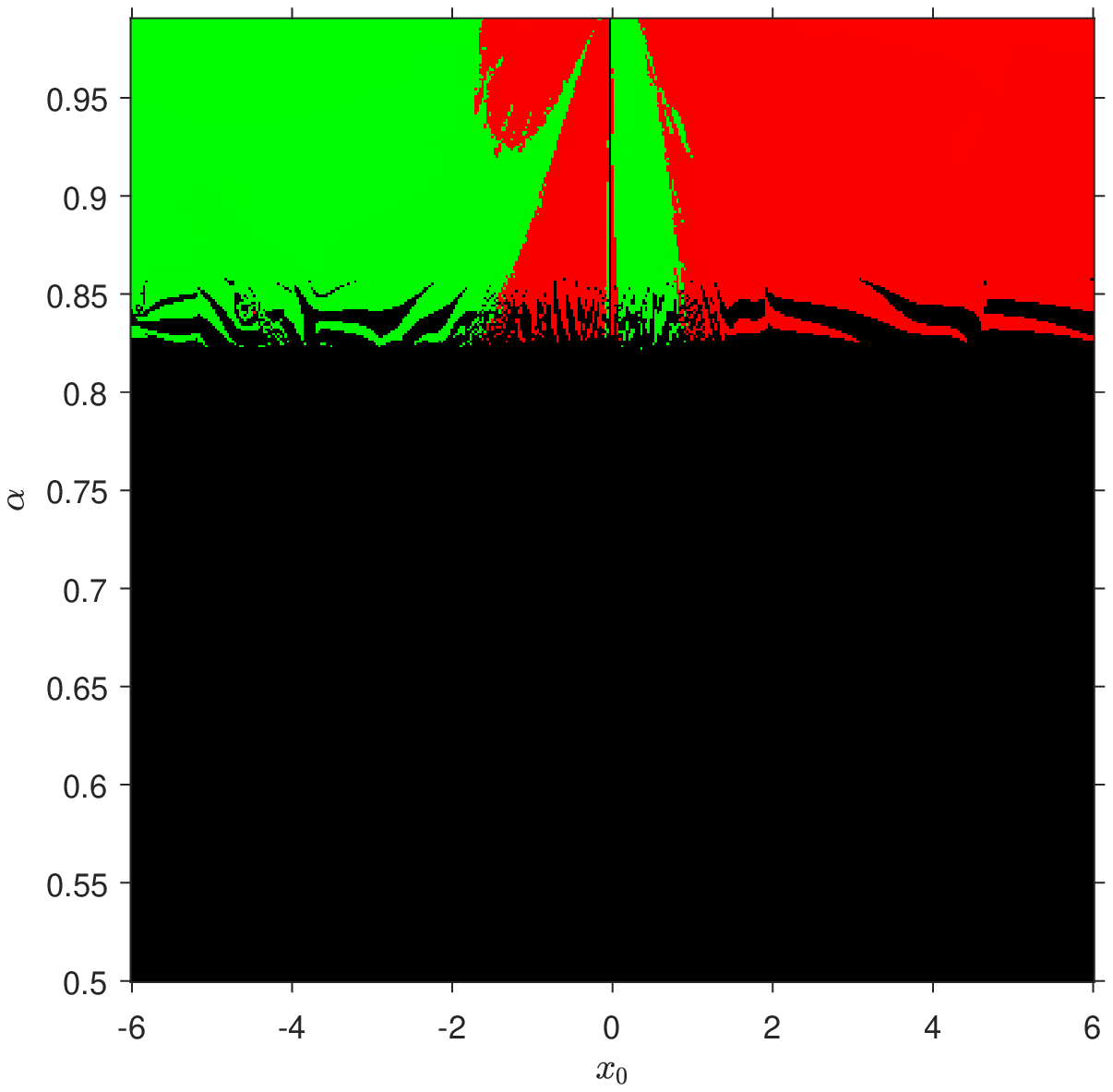}
(b) R-LFT, $-6\leq x_0\leq6$, 31.08\% convergence
\end{center}
\end{minipage}
\captionof{figure}{Convergence planes of R-LFN$_2$ and R-LFT on $f_2(x)$ with $x_0$ real}\label{f12}
\vspace{20pt}
Let us regard $f_3(x)$. For this function the results are very different to $f_1(x)$ and $f_2(x)$. In figure \ref{f13} we can see that CFN$_2$ improves considerably the percentage of convergence of CFN$_1$ for real values of initial estimations, while in figure \ref{f14} there is no convergence at all with both methods for imaginary values of initial estimations. In figure \ref{f15} R-LFN$_2$ slightly improves R-LFN$_1$, while in figure \ref{f16} the percentage convergence holds for both methods. \\
\begin{minipage}[c]{0.5\textwidth}
\begin{center}
\includegraphics[width=\textwidth]{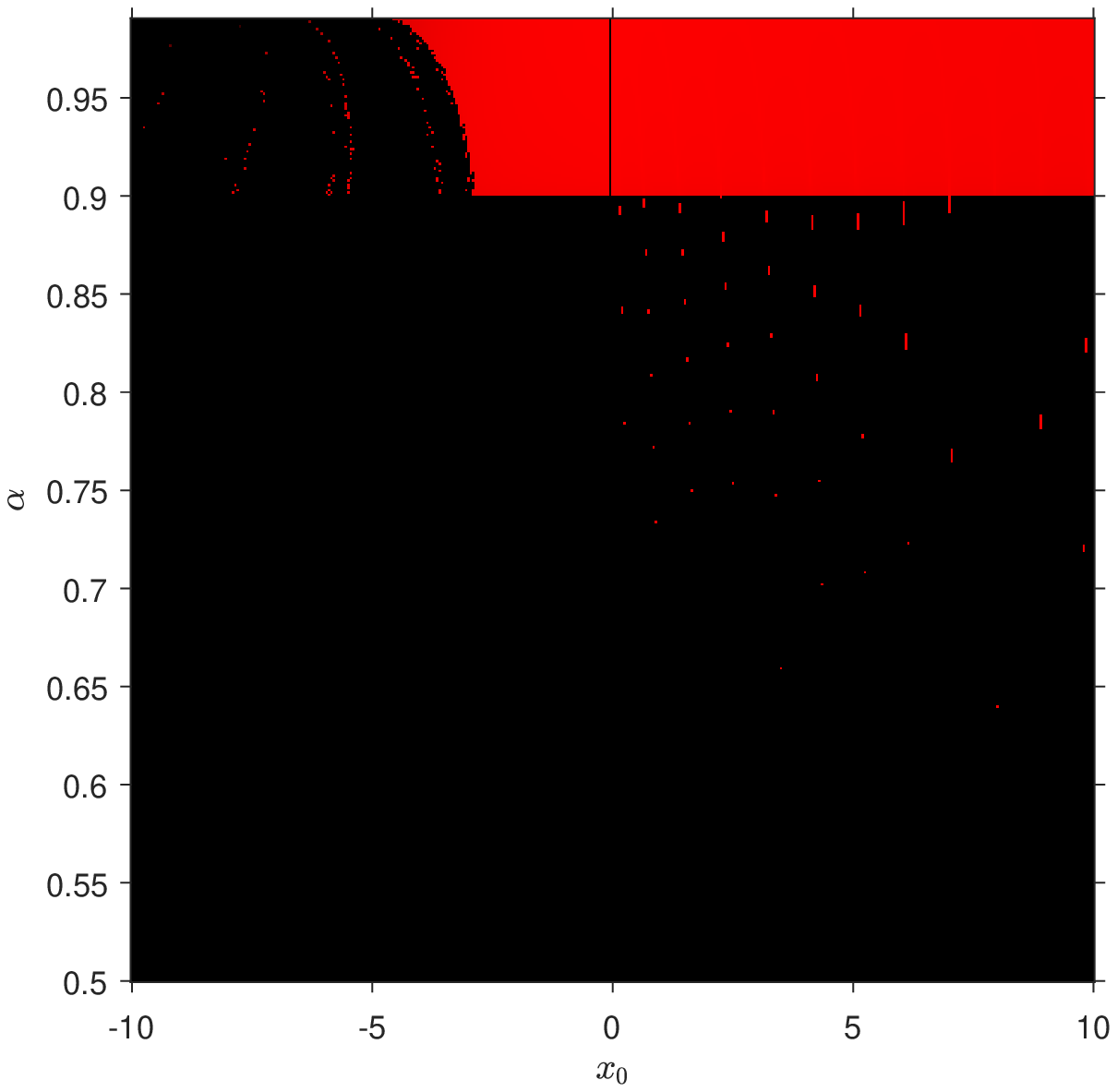}
(a) CFN$_1$, $-10\leq x_0\leq10$, 13.83\% convergence
\end{center}
\end{minipage}
\begin{minipage}[c]{0.5\textwidth}
\begin{center}
\includegraphics[width=\textwidth]{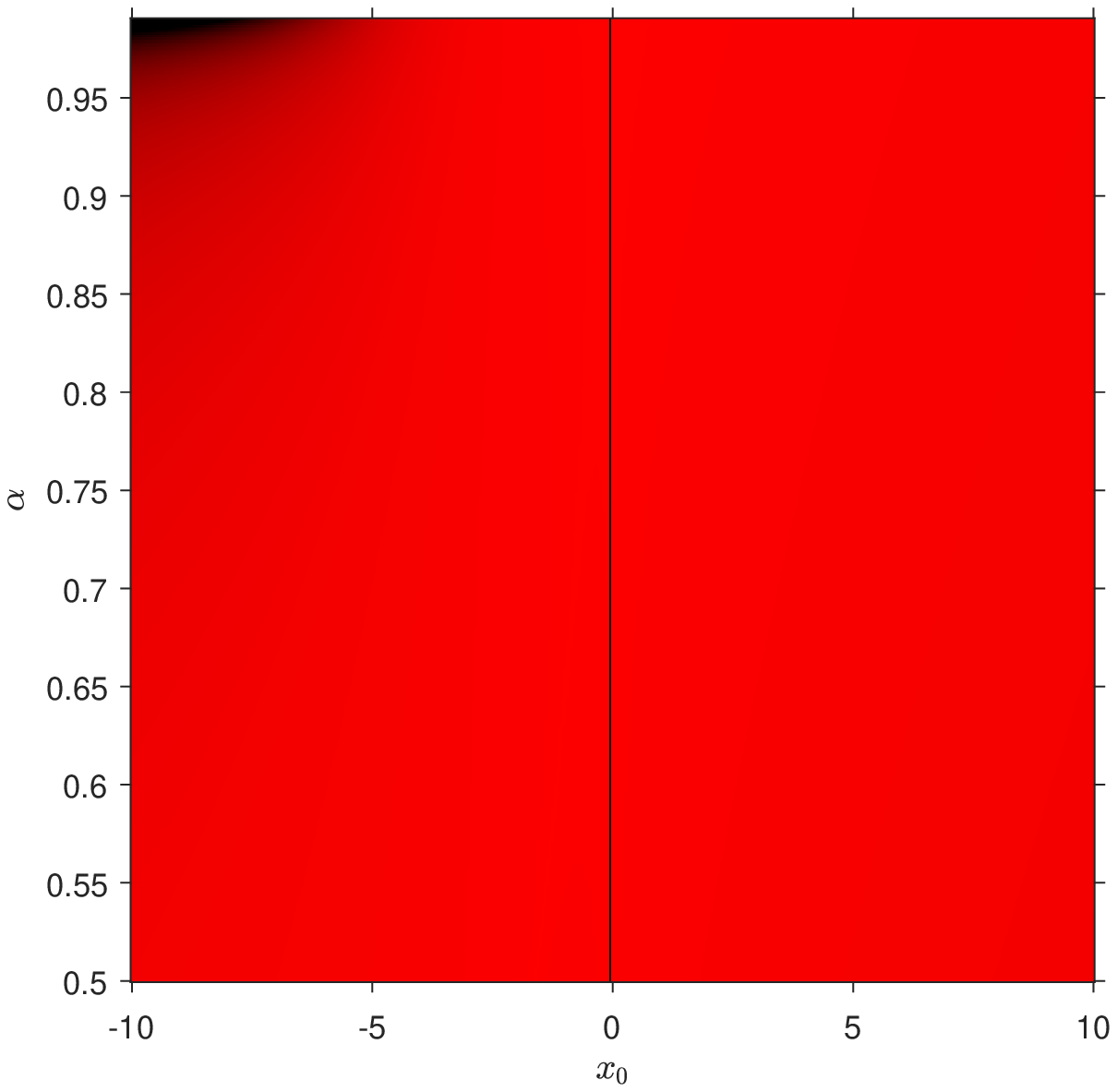}
(b) CFN$_2$, $-10\leq x_0\leq10$, 99.62\% convergence
\end{center}
\end{minipage}
\captionof{figure}{Convergence planes of CFN$_1$ and CFN$_2$ on $f_3(x)$ with $x_0$ real}\label{f13}
\begin{minipage}[c]{0.5\textwidth}
\begin{center}
\includegraphics[width=\textwidth]{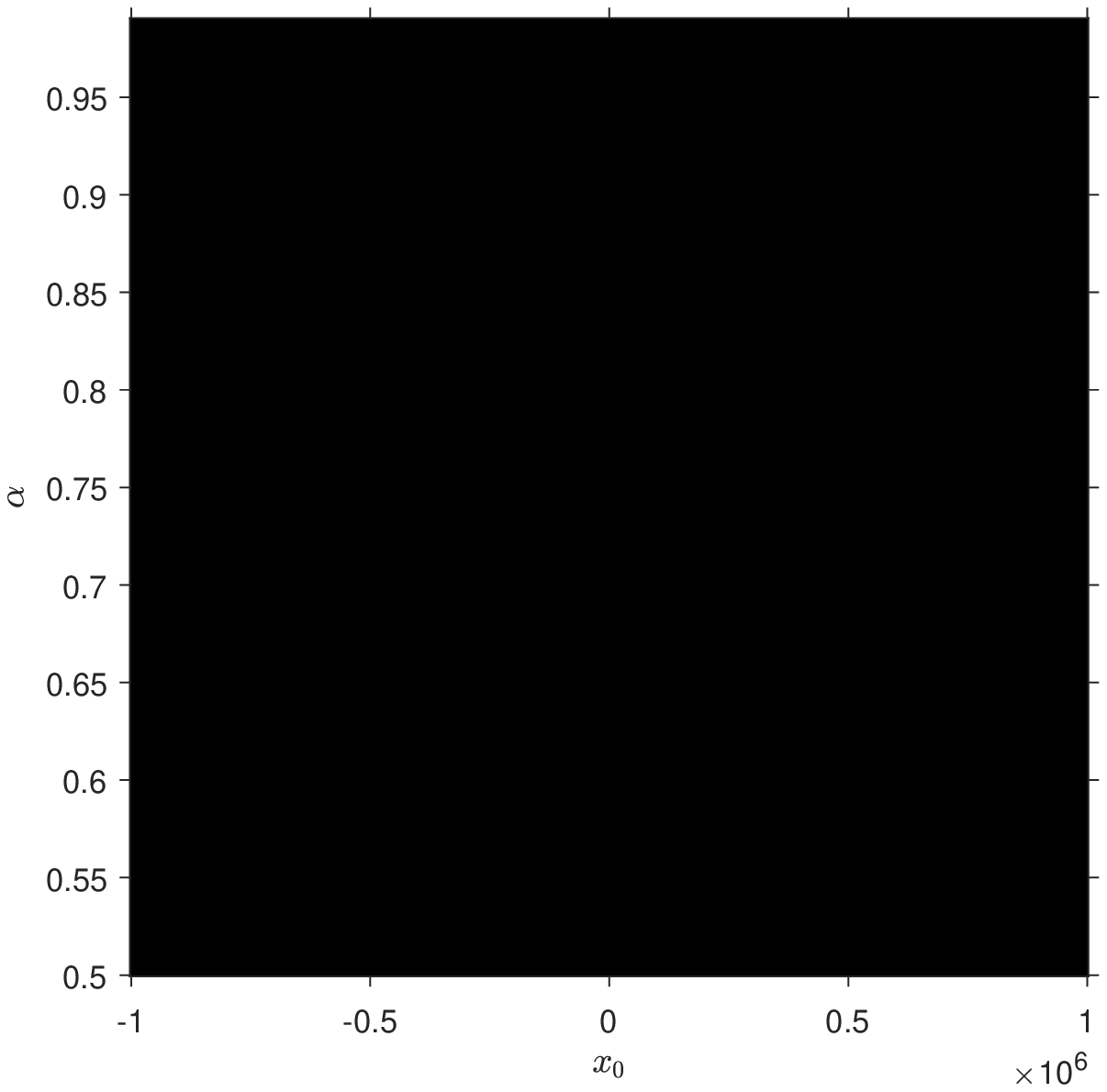}
(a) CFN$_1$, $-1e+06i\leq x_0\leq1e+06i$, 0\% convergence
\end{center}
\end{minipage}
\begin{minipage}[c]{0.5\textwidth}
\begin{center}
\includegraphics[width=\textwidth]{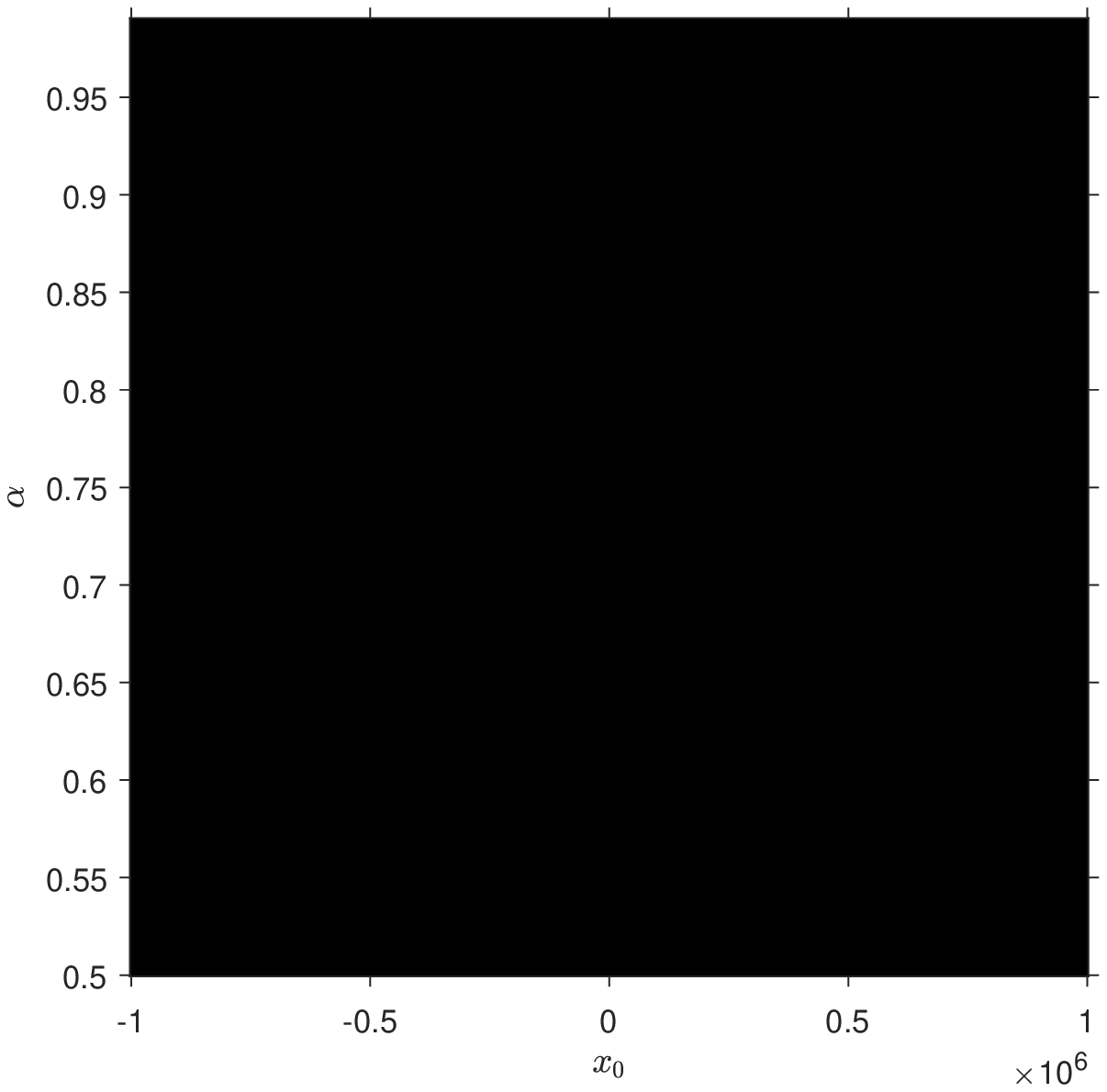}
(b) CFN$_2$, $-1e+06i\leq x_0\leq1e+06i$, 0\% convergence
\end{center}
\end{minipage}
\captionof{figure}{Convergence planes of CFN$_1$ and CFN$_2$ on $f_3(x)$ with $x_0$ imaginary}\label{f14}
\begin{minipage}[c]{0.5\textwidth}
\begin{center}
\includegraphics[width=\textwidth]{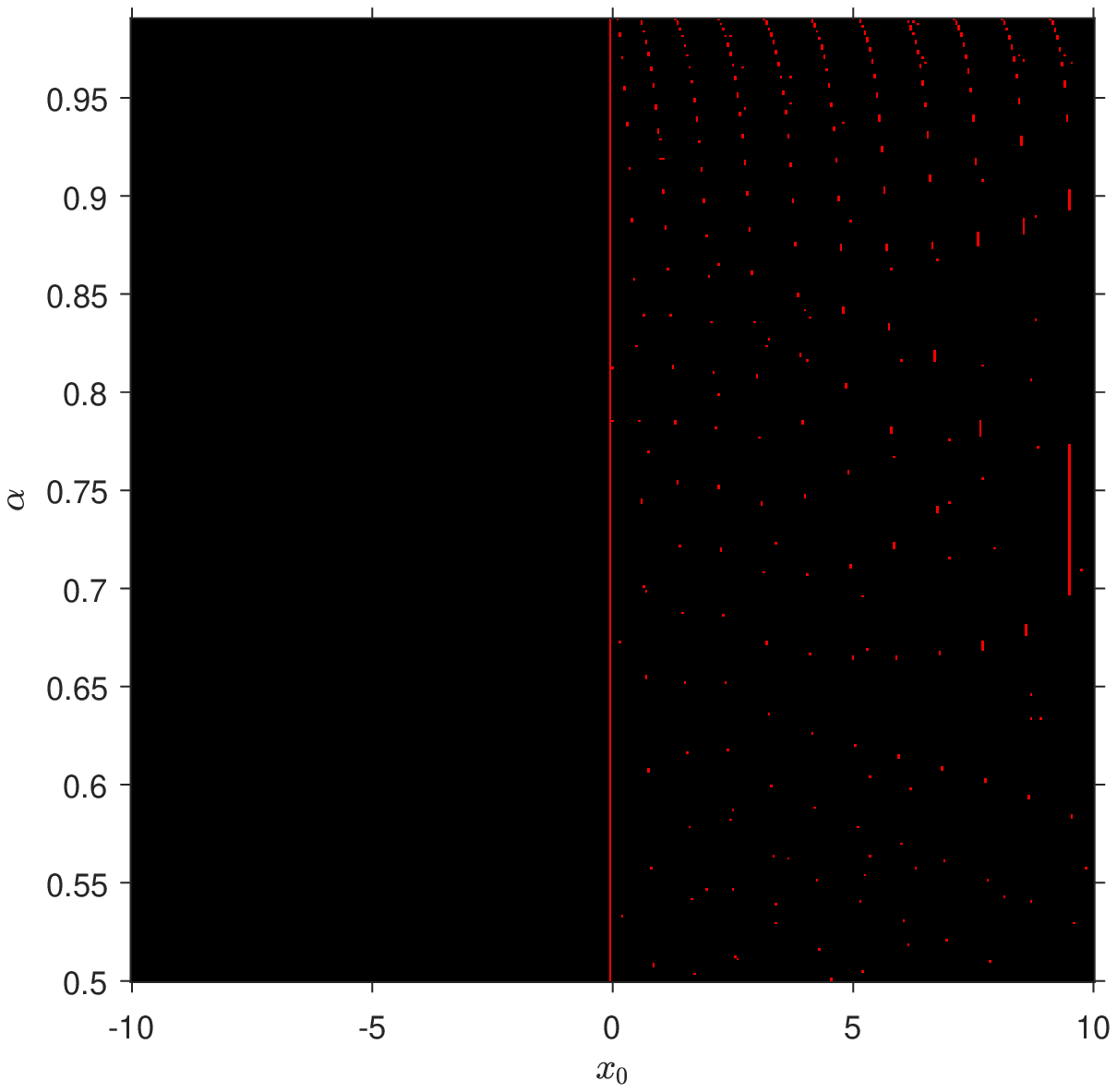}
(a) R-LFN$_1$, $-10\leq x_0\leq10$, 0.58\% convergence
\end{center}
\end{minipage}
\begin{minipage}[c]{0.5\textwidth}
\begin{center}
\includegraphics[width=\textwidth]{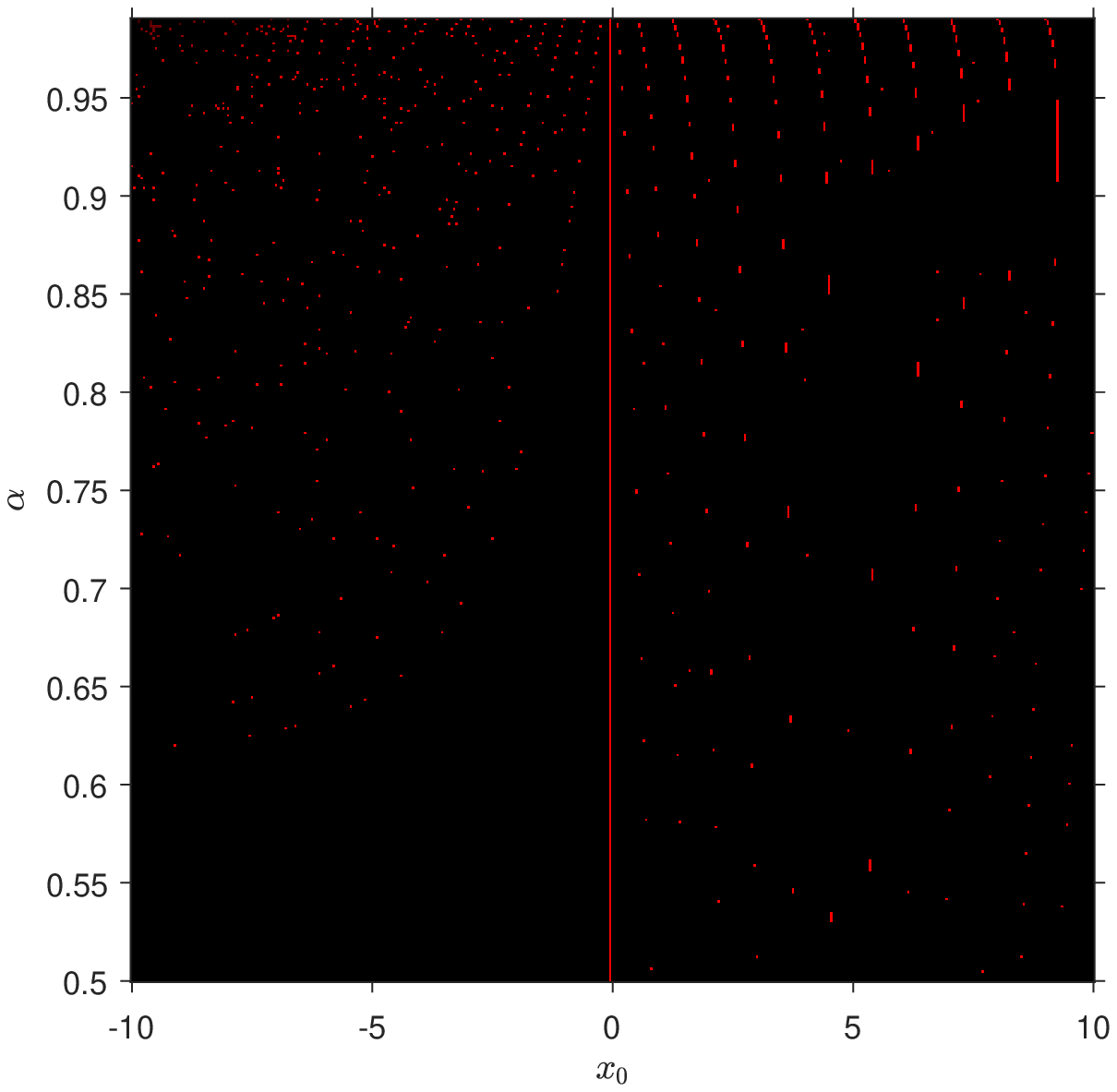}
(b) R-LFN$_2$, $-10\leq x_0\leq10$, 0.8\% convergence
\end{center}
\end{minipage}
\captionof{figure}{Convergence planes of R-LFN$_1$ and R-LFN$_2$ on $f_3(x)$ with $x_0$ real}\label{f15}
\begin{minipage}[c]{0.5\textwidth}
\begin{center}
\includegraphics[width=\textwidth]{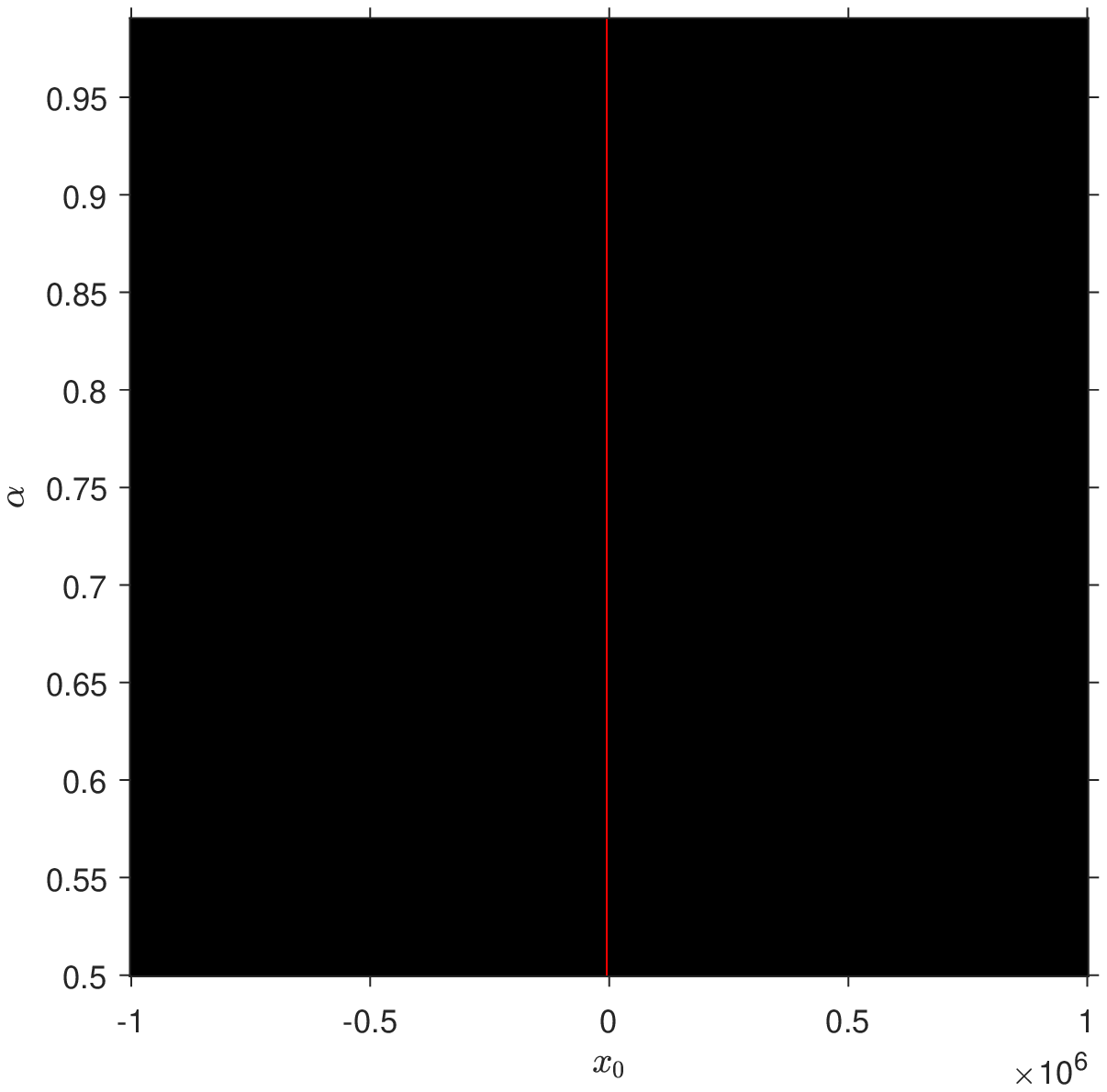}
(a) R-LFN$_1$, $-1e+06i\leq x_0\leq1e+06i$, 0.25\% convergence
\end{center}
\end{minipage}
\begin{minipage}[c]{0.5\textwidth}
\begin{center}
\includegraphics[width=\textwidth]{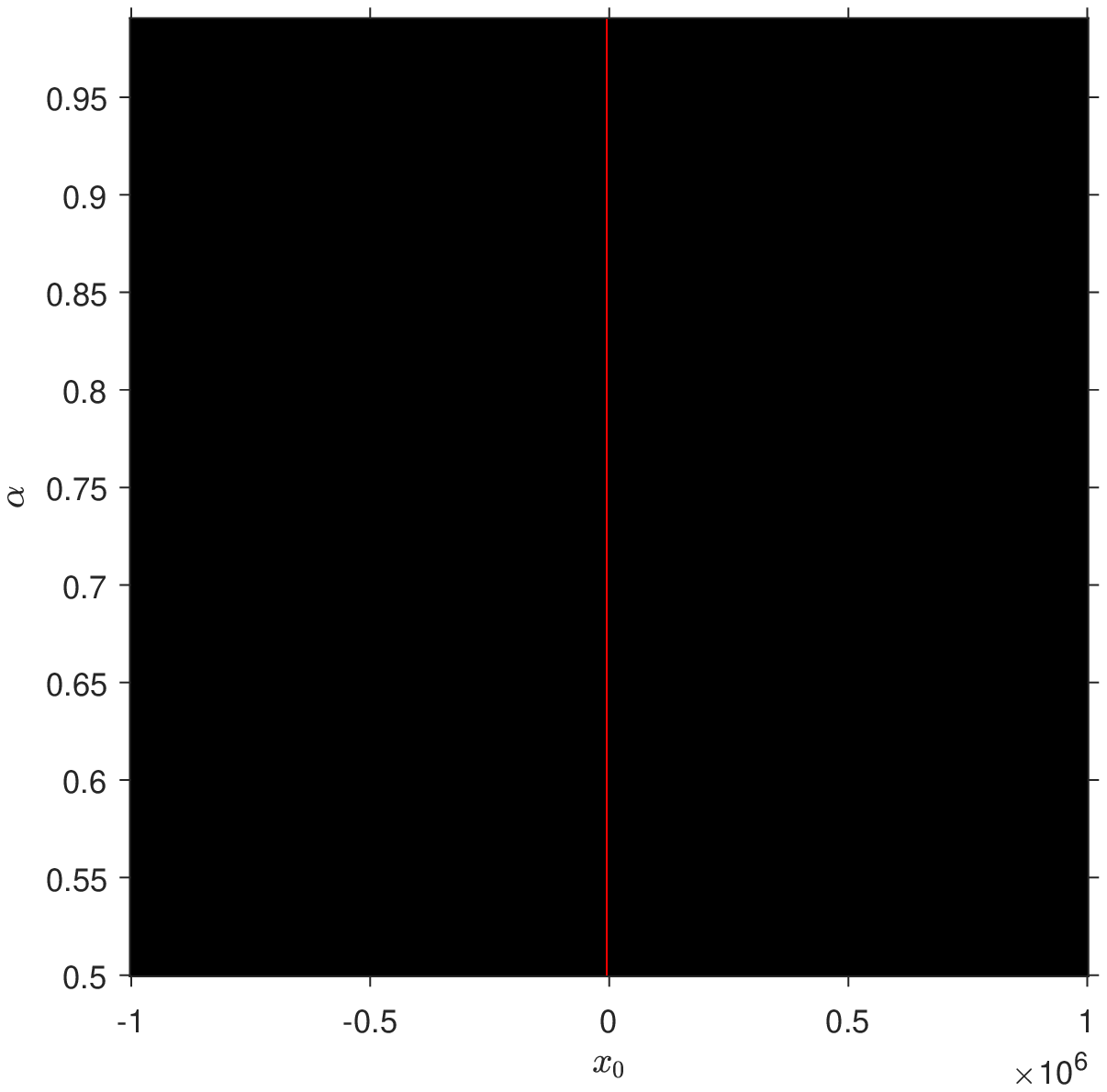}
(b) R-LFN$_2$, $-1e+06i\leq x_0\leq1e+06i$, 0.25\% convergence
\end{center}
\end{minipage}
\captionof{figure}{Convergence planes of R-LFN$_1$ and R-LFN$_2$ on $f_3(x)$ with $x_0$ imaginary}\label{f16}
\vspace{20pt}
In the case of Traub, it can be observed in figure \ref{f17} that Traub method does not improve the percentage convergence of its first step with Caputo derivative, while in figure \ref{f18} we can observe that Traub improves the percentage convergence of its first step with Riemann-Liouville derivative as in $f_1(x)$ and $f_2(x)$. \\
\begin{minipage}[c]{0.5\textwidth}
\begin{center}
\includegraphics[width=\textwidth]{recursos/c_c_3}
(a) CFN$_2$, $-10\leq x_0\leq10$, 99.62\% convergence
\end{center}
\end{minipage}
\begin{minipage}[c]{0.5\textwidth}
\begin{center}
\includegraphics[width=\textwidth]{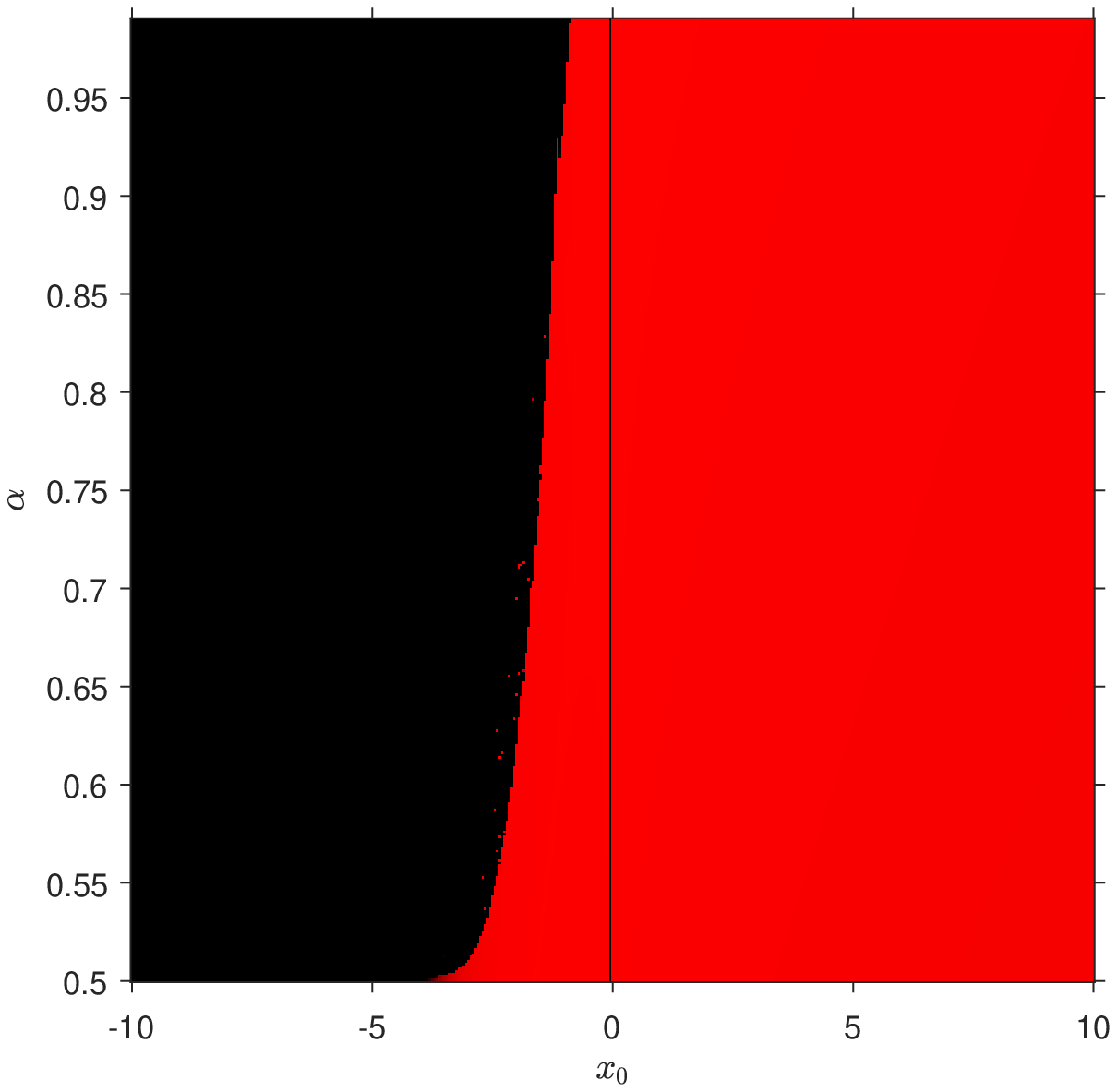}
(b) CFT, $-10\leq x_0\leq10$, 57.78\% convergence
\end{center}
\end{minipage}
\captionof{figure}{Convergence planes of CFN$_2$ and CFT on $f_3(x)$ with $x_0$ real}\label{f17}
\begin{minipage}[c]{0.5\textwidth}
\begin{center}
\includegraphics[width=\textwidth]{recursos/rl_c_3}
(a) R-LFN$_2$, $-10\leq x_0\leq10$, 0.8\% convergence
\end{center}
\end{minipage}
\begin{minipage}[c]{0.5\textwidth}
\begin{center}
\includegraphics[width=\textwidth]{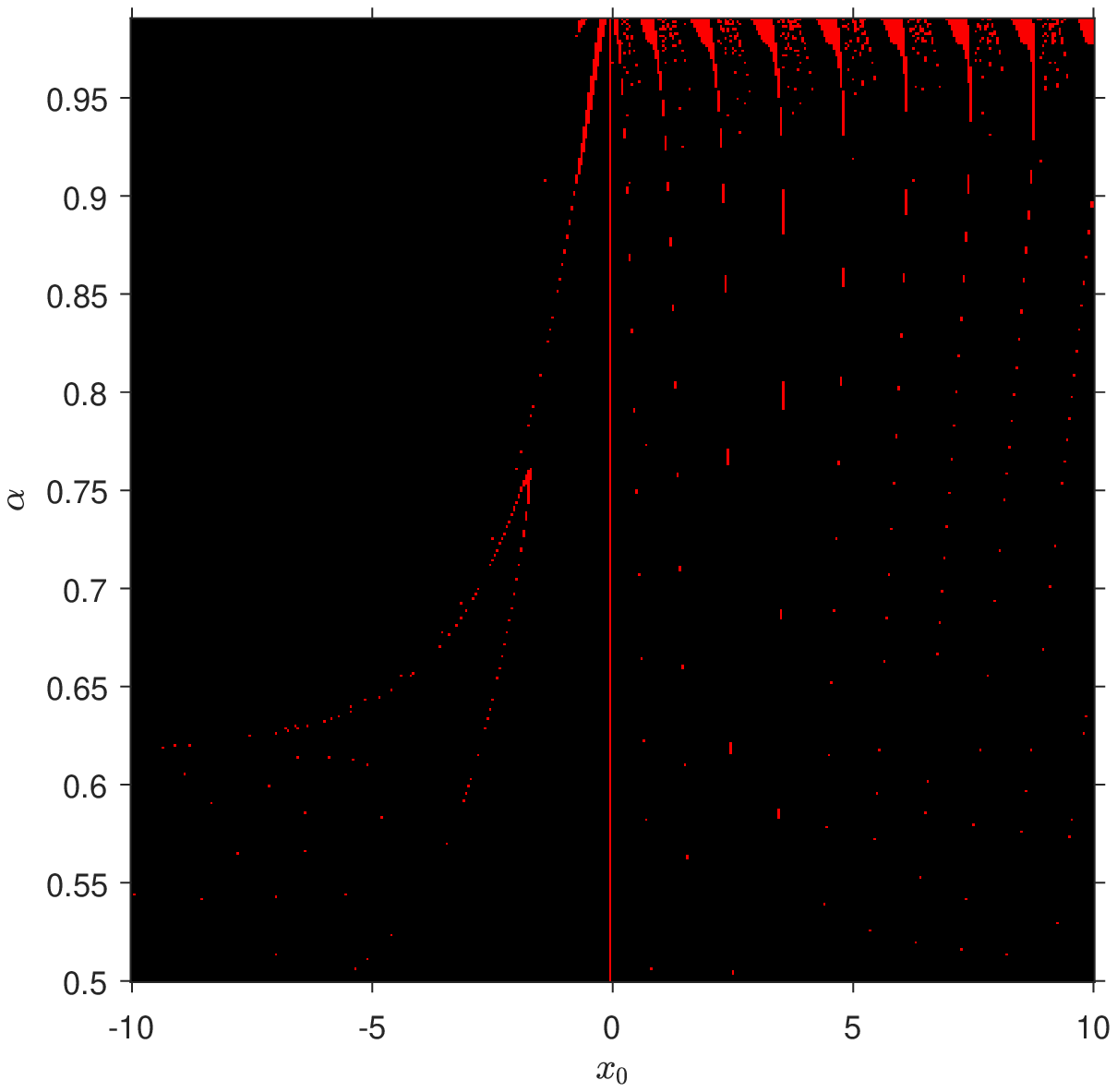}
(b) R-LFT, $-10\leq x_0\leq10$, 1.47\% convergence
\end{center}
\end{minipage}
\captionof{figure}{Convergence planes of R-LFN$_2$ and R-LFT on $f_3(x)$ with $x_0$ real}\label{f18}
\vspace{20pt}
In figures \ref{f19}-\ref{f24} we can see the behavior for $f_4(x)$. \\
\begin{minipage}[c]{0.5\textwidth}
\begin{center}
\includegraphics[width=\textwidth]{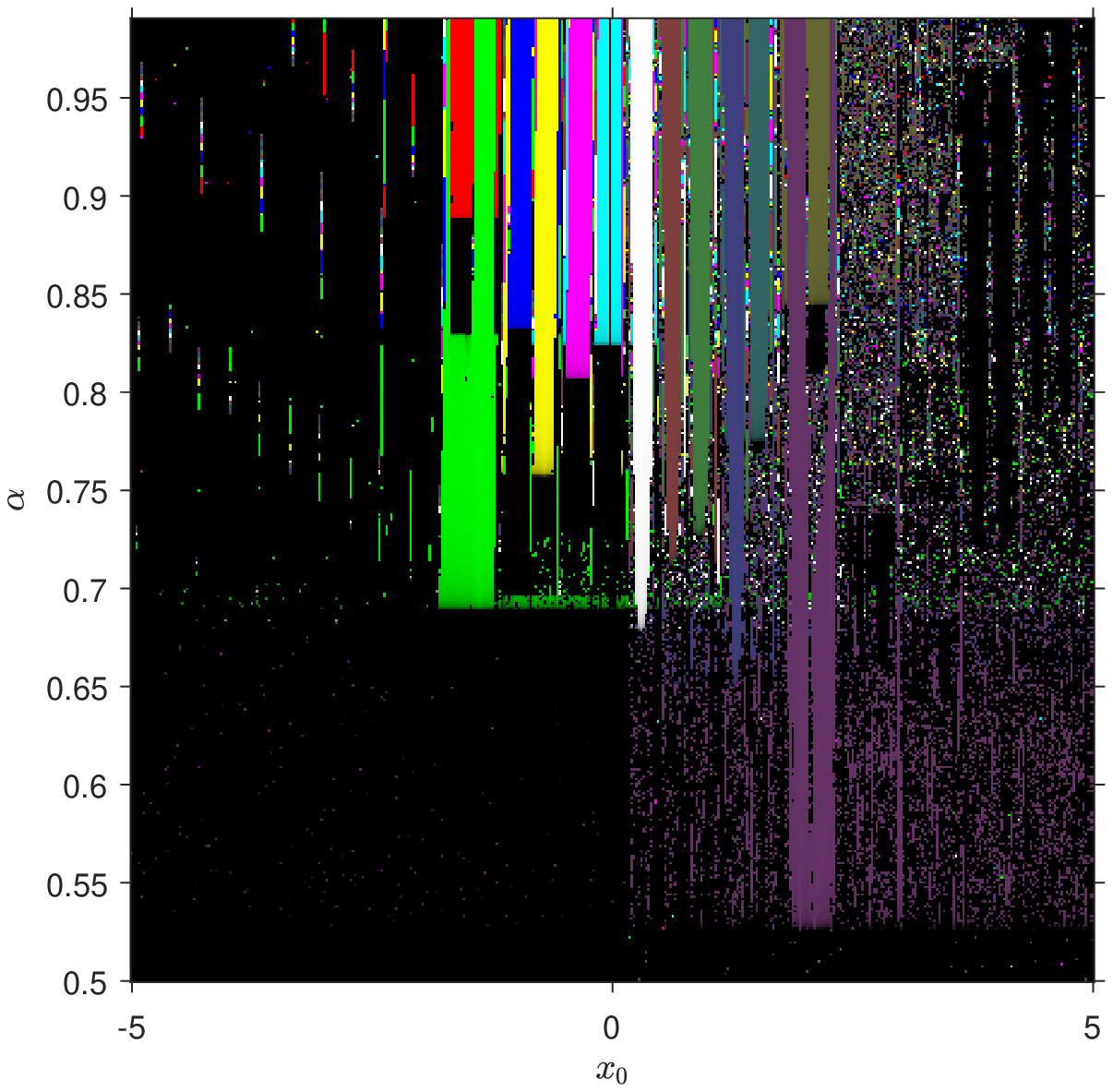}
(a) CFN$_1$, $-5\leq x_0\leq5$, 30.39\% convergence
\end{center}
\end{minipage}
\begin{minipage}[c]{0.5\textwidth}
\begin{center}
\includegraphics[width=\textwidth]{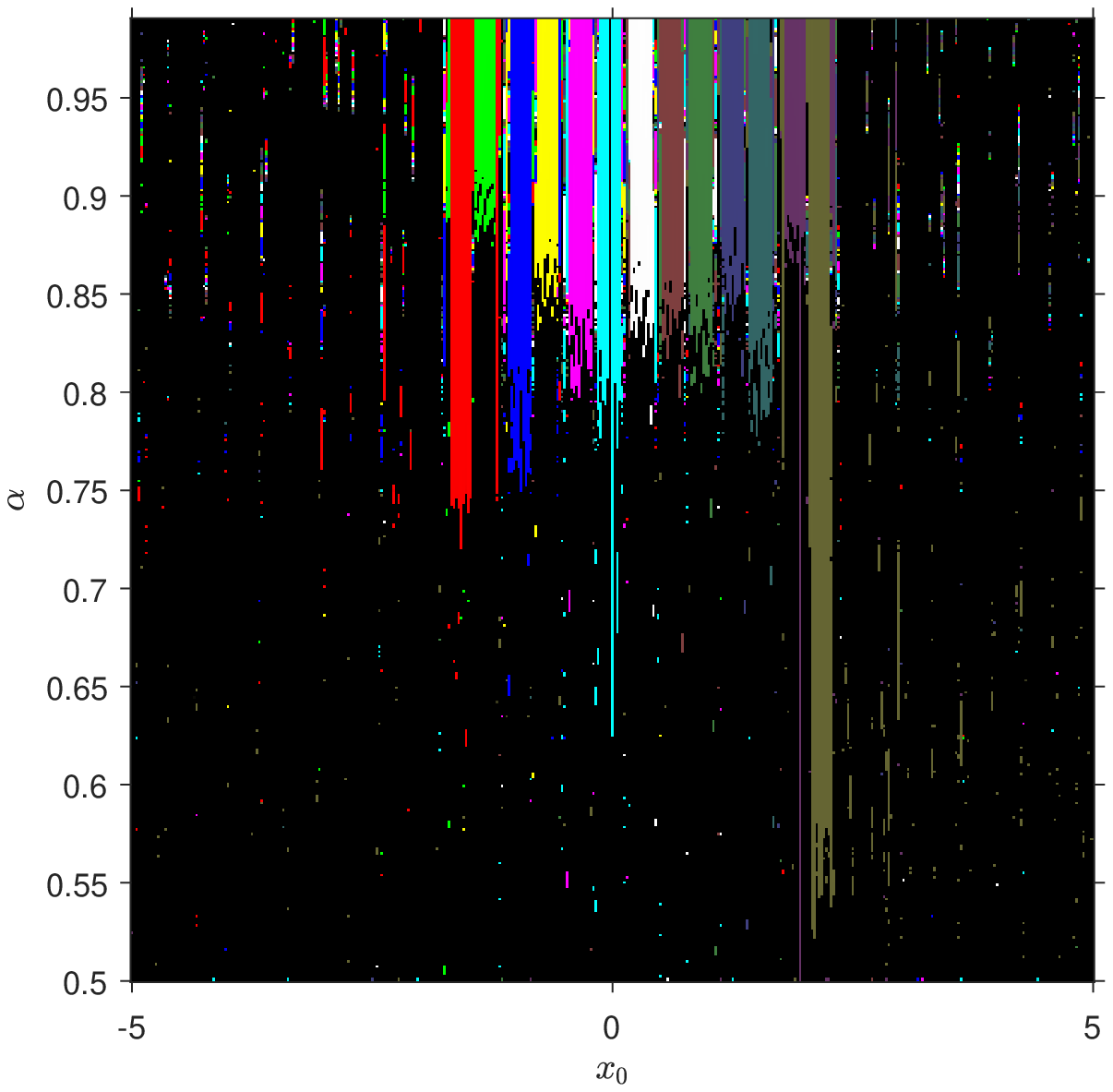}
(b) CFN$_2$, $-5\leq x_0\leq5$, 16.46\% convergence
\end{center}
\end{minipage}
\captionof{figure}{Convergence planes of CFN$_1$ and CFN$_2$ on $f_4(x)$ with $x_0$ real}\label{f19}
\begin{minipage}[c]{0.5\textwidth}
\begin{center}
\includegraphics[width=\textwidth]{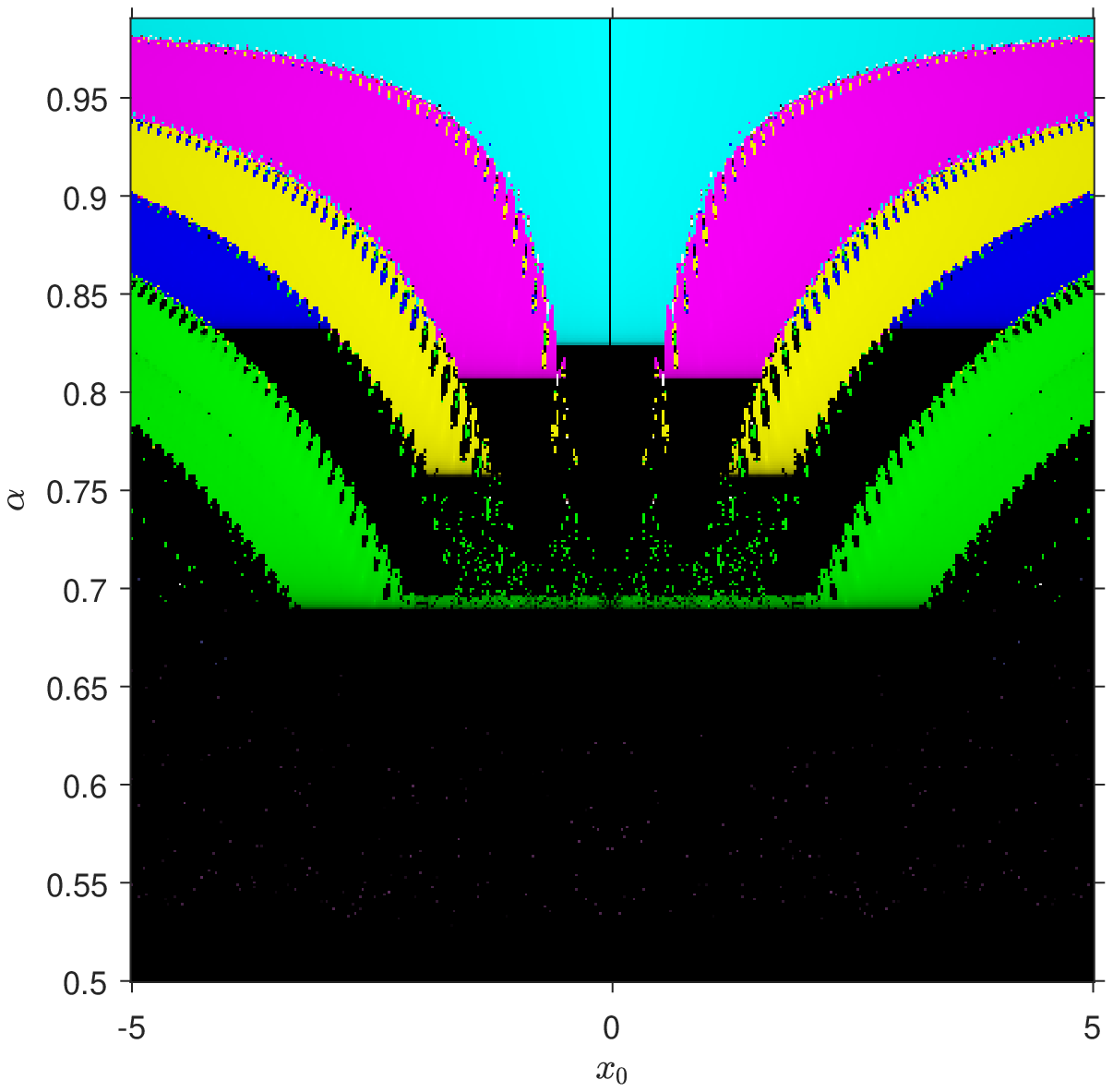}
(a) CFN$_1$, $-5i\leq x_0\leq5i$, 45.2\% convergence
\end{center}
\end{minipage}
\begin{minipage}[c]{0.5\textwidth}
\begin{center}
\includegraphics[width=\textwidth]{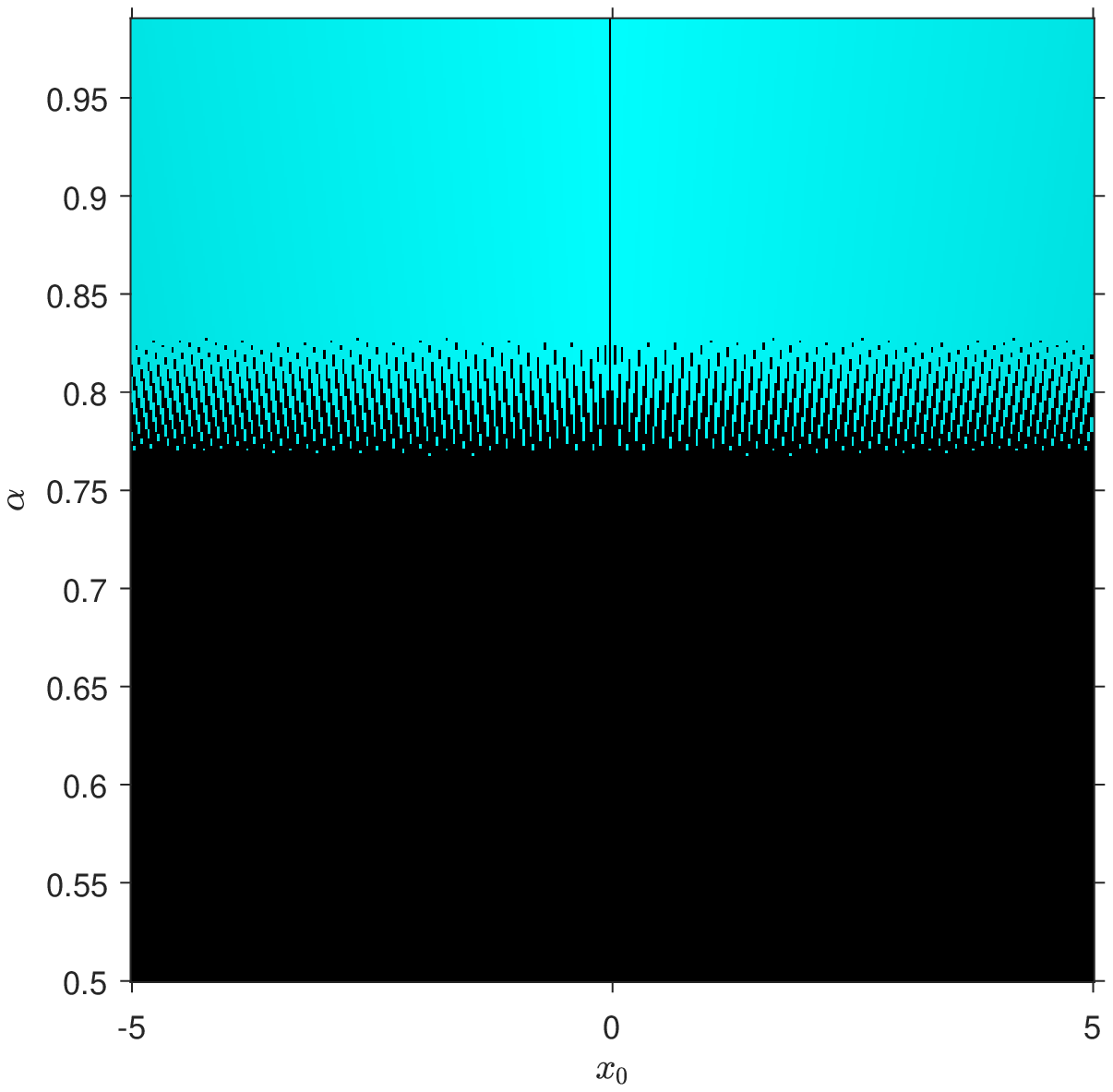}
(b) CFN$_2$, $-5i\leq x_0\leq5i$, 39.85\% convergence
\end{center}
\end{minipage}
\captionof{figure}{Convergence planes of CFN$_1$ and CFN$_2$ on $f_4(x)$ with $x_0$ imaginary}\label{f20}
\begin{minipage}[c]{0.5\textwidth}
\begin{center}
\includegraphics[width=\textwidth]{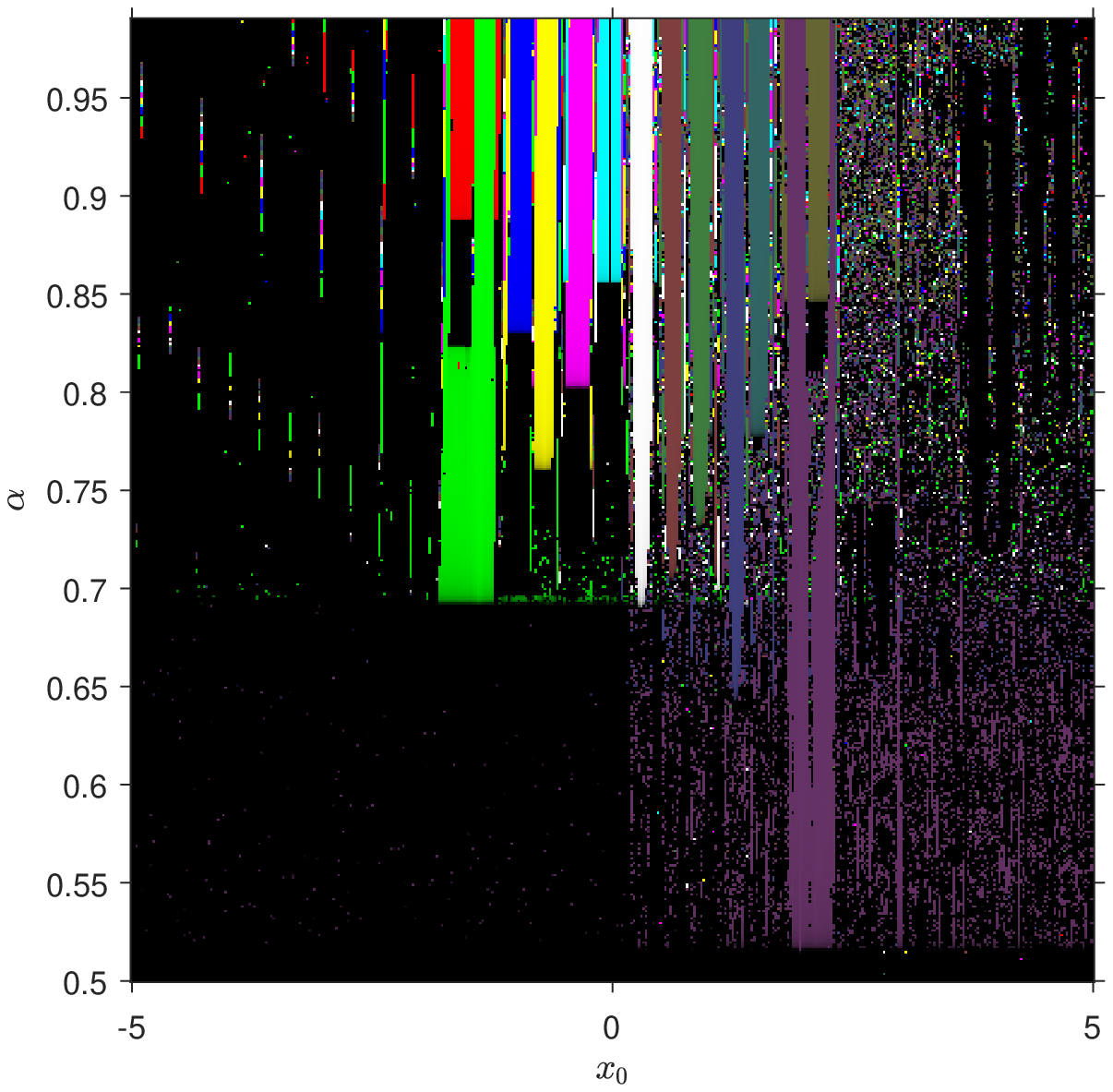}
(a) R-LFN$_1$, $-5\leq x_0\leq5$, 30.06\% convergence
\end{center}
\end{minipage}
\begin{minipage}[c]{0.5\textwidth}
\begin{center}
\includegraphics[width=\textwidth]{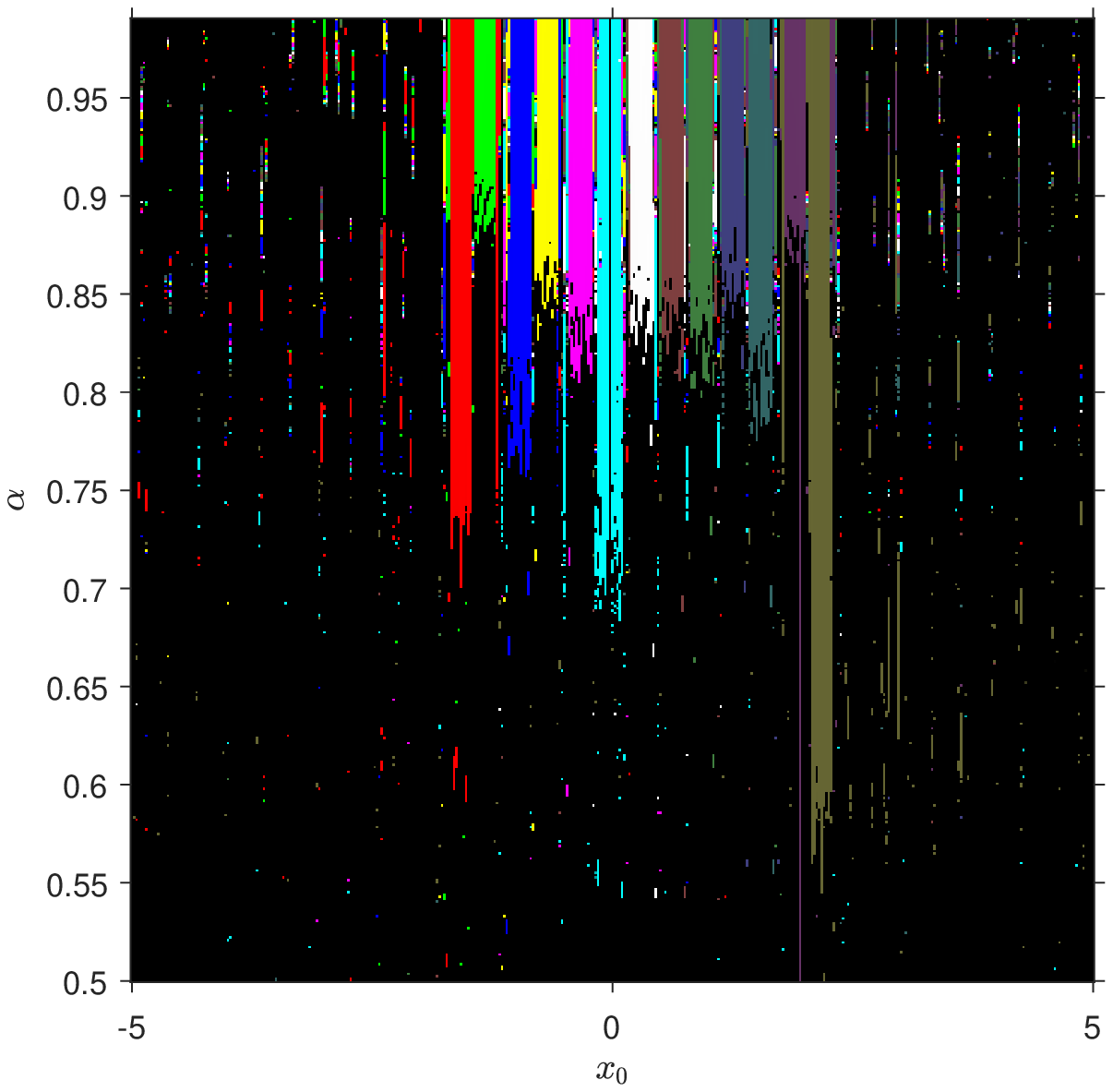}
(b) R-LFN$_2$, $-5\leq x_0\leq5$, 16.88\% convergence
\end{center}
\end{minipage}
\captionof{figure}{Convergence planes of R-LFN$_1$ and R-LFN$_2$ on $f_4(x)$ with $x_0$ real}\label{f21}
\begin{minipage}[c]{0.5\textwidth}
\begin{center}
\includegraphics[width=\textwidth]{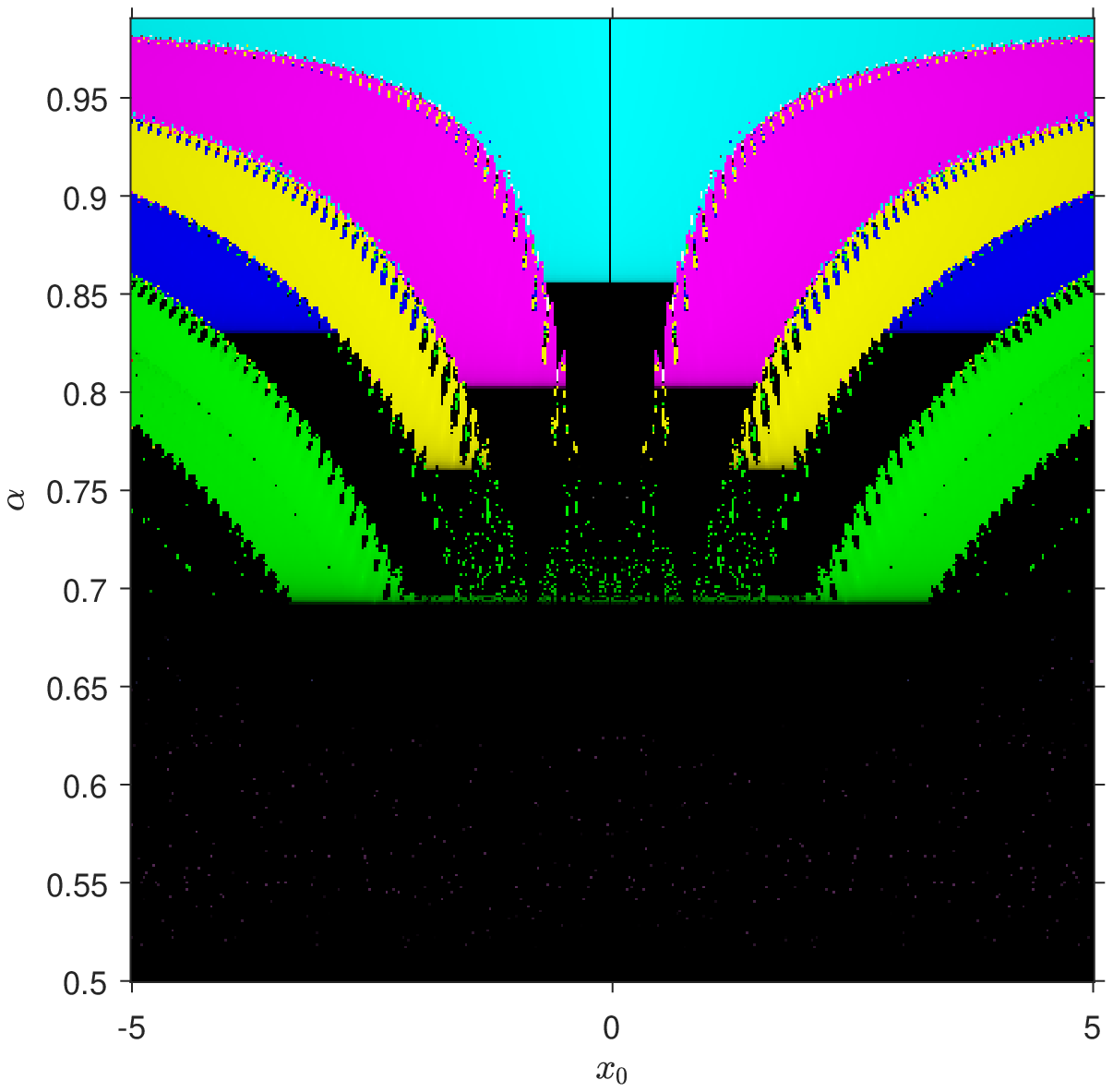}
(a) R-LFN$_1$, $-5i\leq x_0\leq5i$, 44.28\% convergence
\end{center}
\end{minipage}
\begin{minipage}[c]{0.5\textwidth}
\begin{center}
\includegraphics[width=\textwidth]{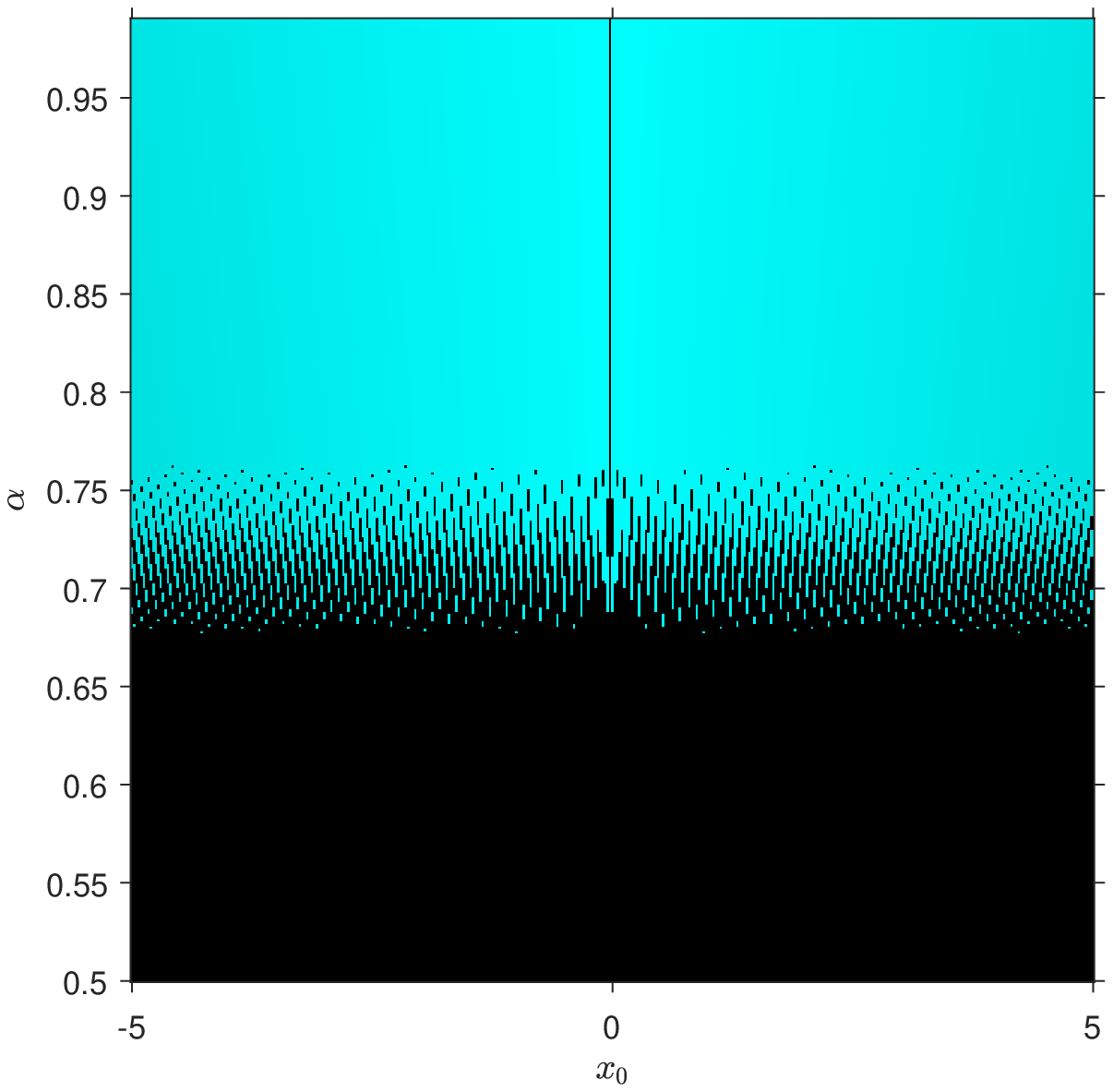}
(b) R-LFN$_2$, $-5i\leq x_0\leq5i$, 55.66\% convergence
\end{center}
\end{minipage}
\captionof{figure}{Convergence planes of R-LFN$_1$ and R-LFN$_2$ on $f_4(x)$ with $x_0$ imaginary}\label{f22}
\begin{minipage}[c]{0.5\textwidth}
\begin{center}
\includegraphics[width=\textwidth]{recursos/c_c_4}
(a) CFN$_2$, $-5\leq x_0\leq5$, 16.46\% convergence
\end{center}
\end{minipage}
\begin{minipage}[c]{0.5\textwidth}
\begin{center}
\includegraphics[width=\textwidth]{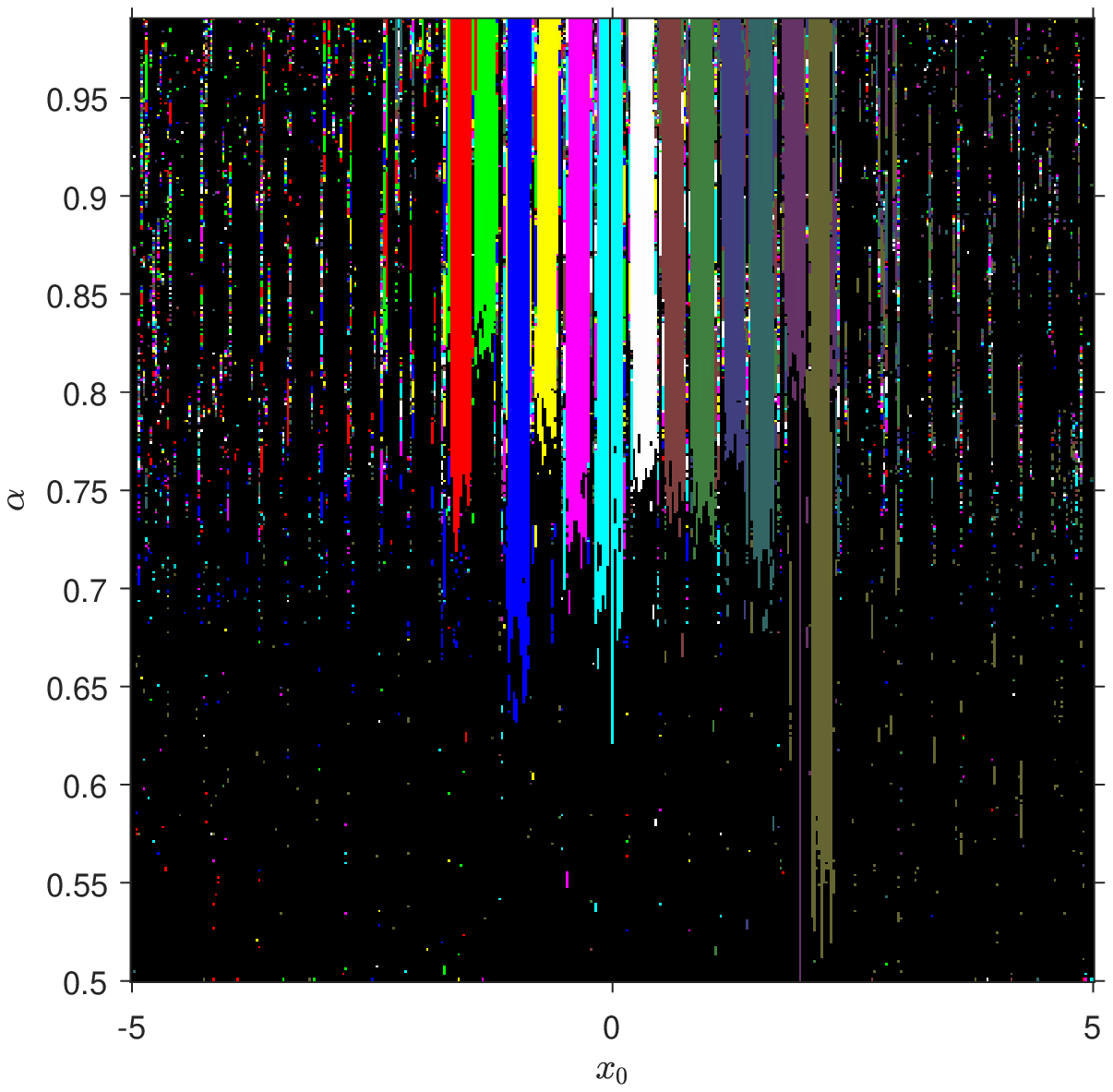}
(b) CFT, $-5\leq x_0\leq5$, 24.29\% convergence
\end{center}
\end{minipage}
\captionof{figure}{Convergence planes of CFN$_2$ and CFT on $f_4(x)$ with $x_0$ real}\label{f23}
\begin{minipage}[c]{0.5\textwidth}
\begin{center}
\includegraphics[width=\textwidth]{recursos/rl_c_4}
(a) R-LFN$_2$, $-5\leq x_0\leq5$, 16.88\% convergence
\end{center}
\end{minipage}
\begin{minipage}[c]{0.5\textwidth}
\begin{center}
\includegraphics[width=\textwidth]{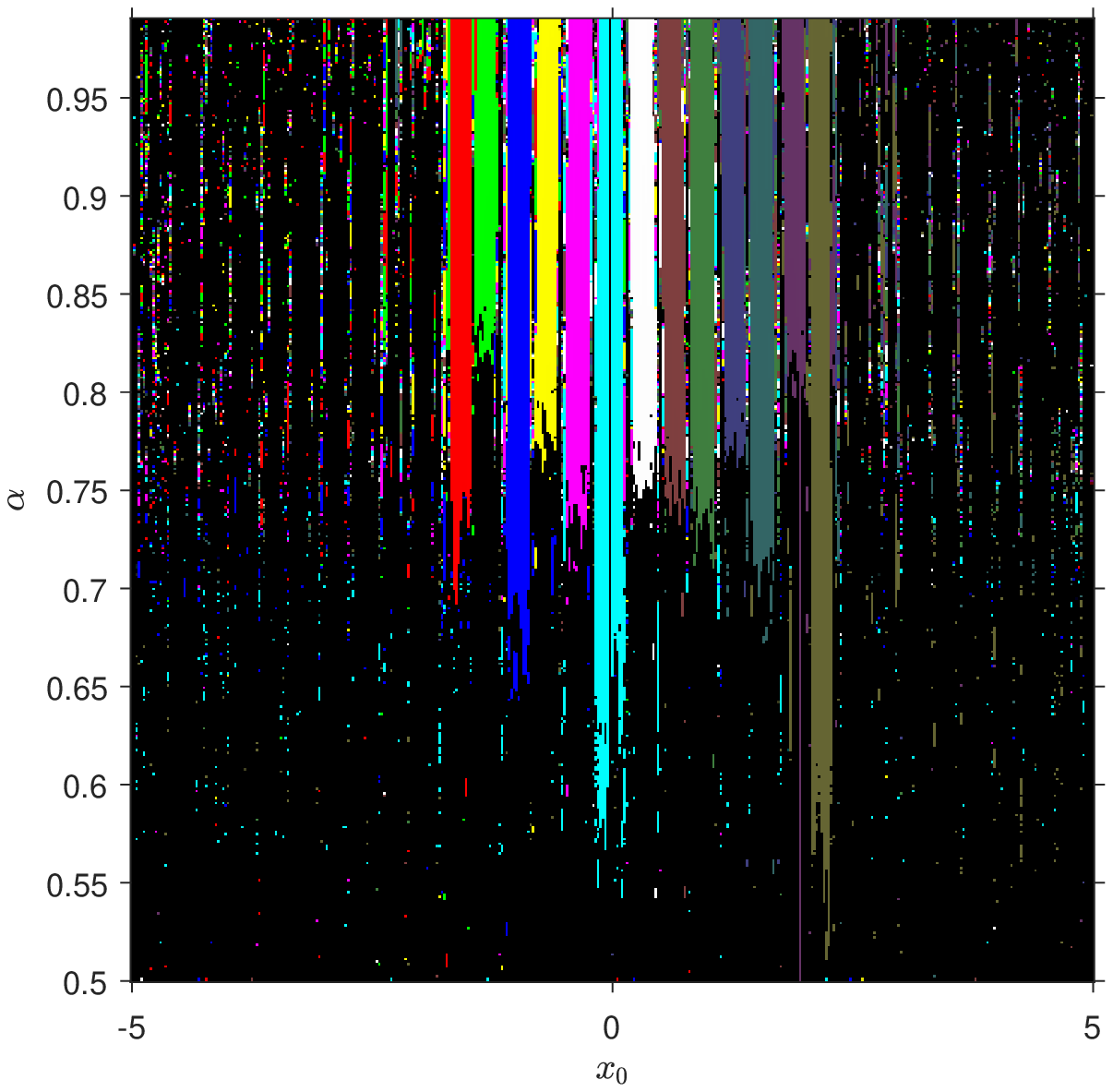}
(b) R-LFT, $-5\leq x_0\leq5$, 25.04\% convergence
\end{center}
\end{minipage}
\captionof{figure}{Convergence planes of R-LFN$_2$ and R-LFT on $f_4(x)$ with $x_0$ real}\label{f24}
\vspace{10pt}
We can see for $f_4(x)$, in general, that CFN$_1$ and R-LFN$_1$ methods have a higher percentage of convergence than CFN$_2$ and R-LFN$_2$ methods respectively. We can also see that Traub methods improve Newton methods.

\section{Concluding Remaks}

Two new fractional Newton methods and two fractional Traub methods have been designed with Caputo and Riemann-Liouville derivatives. These methods do not need a damping parameter to prove the order of convergence. Some tests were made, and the dependence on the initial estimation was analized. In general, the fractional Newton methods proposed in \cite{AJR} has better properties than the fractional Newton methods designed in this paper in terms of wideness of the basins of atractions of the roots, even though the new Newton methods could show better results with large absolute values of the imaginary part of the initial estimations. The Traub methods improve the new Newton methods, not only because require less iterations, but also because have higher percentage of convergence.

\end{document}